\documentclass[a4paper,12pt,reqno]{amsart}

\usepackage[T1]{fontenc}
\usepackage[utf8]{inputenc}
\usepackage{lmodern}
\usepackage[english]{babel}

\usepackage[%
	left=2.5cm,       
	right=2.5cm,      
	top=3.5cm,        
	bottom=3.5cm,     
	heightrounded,    
	bindingoffset=0mm 
]{geometry}

\usepackage{amsmath}
\usepackage{amssymb}
\usepackage{mathtools}
\usepackage{mathrsfs}

\numberwithin{equation}{section}

\usepackage{hyperref} 
\usepackage{bookmark}
\usepackage[noabbrev,capitalize]{cleveref}

\theoremstyle{plain}
	\newtheorem{theorem}{Theorem}[section]
	\newtheorem{lemma}[theorem]{Lemma}
	\newtheorem{proposition}[theorem]{Proposition}
	\newtheorem{corollary}[theorem]{Corollary}

\theoremstyle{definition}

	\newtheorem{remark}[theorem]{Remark}
	\newtheorem{open.problem}[theorem]{Open Problem}

\newcommand{\N}{\mathbb{N}}
\newcommand{\R}{\mathbb{R}}

\newcommand{\eps}{\varepsilon}

\usepackage[initials]{amsrefs}

\usepackage{tikz}
\usepackage{graphicx}
\usepackage{xcolor}
\usepackage[font=small]{caption}

\newcommand{\closure}[2][3]{%
  {}\mkern#1mu\overline{\mkern-#1mu#2}}
  
\usepackage{enumerate}

\usepackage{url}

\newcommand{\de}{\partial}

\newcommand{\weakto}{\rightharpoonup}

\newcommand{\Id}{\mathcal{I}}

\renewcommand{\phi}{\varphi}
\renewcommand{\rho}{\varrho}
\renewcommand{\theta}{\vartheta}

\DeclareMathOperator{\dist}{dist}
\DeclareMathOperator{\diam}{diam}

\DeclareMathOperator{\supp}{supp}
\DeclareMathOperator{\sgn}{sgn}

\DeclareMathOperator{\diverg}{div}

\DeclareMathOperator{\Lip}{Lip}
\DeclareMathOperator{\loc}{loc}
\DeclareMathOperator*{\esssup}{ess\,sup}
\DeclareMathOperator*{\essinf}{ess\,inf}

\DeclarePairedDelimiter{\set}{\{}{\}}
\DeclarePairedDelimiter{\abs}{|}{|}

\mathchardef\ordinarycolon\mathcode`\:
\mathcode`\:=\string"8000
\begingroup \catcode`\:=\active
  \gdef:{\mathrel{\mathop\ordinarycolon}}
\endgroup

\newcommand{\mhypen}{\mbox{\,-}}

\newcommand{\Haus}[1]{\mathscr{H}^{#1}} 
\newcommand{\Leb}[1]{\mathscr{L}^{#1}} 
\renewcommand{\div}{\mathrm{div}} 
\newcommand{\redb}{\mathscr{F}} 
\newcommand{\M}{\mathscr{M}}

\allowdisplaybreaks

\begin{document}

\title[A distributional approach to fractional variation: asymptotics I]{A distributional approach to fractional Sobolev spaces and fractional variation: asymptotics I}

\author[G.~E.~Comi]{Giovanni E. Comi}
\address[G.~E.~Comi]{Dipartimento di Matematica, Universit\`a di Pisa, Largo Bruno Pontecorvo 5, 56127 Pisa, Italy}
\email{giovanni.comi@dm.unipi.it}

\author[G.~Stefani]{Giorgio Stefani}
\address[G.~Stefani]{Department Mathematik und Informatik, Universit\"at Basel, Spiegelgasse 1, CH-4051 Basel, Switzerland}
\email{giorgio.stefani@unibas.ch}

\date{\today}

\keywords{Function with bounded fractional variation, fractional perimeter, fractional derivative, fractional gradient, fractional divergence, Gamma-convergence.}

\subjclass[2010]{26A33, 26B30, 28A33}

\thanks{\textit{Acknowledgements}. %
The authors thank Luigi Ambrosio, Elia Brué, Mattia Calzi, Quoc-Hung Nguyen and Daniel Spector for many valuable suggestions and useful comments. The authors also wish to express their gratitude to the anonymous referees for their insightful remarks. This research was partially supported by the PRIN2015 MIUR Project ``Calcolo delle Variazioni''.
The second author is a member of INdAM--GNAMPA and is partially supported by the ERC Starting Grant 676675 FLIRT -- \textit{Fluid Flows and Irregular Transport} and by the INdAM--GNAMPA Project 2020 \textit{Problemi isoperimetrici con anisotropie} (n.\ prot.\ U-UFMBAZ-2020-000798 15-04-2020)%
}

\begin{abstract}
We continue the study of the space~$BV^\alpha(\R^n)$ of functions with bounded fractional variation in~$\R^n$ of order $\alpha\in(0,1)$ introduced in our previous work~\cite{CS18}.
After some technical improvements of certain results of~\cite{CS18} which may be of some separated insterest, we  deal with the asymptotic behavior of the fractional operators involved as $\alpha\to1^-$.
We prove that the $\alpha$-gradient of a $W^{1,p}$-function converges in $L^p$ to the gradient for all $p\in[1,+\infty)$ as $\alpha\to1^-$.
Moreover, we prove that the fractional $\alpha$-variation converges to the standard De Giorgi's variation both pointwise and in the $\Gamma$-limit sense as $\alpha\to1^-$. 
Finally, we prove that the fractional $\beta$-variation converges to the fractional $\alpha$-variation both pointwise and in the $\Gamma$-limit sense as $\beta\to\alpha^-$ for any given $\alpha\in(0,1)$.
\end{abstract}

\maketitle

\tableofcontents

\section{Introduction}

\subsection{A distributional approach to fractional variation}

In our previous work~\cite{CS18}, we introduced the space~$BV^\alpha(\R^n)$ of functions with bounded fractional variation in~$\R^n$ of order $\alpha\in(0,1)$. Precisely, a function $f\in L^1(\R^n)$ belongs to the space~$BV^\alpha(\R^n)$ if its \emph{fractional $\alpha$-variation} 
\begin{equation}\label{intro_eq:def_frac_var}
|D^\alpha f|(\R^n):=
\sup\set*{\int_{\R^n} f\,\div^\alpha\phi\,dx : \phi\in C^\infty_c(\R^n;\R^n),\ \|\phi\|_{L^\infty(\R^n;\,\R^n)}\le1}
\end{equation}
is finite. Here
\begin{equation}\label{intro_eq:def_div_alpha}
\div^\alpha\phi(x):=\mu_{n,\alpha}\int_{\R^n}\frac{(y-x)\cdot(\phi(y)-\phi(x))}{|y-x|^{n+\alpha+1}}\,dy,
\qquad
x\in\R^n,	
\end{equation} 
is the \emph{fractional $\alpha$-divergence} of $\phi\in C^\infty_c(\R^n;\R^n)$, where 
\begin{equation}\label{intro_eq:def_mu_n_alpha}
\mu_{n, \alpha} := 2^{\alpha} \pi^{- \frac{n}{2}} \frac{\Gamma\left ( \frac{n + \alpha + 1}{2} \right )}{\Gamma\left ( \frac{1 - \alpha}{2} \right )}
\end{equation}
for any given $\alpha\in(0,1)$. The operator~$\div^\alpha$ was introduced in~\cite{S19} as the natural \emph{dual} operator of the much more studied \emph{fractional $\alpha$-gradient} 
\begin{equation}\label{intro_eq:def_nabla_alpha}
\nabla^\alpha f(x):=\mu_{n,\alpha}\int_{\R^n}\frac{(y-x)(f(y)-f(x))}{|y-x|^{n+\alpha+1}}\,dy,
\qquad
x\in\R^n,	
\end{equation} 
defined for all $f\in C^\infty_c(\R^n)$. For an account on the existing literature on the operator~$\nabla^\alpha$, see~\cite{SSS18}*{Section~1}. Here we only refer to~\cites{SSS15,SSVanS17,SSS18,SS15,SS18,S19,Spector18,Spector19} for the articles tightly connected to the present work and to~\cite{P16}*{Section 15.2} for an agile presentation of the fractional operators defined in~\eqref{intro_eq:def_div_alpha} and in~\eqref{intro_eq:def_nabla_alpha} and of some of their elementary properties. According to~\cite{SS18}*{Section~1}, it is interesting to notice that~\cite{H59} seems to be the earliest reference for the operator defined in~\eqref{intro_eq:def_nabla_alpha}. 

The operators in~\eqref{intro_eq:def_div_alpha} and in~\eqref{intro_eq:def_nabla_alpha} are \emph{dual} in the sense that
\begin{equation}\label{intro_eq:duality}
\int_{\R^n}f\,\div^\alpha\phi \,dx=-\int_{\R^n}\phi\cdot\nabla^\alpha f\,dx
\end{equation}
for all $f\in C^\infty_c(\R^n)$ and $\phi\in C^\infty_c(\R^n;\R^n)$, see~\cite{S19}*{Section~6} and~\cite{CS18}*{Lemma~2.5}. Moreover, both operators have good integrability properties when applied to test functions, namely $\nabla^\alpha f\in L^p(\R^n)$ and $\div^\alpha\phi\in L^p(\R^n;\R^n)$ for all $p\in[1,+\infty]$ for any given $f\in C^\infty_c(\R^n)$ and $\phi\in C^\infty_c(\R^n;\R^n)$, see~\cite{CS18}*{Corollary~2.3}.

The integration-by-part formula~\eqref{intro_eq:duality} represents the starting point for the distributional approach to fractional Sobolev spaces and the fractional variation we developed in~\cite{CS18}. In fact, similarly to the classical case, a function $f\in L^1(\R^n)$ belongs to $BV^\alpha(\R^n)$ if and only if there exists a finite vector-valued Radon measure $D^{\alpha} f \in \mathscr M(\R^n; \R^{n})$ such that 
\begin{equation}\label{intro_eq:BV_alpha_duality} 
\int_{\R^{n}} f\, \div^{\alpha} \phi \, dx = - \int_{\R^n} \phi \cdot d D^{\alpha} f 
\end{equation}
for all $\phi \in C^{\infty}_{c}(\R^n; \R^{n})$, see~\cite{CS18}*{Theorem~3.2}. 

Motivated by~\eqref{intro_eq:BV_alpha_duality} and similarly to the classical case, we can define the \emph{weak fractional $\alpha$-gradient} of a function $f\in L^p(\R^n)$, with $p\in[1,+\infty]$, as the function $\nabla^\alpha_w f\in L^1_{\loc}(\R^n;\R^n)$ satisfying
\begin{equation*}
\int_{\R^n}f\,\div^\alpha\phi\, dx
=-\int_{\R^n}\nabla^\alpha_w f\cdot\phi\, dx
\end{equation*} 
for all $\phi\in C^\infty_c(\R^n;\R^n)$. For $\alpha\in(0,1)$ and $p\in[1,+\infty]$, we can thus define the \emph{distributional fractional Sobolev space} 
\begin{equation}\label{eq:def_distrib_frac_Sobolev}
S^{\alpha,p}(\R^n):=\set*{f\in L^p(\R^n) : \exists\, \nabla^\alpha_w f \in L^p(\R^n;\R^n)}
\end{equation}
naturally endowed with the norm
\begin{equation}\label{eq:def_distrib_frac_Sobolev_norm}
\|f\|_{S^{\alpha,p}(\R^n)}:=\|f\|_{L^p(\R^n)}+\|\nabla^\alpha_w f\|_{L^p(\R^n;\,\R^{n})}
\qquad
\forall f\in S^{\alpha,p}(\R^n).
\end{equation}

It is interesting to compare the distributional fractional Sobolev spaces $S^{\alpha,p}(\R^n)$ with the well-known \emph{fractional Sobolev space} $W^{\alpha,p}(\R^n)$, that is, the space
\begin{equation*}
W^{\alpha,p}(\R^n):=\set*{f\in L^p(\R^n) : [f]_{W^{\alpha,p}(\R^n)}:=\left(\int_{\R^{n}} \int_{\R^{n}} \frac{|f(x)-f(y)|^p}{|x-y|^{n+p\alpha}}\,dx\,dy\right)^{\frac{1}{p}}<+\infty}
\end{equation*}
endowed with the norm
\begin{equation*}
\|f\|_{W^{\alpha,p}(\R^n)}:=\|f\|_{L^p(\R^n)}+[f]_{W^{\alpha,p}(\R^n)}
\qquad
\forall f\in W^{\alpha,p}(\R^n).
\end{equation*}
If $p=+\infty$, then $W^{\alpha,\infty}(\R^n)$ naturally coincides with the space of bounded $\alpha$-H\"older continuous functions endowed with the usual norm (see~\cite{DiNPV12} for a detailed account on the spaces $W^{\alpha,p}$).  

For the case $p=1$, starting from the very definition of the fractional gradient~$\nabla^\alpha$, it is plain to see that $W^{\alpha,1}(\R^n)\subset S^{\alpha,1}(\R^n)\subset BV^\alpha(\R^n)$ with both (strict) continuous embeddings, see~\cite{CS18}*{Theorems 3.18 and~3.25}.

For the case $p\in(1,+\infty)$, instead, it is known that $S^{\alpha,p}(\R^n)\supset L^{\alpha,p}(\R^n)$ with continuous embedding, where $L^{\alpha,p}(\R^n)$ is the \emph{Bessel potential space} of parameters $\alpha\in(0,1)$ and $p\in(1,+\infty)$, see~\cite{CS18}*{Section~3.9} and the references therein. In the subsequent paper~\cite{BCCS}, it will be proved that also the inclusion $S^{\alpha,p}(\R^n)\subset L^{\alpha,p}(\R^n)$ holds continuously, so that the spaces $S^{\alpha,p}(\R^n)$ and $L^{\alpha,p}(\R^n)$ coincide. In particular, we get the following relations: $S^{\alpha+\eps,p}(\R^n)\subset W^{\alpha, p}(\R^{n})\subset S^{\alpha-\eps,p}(\R^n)$ with continuous embeddings for all $\alpha\in(0,1)$, $p\in(1,+\infty)$ and $0<\eps<\min\set*{\alpha,1-\alpha}$, see~\cite{SS15}*{Theorem~2.2}; $S^{\alpha, 2}(\R^{n})=W^{\alpha, 2}(\R^{n})$ for all $\alpha \in (0, 1)$, see~\cite{SS15}*{Theorem 2.2}; $W^{\alpha,p}(\R^n)\subset S^{\alpha,p}(\R^n)$ with continuous embedding for all $\alpha\in(0,1)$ and $p\in(1,2]$, see~\cite{S70}*{Chapter~V, Section~5.3}.

In the \emph{geometric regime} $p=1$, our distributional approach to the fractional variation naturally provides a new definition of distributional fractional perimeter. Precisely, for any open set $\Omega\subset\R^n$, the \emph{fractional Caccioppoli $\alpha$-perimeter in $\Omega$} of a measurable set $E\subset\R^n$ is the \emph{fractional $\alpha$-variation} of $\chi_E$ in $\Omega$, i.e.\ 
\begin{equation*}
|D^\alpha\chi_E|(\Omega)=\sup\set*{\int_E\div^\alpha\phi\,dx : \phi\in C^\infty_c(\Omega;\R^n),\ \|\phi\|_{L^\infty(\Omega;\R^n)}\le1}.
\end{equation*}
Thus, $E$ is a set with \emph{finite fractional Caccioppoli $\alpha$-perimeter in $\Omega$} if $|D^\alpha\chi_E|(\Omega)<+\infty$. 

Similarly to the aforementioned embedding $W^{\alpha,1}(\R^n)\subset BV^\alpha(\R^n)$, we have the inequality
\begin{equation} \label{intro_eq:bound_D_alpha_P}
|D^\alpha\chi_E|(\Omega)\le\mu_{n,\alpha}P_\alpha(E;\Omega)
\end{equation}
for any open set $\Omega\subset\R^n$, see~\cite{CS18}*{Proposition~4.8}, where
\begin{equation}\label{intro_eq:def_frac_relative_perim}
P_{\alpha}(E; \Omega) 
:= \int_{\Omega} \int_{\Omega} \frac{|\chi_{E}(x) - \chi_{E}(y)|}{|x - y|^{n + \alpha}} \, dx \, dy + 2 \int_{\Omega} \int_{\R^{n} \setminus \Omega} \frac{|\chi_{E}(x) - \chi_{E}(y)|}{|x - y|^{n + \alpha}} \, dx \, dy
\end{equation} 
is the standard \emph{fractional $\alpha$-perimeter} of a measurable set $E\subset\R^n$ relative to the open set $\Omega\subset\R^n$ (see~\cite{CF17} for an account on the fractional perimeter~$P_\alpha$). Note that, by definition, the \emph{fractional $\alpha$-perimeter} of~$E$ in~$\R^n$ is simply $P_\alpha(E):=P_\alpha(E;\R^n)=[\chi_E]_{W^{\alpha,1}(\R^n)}$. 
We remark that inequality~\eqref{intro_eq:bound_D_alpha_P} is strict in most of the cases, as shown in \cref{sec:comparison_seminorms} below. This completely answers a question left open in our previous work~\cite{CS18}.

\subsection{Asymptotics and \texorpdfstring{$\Gamma$}{Gamma}-convergence in the standard fractional setting}

The fractional Sobolev space $W^{\alpha,p}(\R^n)$ can be understood as an `intermediate space' between the space $L^p(\R^n)$ and the standard Sobolev space $W^{1,p}(\R^n)$. In fact, $W^{\alpha,p}(\R^n)$ can be recovered as a suitable \emph{(real) interpolation space} between the spaces $L^p(\R^n)$ and $W^{1,p}(\R^n)$. We refer to~\cites{BL76,T07} for a general introduction on interpolation spaces and to~\cite{M05} for a more specific treatment of the interpolation space between $L^p(\R^n)$ and $W^{1,p}(\R^n)$.

One then naturally expects that, for a sufficiently regular function~$f$, the fractional Sobolev seminorm $[f]_{W^{\alpha,p}(\R^n)}$, multiplied by a suitable renormalising constant, should tend to $\|f\|_{L^p(\R^n)}$ as $\alpha\to0^+$ and to $\|\nabla f\|_{L^p(\R^n)}$ as $\alpha\to1^-$. Indeed, for $p\in[1,+\infty)$, it is known that
\begin{equation}\label{intro_eq:banana}
\lim_{\alpha\to0^+}\alpha\,[f]_{W^{\alpha,p}(\R^n)}^p
=A_{n,p}\,\|f\|_{L^p(\R^n)}^p
\end{equation}   
for all $f\in\bigcup_{\alpha\in(0,1)}W^{\alpha,p}(\R^n)$, while 
\begin{equation}\label{intro_eq:limit_1}
\lim_{\alpha\to1^-}(1-\alpha)\,[f]_{W^{\alpha,p}(\R^n)}^p
=B_{n,p}\,\|\nabla f\|_{L^p(\R^n; \R^{n})}^p
\end{equation}   
for all $f\in W^{1,p}(\R^n)$. Here $A_{n,p},B_{n,p}>0$ are two constants depending only on~$n,p$. The limit~\eqref{intro_eq:banana} was proved in~\cites{MS02,MS03}, while the limit~\eqref{intro_eq:limit_1} was established in~\cite{BBM01}. As proved in~\cite{D02}, when $p=1$ the limit~\eqref{intro_eq:limit_1} holds in the more general case of $BV$ functions, that is,
\begin{equation}\label{intro_eq:limit_1_BV}
\lim_{\alpha\to1^-}(1-\alpha)\,[f]_{W^{\alpha,1}(\R^n)}
=B_{n,1}\,|Df|(\R^n)
\end{equation}   
for all $f\in BV(\R^n)$. For a different approach to the limits in~\eqref{intro_eq:banana} and in~\eqref{intro_eq:limit_1_BV} based on interpolation techniques, see~\cite{M05}.

The limits~\eqref{intro_eq:limit_1} and~\eqref{intro_eq:limit_1_BV} are special consequences of the celebrated \emph{Bourgain--Brezis--Mironescu} (BBM, for short) \emph{formula}
\begin{equation}
\label{intro_eq:BBM}
\lim_{k\to+\infty}
\int_{\R^n}\int_{\R^n}
\frac{|f(x)-f(y)|^p}{|x-y|^p}\,\rho_k(|x-y|)\,dx\,dy
=
\begin{cases}
C_{n,p}\,\|\nabla f\|_{L^p(\R^n)}^p 
& \text{for}\ p\in(1,+\infty),\\[3mm]
C_{n,1}\,|Df|(\R^n)
& \text{for}\ p=1,
\end{cases}
\end{equation}
where $C_{n,p}>0$ is a constant depending only on~$n$ and~$p$, and $(\rho_k)_{k\in\N}\subset L^1_{\loc}([0,+\infty))$ is a sequence of non-negative radial mollifiers such that
\begin{equation*}
\int_{\R^n}\rho_k(|x|)\,dx=1\
\text{for all $k\in\N$}
\quad
\text{and}\quad
\lim_{k\to+\infty}
\int_\delta^{+\infty}
\rho_k(r)\,r^{n-1}\,dr=0\
\text{for all $\delta>0$.}
\end{equation*}
The BBM formula~\eqref{intro_eq:BBM} has stimulated a profound development in the asymptotic analysis in the fractional framework.
On the one hand, the limit~\eqref{intro_eq:BBM} played a central role in several applications, such as Brezis' analysis~\cite{Brezis02} on how to recognize constant functions, innovative characterizations of Sobolev and $BV$ functions and $\Gamma$-convergence results~\cites{AGMP18,AGMP20,AGP20,BMR20,BN06,LS11,LS14,N07,N08,N11,P04-2}, approximation of Sobolev norms and image processing~\cites{B15,BN16-bis,BN18,BN20}, and last but not least fractional Hardy and Poincaré inequalities~\cites{BBM02-bis,FS08,P04-1}.
On the other hand, the BBM formula~\eqref{intro_eq:BBM} has suggested an alternative path to fractional asymptotic analysis by means of interpolation techniques~\cites{M05,PS17}.
Recently, the BBM formula in~\eqref{intro_eq:BBM} has been revisited in terms of a.e.\ pointwise convergence~\cite{BN16} and in connection with weak $L^p$ quasi-norms~\cite{BVY20}, where the now-called \emph{Brezis--Van Schaftingen--Yung space}
\begin{equation*}
BSY^{\alpha,p}(\R^n)
=
\set*{f\in L_{\loc}^1(\R^n) : \left\|\frac{|f(x)-f(y)|}{|x-y|^{\frac np+\alpha}}\right\|_{L^p_w(\R^n\times\R^n)}<+\infty},
\end{equation*}
defined for $\alpha\in(0,1]$ and $p\in[1,+\infty)$, has opened a very promising perspective in the field~\cite{DM20}. 

The limits \eqref{intro_eq:banana} -- \eqref{intro_eq:BBM} have been connected to variational problems~\cite{AK09}, generalized to various function spaces, for example Besov spaces~\cites{KL05,T11}, Orlicz spaces~\cites{ACPS20,FHR20,FS19} and magnetic and anisotropic Sobolev spaces~\cites{LMP19,NS19,PSV17,PSV19,SV16}, and extended to various ambient spaces, like compact connected Riemannian manifolds~\cite{KM19}, the flat torus~\cite{A20}, Carnot groups~\cites{B11,MP19} and complete doubling metric-measure spaces supporting a local Poincaré inequality~\cite{DiMS19}.

Concerning the fractional perimeter $P_\alpha$ given in~\eqref{intro_eq:def_frac_relative_perim}, one has some additional information besides equations~\eqref{intro_eq:banana} and~\eqref{intro_eq:limit_1_BV}. 

On the one hand, thanks to~\cite{PS17-box}*{Theorem~1.2}, the fractional $\alpha$-perimeter~$P_\alpha$ enjoys the following fractional analogue of Gustin's \emph{Boxing Inequality} (see~\cite{G60} and~\cite{F69}*{Corollary~4.5.4}): there exists a dimensional constant $c_n>0$ such that, for any bounded open set $E\subset\R^n$, one can find a covering 
\begin{equation*}
E\subset\bigcup_{k\in\N} B_{r_k}(x_k)
\end{equation*}
of open balls such that
\begin{equation}\label{intro_eq:boxing}
\sum_{k\in\N}r_k^{n-\alpha}\le c_n\alpha(1-\alpha)P_\alpha(E).
\end{equation}
Inequality~\eqref{intro_eq:boxing} bridges the two limiting behaviors given by~\eqref{intro_eq:banana} and~\eqref{intro_eq:limit_1_BV} and provides a useful tool for recovering Gagliardo--Nirenberg--Sobolev and Poincaré--Sobolev inequalities that remain stable as the exponent $\alpha\in(0,1)$ approaches the endpoints. 

On the other hand, by~\cite{ADPM11}*{Theorem~2}, the fractional $\alpha$-perimeter~$P_\alpha$ $\Gamma$-converges in $L^1_{\loc}(\R^n)$ to the standard De Giorgi's perimeter~$P$ as~$\alpha\to1^-$, that is, if $\Omega\subset\R^n$ is a bounded open set with Lipschitz boundary, then
\begin{equation}\label{intro_eq:gamma_limit_ADPM}
\Gamma(L^1_{\loc})\mhypen\lim_{\alpha\to1^-}(1-\alpha)\,P_\alpha(E;\Omega)
=2 \omega_{n-1} P(E;\Omega)
\end{equation}
for all measurable sets~$E\subset\R^n$, where~$\omega_{n}$ is the volume of the unit ball in~$\R^{n}$ (it should be noted that in~\cite{ADPM11} the authors use a slightly different definition of the fractional $\alpha$-perimeter, since they consider the functional $\mathcal{J}_{\alpha}(E, \Omega) := \frac{1}{2} P_{\alpha}(E, \Omega)$). For a complete account on \emph{$\Gamma$-convergence}, we refer the reader to the monographs~\cites{B02,D93} (throughout all the paper, with the symbol $\Gamma(X)\mhypen\lim$ we denote the $\Gamma$-convergence in the ambient metric space~$X$). The convergence in~\eqref{intro_eq:gamma_limit_ADPM}, besides giving a $\Gamma$-convergence analogue of the limit in~\eqref{intro_eq:limit_1_BV}, is tightly connected with the study of the regularity properties of \emph{non-local minimal surfaces}, that is, (local) minimisers of the fractional $\alpha$-perimeter~$P_\alpha$.   

\subsection{Asymptotics and \texorpdfstring{$\Gamma$}{Gamma}-convergence for the fractional \texorpdfstring{$\alpha$}{alpha}-variation as \texorpdfstring{$\alpha\to1^-$}{alpha tends to 1-}}

The main aim of the present work is to study the asymptotic behavior of the fractional $\alpha$-variation~\eqref{intro_eq:def_frac_var} as $\alpha\to1^-$, both in the pointwise and in the $\Gamma$-convergence sense.

We provide counterparts of the limits~\eqref{intro_eq:limit_1} and~\eqref{intro_eq:limit_1_BV} for the fractional $\alpha$-variation. Indeed, we prove that, if $f\in W^{1,p}(\R^n)$ for some $p\in[1,+\infty)$, then $f\in S^{\alpha,p}(\R^n)$ for all $\alpha\in(0,1)$ and, moreover, 
\begin{equation}\label{intro_eq:limit_1_p_frac_var}
\lim_{\alpha\to1^-}\|\nabla_w^\alpha f-\nabla_w f\|_{L^p(\R^n;\,\R^n)}=0.
\end{equation}
In the geometric regime $p=1$, we show that if $f\in BV(\R^n)$ then $f\in BV^\alpha(\R^n)$ for all $\alpha\in(0,1)$ and, in addition,
\begin{equation}\label{intro_eq:weak_limit_1_1_frac_var}
D^\alpha f\weakto Df\ 
\text{in $\mathscr{M}(\R^n;\R^n)$ and}\ 
|D^\alpha f|\weakto |Df|\
\text{in $\mathscr{M}(\R^n)$ as $\alpha\to1^-$}
\end{equation}
and 
\begin{equation}\label{intro_eq:limit_1_1_frac_var}
\lim_{\alpha\to1^-}|D^\alpha f|(\R^n)=|Df|(\R^n).
\end{equation}
We are also able to treat the case $p=+\infty$. In fact, we prove that if $f\in W^{1,\infty}(\R^n)$ then $f\in S^{\alpha,\infty}(\R^n)$ for all $\alpha\in(0,1)$ and, moreover,
\begin{equation}\label{intro_eq:weak_limit_1_infty_frac_var}
\nabla_w^\alpha f\weakto\nabla_w f
\quad
\text{in $L^\infty(\R^n;\R^n)$ as $\alpha\to1^-$}
\end{equation}
and
\begin{equation}\label{intro_eq:liminf_1_infty_frac_var}
\|\nabla_w f\|_{L^\infty(\R^n;\,\R^n)}\le\liminf_{\alpha\to1^-}\|\nabla^\alpha_w f\|_{L^\infty(\R^n;\,\R^n)}.
\end{equation}
We refer the reader to \cref{res:weak_conv_tot_var}, \cref{res:weak_conv_alpha_p} and \cref{res:weak_conv_alpha_infty} below for the precise statements. We warn the reader that the symbol `$\weakto$' appearing in~\eqref{intro_eq:weak_limit_1_1_frac_var} and~\eqref{intro_eq:weak_limit_1_infty_frac_var} denotes the \emph{weak*-convergence}, see \cref{subsec:notation} below for the notation.

Some of the above results were partially announced in~\cite{S16}. In a similar perspective, we also refer to the work~\cite{MS15}, where the authors proved convergence results for non-local gradient operators on $BV$ functions defined on bounded open sets with smooth boundary. The approach developed in~\cite{MS15} is however completely different from the asymptotic analysis we presently perform for the fractional operator defined in~\eqref{intro_eq:def_nabla_alpha}, since the boundedness of the domain of definition of the integral operators considered in~\cite{MS15} plays a crucial role.

Notice that the renormalising factor $(1-\alpha)^\frac{1}{p}$ is not needed in the limits~\eqref{intro_eq:limit_1_p_frac_var} -- \eqref{intro_eq:liminf_1_infty_frac_var}, contrarily to what happened for the limits~\eqref{intro_eq:limit_1} and~\eqref{intro_eq:limit_1_BV}. In fact, this difference should not come as a surprise, since the constant $\mu_{n,\alpha}$ in~\eqref{intro_eq:def_mu_n_alpha}, encoded in the definition of the operator~$\nabla^\alpha$, satisfies
\begin{equation}\label{intro_eq:mu_n_alpha_asym_1}
\mu_{n,\alpha}\sim\frac{1-\alpha}{\omega_n}
\quad
\text{as $\alpha\to1^-$},
\end{equation}
and thus plays a similar role of the factor $(1-\alpha)^\frac{1}{p}$ in the limit as $\alpha\to1^-$. Thus, differently from our previous work~\cite{CS18}, the constant~$\mu_{n,\alpha}$ appearing in the definition of the operators~$\nabla^\alpha$ and~$\div^\alpha$ is of crucial importance in the asymptotic analysis developed in the present paper. 

Another relevant aspect of our approach is that convergence as $\alpha\to1^-$ holds true not only for the total energies, but also at the level of differential operators, in the strong sense when $p\in(1,+\infty)$ and in the weak* sense for $p=1$ and $p=+\infty$. In simpler terms, the \emph{non-local} fractional $\alpha$-gradient~$\nabla^\alpha$ converges to the \emph{local} gradient~$\nabla$ as $\alpha\to1^-$ in the most natural way every time the limit is well defined.

We also provide a counterpart of~\eqref{intro_eq:gamma_limit_ADPM} for the fractional $\alpha$-variation as $\alpha\to1^-$. Precisely, we prove that, if $\Omega\subset\R^n$ is a bounded open set with Lipschitz boundary, then
\begin{equation}\label{intro_eq:gamma_lim}
\Gamma(L^1_{\loc})\mhypen\lim_{\alpha\to1^-}|D^\alpha\chi_E|(\Omega)
=P(E;\Omega)
\end{equation}
for all measurable set $E\subset\R^n$, see \cref{res:gamma_lim_alpha_to_1}. 
In view of~\eqref{intro_eq:bound_D_alpha_P}, one may ask whether the $\Gamma\mhypen\limsup$ inequality in~\eqref{intro_eq:gamma_lim} could be deduced from the $\Gamma\mhypen\limsup$ inequality in~\eqref{intro_eq:gamma_limit_ADPM}. In fact, by employing~\eqref{intro_eq:bound_D_alpha_P} together with~\eqref{intro_eq:gamma_limit_ADPM} and~\eqref{intro_eq:mu_n_alpha_asym_1}, one can estimate
\begin{equation*}
\Gamma(L^1_{\loc})\mhypen\limsup_{\alpha\to1^-}|D^\alpha\chi_E|(\Omega) \le \Gamma(L^1_{\loc})\mhypen\limsup_{\alpha\to1^-}\mu_{n, \alpha} P_{\alpha}(E, \Omega) = \frac{2 \omega_{n - 1}}{\omega_{n}} P(E, \Omega).
\end{equation*}
However, we have $\frac{2\omega_{n - 1}}{\omega_{n}} > 1$ for any $n \ge 2$ and thus the $\Gamma\mhypen\limsup$ inequality in~\eqref{intro_eq:gamma_lim} follows from the $\Gamma\mhypen\limsup$ inequality in~\eqref{intro_eq:gamma_limit_ADPM} only in the case $n=1$. In a similar way, one sees that the $\Gamma\mhypen\liminf$ inequality in~\eqref{intro_eq:gamma_lim} implies the $\Gamma\mhypen\liminf$ inequality in~\eqref{intro_eq:gamma_limit_ADPM} only in the case $n=1$.

Besides the counterpart of~\eqref{intro_eq:gamma_limit_ADPM}, our approach allows to prove that $\Gamma$-convergence holds true also at the level of functions. 
Indeed, if $f\in BV(\R^n)$ and $\Omega\subset\R^n$ is an open set such that either $\Omega$ is bounded with Lipschitz boundary or $\Omega=\R^n$, then
\begin{equation}\label{intro_eq:gamma_lim_L1}
\Gamma(L^1)\mhypen\lim_{\alpha\to1^-}|D^\alpha f|(\Omega)
=|D f|(\Omega).
\end{equation}
One can regard the limit~\eqref{intro_eq:gamma_lim_L1} as an analogue of the $\Gamma$-convergence results known in the usual fractional setting, see~\cites{P04-2,N11}. We refer the reader to \cref{res:Gamma_liminf}, \cref{res:Gamma_limsup} and \cref{res:gamma_lim_BV_alpha_to_1} for the (even more general) results in this direction. Again, as before and thanks to the asymptotic behavior~\eqref{intro_eq:mu_n_alpha_asym_1}, the renormalising factor $(1-\alpha)$ is not needed in the limits~\eqref{intro_eq:gamma_lim} and~\eqref{intro_eq:gamma_lim_L1}. 

As a byproduct of the techniques developed for the asymptotic study of the fractional $\alpha$-variation as $\alpha\to1^-$, we are also able to characterize the behavior of the fractional $\beta$-variation as $\beta\to\alpha^-$, for any given $\alpha\in(0,1)$. 
On the one hand, if $f\in BV^\alpha(\R^n)$, then
\begin{equation*}
D^\beta f\weakto D^\alpha f\ 
\text{in $\mathscr{M}(\R^n;\R^n)$ and}\ 
|D^\beta f|\weakto |D^\alpha f|\
\text{in $\mathscr{M}(\R^n)$  as $\beta\to\alpha^-$}
\end{equation*}
and, moreover,
\begin{equation*}
\lim_{\beta\to\alpha^-}|D^\beta f|(\R^n)=|D^\alpha f|(\R^n),
\end{equation*}
see \cref{res:weak_conv_nabla_beta_alpha}.
On the other hand, if $f\in BV^\alpha(\R^n)$ and $\Omega\subset\R^n$ is an open set such that either $\Omega$ is bounded and $|D^\alpha f|(\de\Omega)=0$ or $\Omega=\R^n$, then
\begin{equation*}
\Gamma(L^1)\mhypen\lim_{\beta\to\alpha^-}|D^\beta f|(\Omega)
=|D^\alpha f|(\Omega),
\end{equation*}
see \cref{res:Gamma_liminf_alpha_beta} and \cref{res:Gamma_limsup_alpha_beta}.

\subsection{Future developments: asymptotics for the fractional \texorpdfstring{$\alpha$}{alpha}-variation as \texorpdfstring{$\alpha\to0^+$}{alpha tends to 0+}}

Having in mind the limit~\eqref{intro_eq:banana}, it would be interesting to understand what happens to the fractional $\alpha$-variation~\eqref{intro_eq:def_frac_var} as $\alpha\to0^+$. Note that
\begin{equation}\label{intro_eq:mu_n_alpha_asym_zero}
\lim_{\alpha\to0^+}\mu_{n,\alpha}
= \pi^{- \frac{n}{2}} \frac{\Gamma\left ( \frac{n + 1}{2} \right )}{\Gamma\left ( \frac{1}{2} \right )}
=:\mu_{n,0},
\end{equation} 
so there is no renormalization factor as $\alpha\to0^+$, differently from~\eqref{intro_eq:mu_n_alpha_asym_1}.

At least formally, as $\alpha\to0^+$ the fractional $\alpha$-gradient in~\eqref{intro_eq:def_nabla_alpha} is converging to the operator
\begin{equation}\label{intro_eq:def_nabla_0}
\nabla^0 f(x):=\mu_{n,0}\int_{\R^n}\frac{(y-x)(f(y)-f(x))}{|y-x|^{n+1}}\,dy,
\qquad
x\in\R^n.	
\end{equation}
The operator in~\eqref{intro_eq:def_nabla_0} is well defined (in the principal value sense) for all $f\in C^\infty_c(\R^n)$ and, actually, coincides with the well-known vector-valued \emph{Riesz transform}~$Rf$, see~\cite{G14-C}*{Section~5.1.4} and~\cite{S70}*{Chapter~3}. Similarly, the fractional $\alpha$-divergence in~\eqref{intro_eq:def_div_alpha} is formally converging to the operator
\begin{equation}\label{intro_eq:def_div_zero}
\div^0\phi(x):=\mu_{n,0}\int_{\R^n}\frac{(y-x)\cdot(\phi(y)-\phi(x))}{|y-x|^{n+1}}\,dy,
\qquad
x\in\R^n,	
\end{equation} 
which is well defined (in the principal value sense) for all $\phi\in C^\infty_c(\R^n;\R^n)$. 

In perfect analogy with what we did before, we can introduce the space $BV^0(\R^n)$ as the space of functions $f\in L^1(\R^n)$ such that the quantity
\begin{equation*}
|D^0 f|(\R^n):=\sup\set*{\int_{\R^n}f\,\div^0\phi\,dx : \phi\in C^\infty_c(\R^n;\R^n),\ \|\phi\|_{L^\infty(\R^n;\,\R^n)}\le1}
\end{equation*}
is finite.
Surprisingly (and differently from the fractional $\alpha$-variation, recall~\cite{CS18}*{Section~3.10}), it turns out that $|D^0 f|\ll\Leb{n}$ for all $f\in BV^0(\R^n)$. More precisely, one can actually prove that $BV^0(\R^n)=H^1(\R^n)$, in the sense that $f\in BV^0(\R^n)$ if and only if $f\in H^1(\R^n)$, with $D^0f=Rf\Leb{n}$ in~$\mathscr M(\R^n;\R^n)$. Here
\begin{equation*}
H^1(\R^n):=\set*{f\in L^1(\R^n) : Rf\in L^1(\R^n;\R^n)}	
\end{equation*}
is the (real) \emph{Hardy space}, see~\cite{S93}*{Chapter~III} for the precise definition. Thus, it would be interesting to understand for which functions $f\in L^1(\R^n)$ the fractional $\alpha$-gradient $\nabla^\alpha f$ tends (in a suitable sense) to the Riesz transform $Rf$ as $\alpha\to0^+$. Of course, if $Rf\notin L^1(\R^n;\R^n)$, that is, $f\notin H^1(\R^n)$, then one cannot expect strong convergence in~$L^1$ and, instead, may consider the asymptotic behavior of the rescaled fractional gradient $\alpha\,\nabla^\alpha f$ as~$\alpha\to0^+$, in analogy with the limit in~\eqref{intro_eq:banana}.     
This line of research, as well as the identifications $BV^0=H^1$ and $S^{\alpha,p}=L^{\alpha,p}$ mentioned above, it is the subject of the subsequent paper~\cite{BCCS}.

\subsection{Organization of the paper}

The paper is organized as follows.

In \cref{sec:preliminaires}, after having briefly recalled the definitions and the main properties of the operators~$\nabla^\alpha$ and~$\div^\alpha$, we extend certain technical results of~\cite{CS18}. 

In \cref{sec:estimates_representations}, we prove several integrability properties of the fractional $\alpha$-gradient and two useful representation formulas for the fractional $\alpha$-variation of functions with bounded De Giorgi's variation. We are also able to prove similar results for the fractional $\beta$-gradient of functions with bounded fractional $\alpha$-variation, see \cref{subsec:inclusion_BV_alpha_Sobolev_beta}. 

In \cref{sec:asymptotic_alpha_to_1}, we study the asymptotic behavior of the fractional $\alpha$-variation as $\alpha\to1^-$ and prove pointwise-convergence and $\Gamma$-convergence results, dealing separately with the integrability exponents $p=1$, $p\in(1,+\infty)$ and $p=+\infty$. 

In \cref{sec:asymptotic_beta_to_alpha}, we show that the fractional $\beta$-variation weakly converges and $\Gamma$-converges to the fractional $\alpha$-variation as $\beta\to\alpha^-$ for any $\alpha\in(0,1)$. 

In \cref{sec:appendix}, for the reader's convenience, we state and prove two known results on the truncation and the approximation of~$BV$ functions and sets with finite perimeter that are used in \cref{sec:estimates_representations} and in \cref{sec:asymptotic_alpha_to_1}.

\section{Preliminaries}
\label{sec:preliminaires}

\subsection{General notation}
\label{subsec:notation}

We start with a brief description of the main notation used in this paper. In order to keep the exposition the most reader-friendly as possible, we retain the same notation adopted in our previous work~\cite{CS18}.

Given an open set $\Omega$, we say that a set $E$ is compactly contained in $\Omega$, and we write \mbox{$E\Subset\Omega$}, if the $\overline{E}$ is compact and contained in $\Omega$.
We denote by $\Leb{n}$ and $\Haus{\alpha}$ the $n$-dimensional Lebesgue measure and the $\alpha$-dimensional Hausdorff measure on $\R^n$ respectively, with $\alpha \ge 0$. Unless otherwise stated, a measurable set is a $\Leb{n}$-measurable set. We also use the notation $|E|=\Leb{n}(E)$. All functions we consider in this paper are Lebesgue measurable, unless otherwise stated. We denote by $B_r(x)$ the standard open Euclidean ball with center $x\in\R^n$ and radius $r>0$. We let $B_r=B_r(0)$. Recall that $\omega_{n} := |B_1|=\pi^{\frac{n}{2}}/\Gamma\left(\frac{n+2}{2}\right)$ and $\Haus{n-1}(\partial B_{1}) = n \omega_n$, where $\Gamma$ is Euler's \emph{Gamma function}, see~\cite{A64}. 

We let $\mathrm{GL}(n)\supset\mathrm{O}(n)\supset\mathrm{SO}(n)$ be the \emph{general linear group}, the \emph{orthogonal group} and the \emph{special orthogonal group} respectively. We tacitly identify $\mathrm{GL}(n)\subset\R^{n^2}$ with the space of invertible $n\times n$\,-\,matrices and we endow it with the usual Euclidean distance in~$\R^{n^2}$.

For $k \in \N_{0} \cup \set{+ \infty}$ and $m \in \N$, we denote by $C^{k}_{c}(\Omega ; \R^{m})$ and $\Lip_c(\Omega; \R^{m})$ the spaces of $C^{k}$-regular and, respectively, Lipschitz-regular, $m$-vector-valued functions defined on~$\R^n$ with compact support in~$\Omega$.

For any exponent $p\in[1,+\infty]$, we denote by $L^p(\Omega;\R^m)$ the space of $m$-vector-valued Lebesgue $p$-integrable functions on~$\Omega$.
For $p\in[1,+\infty]$, we say that $(f_k)_{k\in\N}\subset L^p(\Omega;\R^m)$ \emph{weakly converges} to $f\in L^p(\Omega;\R^m)$, and we write $f_k\weakto f$ in $L^p(\Omega;\R^m)$ as $k\to+\infty$, if 
\begin{equation}\label{eq:def_weak_conv_Lp}
\lim_{k\to+\infty}\int_\Omega f_k\cdot\phi\,dx
=\int_\Omega f\cdot\phi\,dx
\end{equation}
for all $\phi\in L^q(\Omega;\R^m)$, with $q\in[1,+\infty]$ the \emph{conjugate exponent} of~$p$, that is, $\frac{1}{p}+\frac{1}{q}=1$ (with the usual convention $\frac{1}{+\infty}=0$). Note that in the case $p=+\infty$ we make a little abuse of terminology, since the limit in~\eqref{eq:def_weak_conv_Lp} actually defines the \emph{weak*-convergence} in~$L^\infty(\Omega;\R^m)$. 

We let
\begin{equation*}
W^{1,p}(\Omega;\R^m):=\set*{u\in L^p(\Omega;\R^m) : [u]_{W^{1,p}(\Omega;\,\R^m)}:=\|\nabla u\|_{L^p(\Omega;\,\R^{nm})}<+\infty}
\end{equation*}
be the space of $m$-vector-valued Sobolev functions on~$\Omega$, see for instance~\cite{L09}*{Chapter~10} for its precise definition and main properties. We also let
\begin{equation*}
w^{1,p}(\Omega;\R^m):=\set*{u\in L^1_{\loc}(\Omega;\R^m) : [u]_{W^{1,p}(\Omega;\,\R^m)}<+\infty}.
\end{equation*}
We let
\begin{equation*}
BV(\Omega;\R^m):=\set*{u\in L^1(\Omega;\R^m) : [u]_{BV(\Omega;\,\R^m)}:=|Du|(\Omega)<+\infty}
\end{equation*}
be the space of $m$-vector-valued functions of bounded variation on~$\Omega$, see for instance~\cite{AFP00}*{Chapter~3} or~\cite{EG15}*{Chapter~5} for its precise definition and main properties. We also let 
\begin{equation*}
bv(\Omega;\R^m):=\set*{u\in L^1_{\loc}(\Omega;\R^m) : [u]_{BV(\Omega;\,\R^m)}<+\infty}.
\end{equation*}
For $\alpha\in(0,1)$ and $p\in[1,+\infty)$, we let
\begin{equation*}
W^{\alpha,p}(\Omega;\R^m):=\set*{u\in L^p(\Omega;\R^m) : [u]_{W^{\alpha,p}(\Omega;\,\R^m)}\!:=\left(\int_\Omega\int_\Omega\frac{|u(x)-u(y)|^p}{|x-y|^{n+p\alpha}}\,dx\,dy\right)^{\frac{1}{p}}\!<+\infty}
\end{equation*}
be the space of $m$-vector-valued fractional Sobolev functions on~$\Omega$, see~\cite{DiNPV12} for its precise definition and main properties. We also let 
\begin{equation*}
w^{\alpha,p}(\Omega; \R^m):=\set*{u\in L^1_{\loc}(\Omega;\R^m) : [u]_{W^{\alpha,p}(\Omega;\,\R^m)}<+\infty}.
\end{equation*}
For $\alpha\in(0,1)$ and $p=+\infty$, we simply let
\begin{equation*}
W^{\alpha,\infty}(\Omega;\R^m):=\set*{u\in L^\infty(\Omega;\R^m) : \sup_{x,y\in \Omega,\, x\neq y}\frac{|u(x)-u(y)|}{|x-y|^\alpha}<+\infty},
\end{equation*}
so that $W^{\alpha,\infty}(\Omega;\R^m)=C^{0,\alpha}_b(\Omega;\R^m)$, the space of $m$-vector-valued bounded $\alpha$-H\"older continuous functions on~$\Omega$.

We let $\mathscr{M}(\Omega;\R^m)$ be the space of $m$-vector-valued Radon measures with finite total variation, precisely
\begin{equation*}
|\mu|(\Omega):=\sup\set*{\int_\Omega\phi\cdot d\mu : \phi\in C_c^0(\Omega;\R^m),\ \|\phi\|_{L^\infty(\Omega;\,\R^m)}\le1}
\end{equation*}
for $\mu\in\mathscr{M}(\Omega;\R^m)$. We say that $(\mu_k)_{k\in\N}\subset\mathscr{M}(\Omega;\R^m)$ \emph{weakly converges} to $\mu\in\mathscr{M}(\Omega;\R^m)$, and we write $\mu_k\weakto\mu$ in $\mathscr{M}(\Omega;\R^m)$ as $k\to+\infty$, if 
\begin{equation}\label{eq:def_weak_conv_meas}
\lim_{k\to+\infty}\int_\Omega\phi\cdot d\mu_k=\int_\Omega\phi\cdot d\mu
\end{equation} 
for all $\phi\in C_c^0(\Omega;\R^m)$. Note that we make a little abuse of terminology, since the limit in~\eqref{eq:def_weak_conv_meas} actually defines the \emph{weak*-convergence} in~$\mathscr{M}(\Omega;\R^m)$. 

In order to avoid heavy notation, if the elements of a function space $F(\Omega;\R^m)$ are real-valued (i.e.~$m=1$), then we will drop the target space and simply write~$F(\Omega)$.

\subsection{Basic properties of~\texorpdfstring{$\nabla^\alpha$}{nablaˆalpha} and~\texorpdfstring{$\div^\alpha$}{divˆalpha}}

We recall the non-local operators~$\nabla^\alpha$ and~$\diverg^\alpha$ introduced by \v{S}ilhav\'{y} in~\cite{S19} (see also our previous work~\cite{CS18}).

Let $\alpha\in(0,1)$ and set  
\begin{equation*}
\mu_{n, \alpha} := 2^{\alpha} \pi^{- \frac{n}{2}} \frac{\Gamma\left ( \frac{n + \alpha + 1}{2} \right )}{\Gamma\left ( \frac{1 - \alpha}{2} \right )}.
\end{equation*}
We let
\begin{equation*}
\nabla^{\alpha} f(x) := \mu_{n, \alpha} \lim_{\eps \to 0} \int_{\{ |z| > \eps \}} \frac{z f(x + z)}{|z|^{n + \alpha + 1}} \, dz
\end{equation*}
be the \emph{fractional $\alpha$-gradient} of $f\in\Lip_c(\R^n)$ at $x\in\R^n$. We also let
\begin{equation*}
\div^{\alpha} \varphi(x) := \mu_{n, \alpha} \lim_{\eps \to 0} \int_{\{ |z| > \eps \}} \frac{z \cdot \varphi(x + z)}{|z|^{n + \alpha + 1}} \, dz
\end{equation*}
be the \emph{fractional $\alpha$-divergence} of $\phi\in\Lip_c(\R^n;\R^n)$ at $x\in\R^n$. The non-local operators~$\nabla^\alpha$ and~$\diverg^\alpha$ are well defined in the sense that the involved integrals converge and the limits exist, see~\cite{S19}*{Section~7} and \cite{CS18}*{Section 2}. Moreover, since 
\begin{equation*}
\int_{\set*{|z| > \eps}} \frac{z}{|z|^{n + \alpha + 1}} \, dz=0,
\qquad
\forall\eps>0,
\end{equation*}
it is immediate to check that $\nabla^{\alpha}c=0$ for all $c\in\R$ and
\begin{align*}
\nabla^{\alpha} f(x)
&=\mu_{n, \alpha} \lim_{\eps \to 0} \int_{\{ |y -x| > \eps \}} \frac{(y - x)}{|y - x|^{n + \alpha + 1}} f(y) \, dy\\
&= \mu_{n, \alpha} \lim_{\eps \to 0} \int_{\{ |x - y| > \eps \}} \frac{(y - x)  (f(y) - f(x)) }{|y - x|^{n + \alpha + 1}} \, dy\\
&=\mu_{n, \alpha} \int_{\R^{n}} \frac{(y - x)  (f(y) - f(x)) }{|y - x|^{n + \alpha + 1}} \, dy,
\qquad
\forall x\in\R^n,
\end{align*}
for all $f\in\Lip_c(\R^n)$. Analogously, we also have
\begin{align*}
\div^{\alpha} \varphi(x) 
&= \mu_{n, \alpha} \lim_{\eps \to 0} \int_{\{ |x - y| > \eps \}} \frac{(y - x) \cdot \varphi(y) }{|y - x|^{n + \alpha + 1}} \, dy,\\
&= \mu_{n, \alpha} \lim_{\eps \to 0} \int_{\{ |x - y| > \eps \}} \frac{(y - x) \cdot (\varphi(y) - \varphi(x)) }{|y - x|^{n + \alpha + 1}} \, dy,\\
&= \mu_{n, \alpha} \int_{\R^{n}} \frac{(y - x) \cdot (\varphi(y) - \varphi(x)) }{|y - x|^{n + \alpha + 1}} \, dy,
\qquad
\forall x\in\R^n,
\end{align*}
for all $\phi\in\Lip_c(\R^n)$.

Given $\alpha\in(0,n)$, we let
\begin{equation}\label{eq:Riesz_potential_def} 
I_{\alpha} u(x) := 
\frac{\mu_{n, 1-\alpha}}{n - \alpha}
\int_{\R^{n}} \frac{u(y)}{|x - y|^{n - \alpha}} \, dy, 
\qquad
x\in\R^n,
\end{equation}
be the \emph{Riesz potential} of order $\alpha\in(0,n)$ of a function $u\in C^\infty_c(\R^n;\R^m)$. We recall that, if $\alpha,\beta\in(0,n)$ satisfy $\alpha+\beta<n$, then we have the following \emph{semigroup property}
\begin{equation}\label{eq:Riesz_potential_semigroup}
I_{\alpha}(I_\beta u)=I_{\alpha+\beta}u
\end{equation}
for all $u\in C^\infty_c(\R^n;\R^m)$. In addition, if $1<p<q<+\infty$ satisfy 
\begin{equation*}
\frac{1}{q}=\frac{1}{p}-\frac{\alpha}{n},	
\end{equation*}
then there exists a constant $C_{n,\alpha,p}>0$ such that the operator in~\eqref{eq:Riesz_potential_def} satisfies
\begin{equation}\label{eq:Riesz_potential_boundedness}
\|I_\alpha u\|_{L^q(\R^n;\,\R^m)}\le C_{n,\alpha,p}\|u\|_{L^p(\R^n;\,\R^m)}
\end{equation}
for all $u\in C^\infty_c(\R^n;\,\R^m)$. As a consequence, the operator in~\eqref{eq:Riesz_potential_def} extends to a linear continuous operator from $L^p(\R^n;\R^m)$ to $L^q(\R^n;\R^m)$, for which we retain the same notation. For a proof of~\eqref{eq:Riesz_potential_semigroup} and~\eqref{eq:Riesz_potential_boundedness}, we refer the reader to~\cite{S70}*{Chapter~V, Section~1} and to~\cite{G14}*{Section~1.2.1}.

We can now recall the following result, see~\cite{CS18}*{Proposition~2.2 and Corollary 2.3}.

\begin{proposition}\label{prop:frac_div_repr} 
Let $\alpha\in(0,1)$. If $f\in\Lip_c(\R^{n})$, then 
\begin{equation}\label{eq:frac_nabla_repres}
\nabla^{\alpha} f
= I_{1-\alpha}\nabla f 
= \nabla I_{1 - \alpha} f
\end{equation}
and $\nabla^{\alpha}f \in L^1(\R^n; \R^{n})\cap L^{\infty}(\R^{n}; \R^{n})$, with
\begin{equation}\label{eq:frac_nabla_repr_Lip_L1_estimate} 
\| \nabla^{\alpha} f \|_{L^1(\R^{n};\,\R^n)} \le \mu_{n, \alpha}  [f]_{W^{\alpha,1}(\R^{n})}
\end{equation}
and
\begin{equation} \label{eq:frac_nabla_repr_Lip_estimate} 
\| \nabla^{\alpha} f \|_{L^{\infty}(\R^{n};\,\R^n)} \le C_{n, \alpha, U}  \|\nabla f\|_{L^{\infty}(\R^{n};\,\R^n)}
\end{equation}
for any bounded open set $U\subset\R^n$ such that $\supp(f) \subset U$, where 
\begin{equation} \label{eq:constant_estimate} 
C_{n, \alpha, U} := 
\frac{n \mu_{n, \alpha}}{(1 - \alpha)(n + \alpha - 1)}
\left(
	\omega_n\diam(U)^{1 - \alpha} 
	+\left(
		\frac{n \omega_{n}}{n+\alpha-1}
	\right)^\frac{n + \alpha - 1}{n}|U|^\frac{1 - \alpha}{n}
\right). 
\end{equation}
Analogously, if $\phi \in\Lip_c(\R^{n}; \R^{n})$ then 
\begin{equation}\label{eq:frac_div_repres}
\div^{\alpha} \phi= I_{1-\alpha}\div \phi 
= \div I_{1 - \alpha} \phi
\end{equation}
and $\div^{\alpha} \phi \in L^1(\R^n)\cap L^{\infty}(\R^{n})$, with
\begin{equation}\label{eq:frac_div_repr_Lip_L1_estimate} 
\| \div^{\alpha} \phi \|_{L^1(\R^{n})} \le \mu_{n, \alpha}  [\phi]_{W^{\alpha,1}(\R^{n};\,\R^n)}
\end{equation}
and
\begin{equation} \label{eq:frac_div_repr_Lip_estimate} 
\| \div^{\alpha} \phi \|_{L^{\infty}(\R^{n})} \le C_{n, \alpha, U}  \|\div \phi \|_{L^{\infty}(\R^{n})}
\end{equation}
for any bounded open set $U\subset\R^n$ such that $\supp(\phi) \subset U$, where $C_{n, \alpha, U}$ is as in~\eqref{eq:constant_estimate}.
\end{proposition}

\subsection{Extension of~\texorpdfstring{$\nabla^\alpha$}{nablaˆalpha} and~\texorpdfstring{$\div^\alpha$}{divˆalpha} to \texorpdfstring{$\Lip_b$}{bounded Lip}-regular tests}

In the following result, we extend the fractional $\alpha$-divergence to $\Lip_b$-regular vector fields.

\begin{lemma}[Extension of~$\div^\alpha$ to~$\Lip_b$]
\label{res:frac_div_Lip_b}
Let $\alpha\in(0,1)$. The operator 
\begin{equation*}
\div^\alpha\colon\Lip_b(\R^n;\R^n)\to L^\infty(\R^n)
\end{equation*}
given by
\begin{equation}\label{eq:def_frac_div_Lip_b}
\div^\alpha\phi(x):=\mu_{n,\alpha}\int_{\R^n}\frac{(y-x)\cdot(\phi(y)-\phi(x))}{|y-x|^{n+\alpha+1}}\,dy,
\quad x\in\R^n,
\end{equation}
for all $\phi\in\Lip_b(\R^n;\R^n)$, is well defined, with
\begin{equation}\label{eq:frac_div_Lip_b_bounded}
\|\div^\alpha\phi\|_{L^\infty(\R^n)}	
\le\frac{2^{1-\alpha}n\omega_n\mu_{n,\alpha}}{\alpha(1-\alpha)}\Lip(\phi)^\alpha\|\phi\|_{L^\infty(\R^n;\,\R^n)}^{1-\alpha},
\end{equation}
and satisfies
\begin{equation}\label{eq:frac_div_Lip_b_def_eps}
\begin{split}
\div^\alpha\phi(x)
&=\mu_{n,\alpha}\lim_{\eps\to0^+}\int_{\set*{|y-x|>\eps}}\frac{(y-x)\cdot(\phi(y)-\phi(x))}{|y-x|^{n+\alpha+1}}\,dy\\
&=\mu_{n,\alpha}\lim_{\eps\to0^+}\int_{\set*{|y-x|>\eps}}\frac{(y-x)\cdot\phi(y)}{|y-x|^{n+\alpha+1}}\,dy
\end{split}
\end{equation}
for all $x\in\R^n$. Moreover, if in addition $I_{1-\alpha}|\div\phi|\in L^1_{\loc}(\R^n)$, then 
\begin{equation}\label{eq:frac_div_Lip_b_riesz}
\div^\alpha\phi(x)=I_{1-\alpha}\div\phi(x)
\end{equation}
for a.e.~$x\in\R^n$.
\end{lemma}

\begin{proof}
We split the proof in two steps.

\smallskip

\textit{Step~1: proof of~\eqref{eq:def_frac_div_Lip_b}, \eqref{eq:frac_div_Lip_b_bounded} and~\eqref{eq:frac_div_Lip_b_def_eps}}.
Given $x\in\R^n$ and $r>0$, we can estimate
\begin{equation*}
\int_{\set*{|y-x|\le r}}\abs*{\frac{(y-x)\cdot(\phi(y)-\phi(x))}{|y-x|^{n+\alpha+1}}}\,dy
\le n\omega_n\!\Lip(\phi)\int_0^r\rho^{-\alpha}\,d\rho
\end{equation*}
and
\begin{equation*}
\int_{\set*{|y-x|>r}}\abs*{\frac{(y-x)\cdot(\phi(y)-\phi(x))}{|y-x|^{n+\alpha+1}}}\,dy
\le2n\omega_n\|\phi\|_{L^\infty(\R^n;\,\R^n)}\int_r^{+\infty}\rho^{-(1+\alpha)}\,d\rho.
\end{equation*}
Hence the function in~\eqref{eq:def_frac_div_Lip_b} is well defined for all $x\in\R^n$ and
\begin{equation*}
\|\div^\alpha\phi\|_{L^\infty(\R^n)}	
\le n\omega_n\left(\frac{\Lip(\phi)}{1-\alpha}\,r^{1-\alpha}+\frac{2\|\phi\|_{L^\infty(\R^n;\,\R^n)}}{\alpha}\,r^{-\alpha}\right),
\end{equation*}
so that~\eqref{eq:frac_div_Lip_b_bounded} follows by optimising the right-hand side in~$r>0$.
Moreover, since
\begin{align*}
&\abs*{\frac{(y-x)\cdot(\phi(y)-\phi(x))}{|y-x|^{n+\alpha+1}}\,\chi_{(\eps,+\infty)}(|y-x|)}\\
&\qquad\le\Lip(\phi)\,\frac{\chi_{(0,1)}(|y-x|)}{|y-x|^{n+\alpha-1}}
+2\|\phi\|_{L^\infty(\R^n;\,\R^n)}\,\frac{\chi_{[1,+\infty)}(|y-x|)}{|y-x|^{n+\alpha}}
\in L^1_{x,y}(\R^n)
\end{align*}
and 
\begin{equation*}
\int_{\set*{|z|>\eps}}\frac{z}{|z|^{n+\alpha+1}}\,dy=0
\end{equation*}
for all $\eps>0$, by Lebesgue's Dominated Convergence Theorem we immediately get the two equalities in~\eqref{eq:frac_div_Lip_b_def_eps} for all $x\in\R^n$. 

\smallskip

\textit{Step~2: proof of~\eqref{eq:frac_div_Lip_b_riesz}}. Assume that $I_{1-\alpha}|\div\phi|\in L^1_{\loc}(\R^n)$. Then 
\begin{equation}\label{eq:a_e_L1_div_riesz}
\frac{|\div\phi(y)|}{|y-x|^{n+\alpha-1}}\in L^1_y(\R^n)
\end{equation}  
for a.e.~$x\in\R^n$. Hence, by Lebesgue's Dominated Convergence Theorem, we can write 
\begin{equation*}
I_{1-\alpha}\div\phi(x)
=\mu_{n,\alpha}\lim_{\eps\to0^+}\int_{\set*{|y-x|>\eps}}\frac{\div\phi(y)}{|y-x|^{n+\alpha-1}}\,dy
\end{equation*}
for a.e.~$x\in\R^n$. Now let $\eps>0$ be fixed and let $R>0$. Again by~\eqref{eq:a_e_L1_div_riesz} and Lebesgue's Dominated Convergence Theorem, we have
\begin{equation*}
\lim_{R\to+\infty}\int_{\set*{R>|y-x|>\eps}}\frac{\div\phi(y)}{|y-x|^{n+\alpha-1}}\,dy=\int_{\set*{|y-x|>\eps}}\frac{\div\phi(y)}{|y-x|^{n+\alpha-1}}\,dy
\end{equation*}
for a.e.~$x\in\R^n$. Moreover, integrating by parts, we get
\begin{align*}
\int_{\set*{R>|y-x|>\eps}}&\frac{\div\phi(y)}{|y-x|^{n+\alpha-1}}\,dy
=\int_{\set*{R>|y|>\eps}}\frac{\div_y\phi(y+x)}{|y|^{n+\alpha-1}}\,dy\\
&=\int_{\set*{|y|=R}}\frac{y}{|y|}\frac{\phi(y+x)}{|y|^{n+\alpha-1}}\,d\Haus{n-1}(y)
-\int_{\set*{|y|=\eps}}\frac{y}{|y|}\frac{\phi(y+x)}{|y|^{n+\alpha-1}}\,d\Haus{n-1}(y)\\
&\quad+\int_{\set*{R>|y|>\eps}}\frac{y\cdot\phi(y+x)}{|y|^{n+\alpha+1}}\,dy
\end{align*}
for all $R>0$ and for a.e.~$x\in\R^n$. Since $\phi\in L^\infty(\R^n;\R^n)$, by Lebesgue's Dominated Convergence Theorem we have
\begin{equation*}
\lim_{R\to+\infty}\int_{\set*{R>|y|>\eps}}\frac{y\cdot\phi(y+x)}{|y|^{n+\alpha+1}}\,dy
=\int_{\set*{|y|>\eps}}\frac{y\cdot\phi(y+x)}{|y|^{n+\alpha+1}}\,dy
\end{equation*}
for all $\eps>0$ and all $x\in\R^n$. We can also estimate
\begin{equation*}
\abs*{\int_{\set*{|y|=R}}\frac{y}{|y|}\frac{\phi(y+x)}{|y|^{n+\alpha-1}}\,d\Haus{n-1}(y)}
\le n\omega_n\|\phi\|_{L^\infty(\R^n;\,\R^n)}R^{-\alpha}
\end{equation*}
for all $R>0$ and all $x\in\R^n$. We thus have that
\begin{align*}
\int_{\set*{|y-x|>\eps}}\frac{\div\phi(y)}{|y-x|^{n+\alpha-1}}\,dy
=\int_{\set*{|y|>\eps}}\frac{y\cdot\phi(y+x)}{|y|^{n+\alpha+1}}\,dy
-\int_{\set*{|y|=\eps}}\frac{y}{|y|}\frac{\phi(y+x)}{|y|^{n+\alpha-1}}\,d\Haus{n-1}(y)
\end{align*}
for all $\eps>0$ and a.e.~$x\in\R^n$. Since also
\begin{align*}
\abs*{\int_{\set*{|y|=\eps}}\frac{y}{|y|}\frac{\phi(y+x)}{|y|^{n+\alpha-1}}\,d\Haus{n-1}(y)}
&=\abs*{\int_{\set*{|y|=\eps}}\frac{y}{|y|}\frac{\phi(y+x)-\phi(x)}{|y|^{n+\alpha-1}}\,d\Haus{n-1}(y)}\\
&\le n\omega_n\Lip(\phi)\,\eps^{1-\alpha}
\end{align*}
for all $\eps>0$ and $x\in\R^n$, we conclude that
\begin{align*}
\lim_{\eps\to0^+}\int_{\set*{|y-x|>\eps}}\frac{\div\phi(y)}{|y-x|^{n+\alpha-1}}\,dy
=\lim_{\eps\to0^+}\int_{\set*{|y-x|>\eps}}\frac{(y-x)\cdot\phi(y)}{|y-x|^{n+\alpha+1}}\,dy
\end{align*}
for a.e.~$x\in\R^n$, proving~\eqref{eq:frac_div_Lip_b_riesz}.
\end{proof}

We can also extend the fractional $\alpha$-gradient to~$\Lip_b$-regular functions. The proof is very similar to the one of \cref{res:frac_div_Lip_b} and is left to the reader.

\begin{lemma}[Extension of~$\nabla^\alpha$ to~$\Lip_b$] \label{res:frac_nabla_Lip_b}
Let $\alpha\in(0,1)$. The operator 
\begin{equation*}
\nabla^\alpha\colon\Lip_b(\R^n)\to L^\infty(\R^n;\R^n)
\end{equation*}
given by
\begin{equation*}
\nabla^\alpha f(x):=\mu_{n,\alpha}\int_{\R^n}\frac{(y-x)\cdot(f(y)-f(x))}{|y-x|^{n+\alpha+1}}\,dy,
\quad x\in\R^n,
\end{equation*}
for all $f\in\Lip_b(\R^n)$, is well defined, with
\begin{equation*}
\|\nabla^\alpha f\|_{L^\infty(\R^n;\,\R^n)}	
\le\frac{2^{1-\alpha}n\omega_n\mu_{n,\alpha}}{\alpha(1-\alpha)}\Lip(f)^\alpha\|f\|_{L^\infty(\R^n)}^{1-\alpha},
\end{equation*}
and satisfies
\begin{equation*}
\begin{split}
\nabla^\alpha f(x)
&=\mu_{n,\alpha}\lim_{\eps\to0^+}\int_{\set*{|y-x|>\eps}}\frac{(y-x)\cdot(f(y)-f(x))}{|y-x|^{n+\alpha+1}}\,dy\\
&=\mu_{n,\alpha}\lim_{\eps\to0^+}\int_{\set*{|y-x|>\eps}}\frac{(y-x)\cdot f(y)}{|y-x|^{n+\alpha+1}}\,dy
\end{split}
\end{equation*}
for all $x\in\R^n$. Moreover, if in addition $I_{1-\alpha}|\nabla f|\in L^1_{\loc}(\R^n)$, then 
\begin{equation*}
\nabla^\alpha f(x)=I_{1-\alpha}\nabla f(x)
\end{equation*}
for a.e.~$x\in\R^n$.
\end{lemma}

\subsection{Extended Leibniz's rules for~\texorpdfstring{$\nabla^\alpha$}{nablaˆalpha} and~\texorpdfstring{$\div^\alpha$}{divˆalpha}}

The following two results extend the validity of Leibniz's rules proved in~\cite{CS18}*{Lemmas~2.6 and~2.7} to $\Lip_b$-regular functions and $\Lip_b$-regular vector fields. The proofs are very similar to the ones given in~\cite{CS18} and to that of \cref{res:frac_div_Lip_b}, and thus are left to the reader.

\begin{lemma}[Extended Leibniz's rule for $\nabla^\alpha$]
\label{res:leibniz_frac_nabla_Lip_b}
Let $\alpha\in(0,1)$. If $f\in\Lip_b(\R^n)$ and $\eta\in\Lip_c(\R^n)$, then
\begin{equation*}
\nabla^\alpha(\eta f)
=\eta\,\nabla^\alpha f+ f\,\nabla^\alpha\eta+\nabla^\alpha_{\rm NL}(\eta,f),
\end{equation*}
where
\begin{equation*}
\nabla^\alpha_{\rm NL}(\eta,f)(x)
=\mu_{n,\alpha}\int_{\R^n}\frac{(y-x)\cdot(f(y)-f(x))(\eta(y)-\eta(x))}{|y-x|^{n+\alpha+1}}\,dy
\end{equation*}
for all $x\in\R^n$, with
\begin{equation*}
\|\nabla^\alpha_{\rm NL}(\eta,f)\|_{L^\infty(\R^n;\,\R^n)}
\le\frac{2^{2-\alpha}n\omega_n\mu_{n,\alpha}\|f\|_{L^\infty(\R^n)}}{\alpha(1-\alpha)}\Lip(\eta)^\alpha\|\eta\|_{L^\infty(\R^n)}^{1-\alpha}
\end{equation*}
and
\begin{equation*}
\|\nabla^\alpha_{\rm NL}(\eta,f)\|_{L^1(\R^n;\,\R^n)}
\le
2\mu_{n,\alpha}\|f\|_{L^\infty(\R^n)}[\eta]_{W^{\alpha,1}(\R^n)}.
\end{equation*}
\end{lemma}

\begin{lemma}[Extended Leibniz's rule for $\div^\alpha$]
\label{res:leibniz_frac_div_Lip_b}
Let $\alpha\in(0,1)$. If $\phi\in\Lip_b(\R^n;\R^n)$ and $\eta\in\Lip_c(\R^n)$, then
\begin{equation*}
\div^\alpha(\eta\phi)
=\eta\,\div^\alpha\phi+\phi\cdot\nabla^\alpha\eta+\div^\alpha_{\rm NL}(\eta,\phi),
\end{equation*}
where
\begin{equation*}
\div^\alpha_{\rm NL}(\eta,\phi)(x)
=\mu_{n,\alpha}\int_{\R^n}\frac{(y-x)\cdot(\phi(y)-\phi(x))(\eta(y)-\eta(x))}{|y-x|^{n+\alpha+1}}\,dy
\end{equation*}
for all $x\in\R^n$, with
\begin{equation*}
\|\div^\alpha_{\rm NL}(\eta,\phi)\|_{L^\infty(\R^n)}
\le\frac{2^{2-\alpha}n\omega_n\mu_{n,\alpha}\|\phi\|_{L^\infty(\R^n;\,\R^n)}}{\alpha(1-\alpha)}\Lip(\eta)^\alpha\|\eta\|_{L^\infty(\R^n)}^{1-\alpha}
\end{equation*}
and
\begin{equation*}
\|\div^\alpha_{\rm NL}(\eta,\phi)\|_{L^1(\R^n)}
\le
2\mu_{n,\alpha}\|\phi\|_{L^\infty(\R^n;\,\R^n)}[\eta]_{W^{\alpha,1}(\R^n)}.
\end{equation*}
\end{lemma}

\subsection{Extended integration-by-part formulas}

We now recall the definition of the space of functions with bounded fractional $\alpha$-variation. Given $\alpha \in (0, 1)$, we let
\begin{equation*}\label{intro_eq:BV_alpha}
BV^\alpha(\R^n):=\set*{f\in L^1(\R^n) : |D^\alpha f|(\R^n)<+\infty},
\end{equation*}
where
\begin{equation*}\label{intro_eq:def_frac_variation}
|D^\alpha f|(\R^n)=\sup\set*{\int_{\R^n}f\,\div^\alpha\phi\,dx : \phi\in C^\infty_c(\R^n;\R^n),\ \|\phi\|_{L^\infty(\R^n;\R^n)}\le1}
\end{equation*}
is the \emph{fractional $\alpha$-variation} of $f\in L^1(\R^n)$. We refer the reader to~\cite{CS18}*{Section~3} for the basic properties of this function space. Here we just recall the following result, see~\cite{CS18}*{Theorem 3.2 and Proposition 3.6} for the proof.

\begin{theorem}[Structure theorem for $BV^\alpha$ functions]
Let $\alpha\in(0,1)$. If $f \in L^{1}(\R^{n})$, then $f \in BV^{\alpha}(\R^{n})$ if and only if there exists a finite vector-valued Radon measure $D^{\alpha} f \in \M(\R^{n}; \R^{n})$ such that
\begin{equation}\label{eq:BV_alpha_duality} 
\int_{\R^{n}} f\, \div^{\alpha} \varphi \, dx = - \int_{\R^n} \varphi \cdot d D^{\alpha} f 
\end{equation}
for all $\phi \in \Lip_{c}(\R^{n}; \R^{n})$.
\end{theorem}

Thanks to \cref{res:leibniz_frac_div_Lip_b}, we can actually prove that a function in~$BV^\alpha(\R^n)$ can be tested against any $\Lip_b$-regular vector field.

\begin{proposition}[$\Lip_b$-regular test for $BV^\alpha$ functions]
\label{res:Lip_b_test}
Let $\alpha\in(0,1)$. If $f\in BV^\alpha(\R^n)$, then \eqref{eq:BV_alpha_duality} holds for all $\phi\in\Lip_b(\R^n;\R^n)$.	
\end{proposition}

\begin{proof}
We argue as in the proof of~\cite{CS18}*{Theorem~3.8}. Fix $\phi\in\Lip_b(\R^n;\R^n)$ and let $(\eta_R)_{R>0}\subset C^\infty_c(\R^n)$ be a family of cut-off functions as in~\cite{CS18}*{Section~3.3}. On the one hand, since
\begin{equation*}
\abs*{\int_{\R^n}f\eta_R\,\div^\alpha\phi\,dx-\int_{\R^n}f\,\div^\alpha\phi\,dx}
\le\|\div^\alpha\phi\|_{L^\infty(\R^n)}	\int_{\R^n}|f|\,(1-\eta_R)\,dx
\end{equation*}
for all $R>0$, by Lebesgue's Dominated Convergence Theorem we have
\begin{equation*}
\lim_{R\to+\infty}\int_{\R^n}f\eta_R\,\div^\alpha\phi\,dx
=\int_{\R^n}f\,\div^\alpha\phi\,dx.
\end{equation*}
On the other hand, by \cref{res:leibniz_frac_div_Lip_b} we can write
\begin{align*}
\int_{\R^n}f\eta_R\,\div^\alpha\phi\,dx
=\int_{\R^n}f\,\div^\alpha(\eta_R\phi)\,dx
-\int_{\R^n}f\,\phi\cdot\nabla^\alpha\eta_R\,dx
-\int_{\R^n}f\,\div^\alpha_{\rm NL}(\eta_R,\phi)\,dx
\end{align*} 
for all $R>0$. By~\cite{CS18}*{Proposition~3.6}, we have
\begin{equation*}
\int_{\R^n}f\,\div^\alpha(\eta_R\phi)\,dx
=-\int_{\R^n}\eta_R\phi\cdot dD^\alpha f
\end{equation*}
for all $R>0$. Since
\begin{equation*}
\abs*{\int_{\R^n}\eta_R\phi\cdot dD^\alpha f-\int_{\R^n}\phi\cdot dD^\alpha f}
\le\|\phi\|_{L^\infty(\R^n;\R^n)}\int_{\R^n}(1-\eta_R)\,d|D^\alpha f|
\end{equation*}
for all $R>0$, by Lebesgue's Dominated Convergence Theorem (with respect to the finite measure~$|D^\alpha f|$) we have
\begin{equation*}
\lim_{R\to+\infty}\int_{\R^n}\eta_R\phi\cdot dD^\alpha f
=\int_{\R^n}\phi\cdot dD^\alpha f.
\end{equation*}
Finally, we can estimate
\begin{align*}
\abs*{\int_{\R^n}f\,\phi\cdot\nabla^\alpha\eta_R\,dx}
\le\mu_{n,\alpha}\|\phi\|_{L^\infty(\R^n;\R^n)}
\int_{\R^n}|f(x)|\int_{\R^n}\frac{|\eta_R(y)-\eta_R(x)|}{|y-x|^{n+\alpha}}\,dy\,dx
\end{align*}
and, similarly,
\begin{equation*}
\abs*{\int_{\R^n}f\,\div^\alpha_{\rm NL}(\eta_R,\phi)\,dx}
\le2\mu_{n,\alpha}\|\phi\|_{L^\infty(\R^n;\R^n)}
\int_{\R^n}|f(x)|\int_{\R^n}\frac{|\eta_R(y)-\eta_R(x)|}{|y-x|^{n+\alpha}}\,dy\,dx.
\end{equation*}
By Lebesgue's Dominated Convergence Theorem, we thus get that
\begin{equation*}
\lim_{R\to+\infty}\left(\int_{\R^n}f\,\phi\cdot\nabla^\alpha\eta_R\,dx
+\int_{\R^n}f\,\div^\alpha_{\rm NL}(\eta_R,\phi)\,dx\right)=0
\end{equation*} 
and the conclusion follows.
\end{proof}

Thanks to \cref{res:leibniz_frac_nabla_Lip_b}, we can prove that a function in~$\Lip_b(\R^n)$ can be tested against any $\Lip_c$-regular vector field. The proof is very similar to the one of \cref{res:Lip_b_test} and is thus left to the reader.

\begin{proposition}[Integration by parts for $\Lip_b$-regular functions]
\label{res:Lip_b_int_by_parts}
Let $\alpha\in(0,1)$. If $f\in\Lip_b(\R^n)$, then
\begin{equation*}
\int_{\R^n}f\,\div^\alpha\phi\,dx=-\int_{\R^n}\phi\cdot\nabla^\alpha f\,dx
\end{equation*}	
for all $\phi\in\Lip_c(\R^n;\R^n)$.
\end{proposition}

\subsection{Comparison between \texorpdfstring{$W^{\alpha,1}$}{Wˆ{alpha,1}} and \texorpdfstring{$BV^{\alpha}$}{BVˆalpha} seminorms}\label{sec:comparison_seminorms}

In this section, we completely answer a question left open in~\cite{CS18}*{Section 1.4}. Given $\alpha \in (0, 1)$ and an open set $\Omega\subset\R^n$, we want to study the equality cases in the inequalities
\begin{equation*}
\|\nabla^{\alpha} f\|_{L^{1}(\R^{n};\, \R^{n})} 
\le \mu_{n, \alpha} [f]_{W^{\alpha, 1}(\R^{n})},
\qquad
|D^\alpha\chi_E|(\Omega)
\le \mu_{n,\alpha} P_\alpha(E;\Omega),
\end{equation*}
as long as $f\in W^{\alpha,1}(\R^n)$ and $P_\alpha(E;\Omega)<+\infty$. The key idea to the solution of this problem lies in the following simple result.

\begin{lemma}\label{res:easy_norm_ineq}
Let $A\subset\R^n$ be a measurable set with $\Leb{n}(A)>0$. If $F \in L^{1}(A; \R^m)$, then 
\begin{equation*}
\bigg | \int_A F(x) \, dx \, \bigg | 
\le \int_A | F(x) | \, dx,
\end{equation*}
with equality if and only if $F = f \nu$ a.e.\ in~$A$ for some constant direction $\nu \in \mathbb{S}^{m - 1}$ and some scalar function $f \in L^{1}(A)$ with $f \ge 0$ a.e.\ in~$A$.
\end{lemma}

\begin{proof}
The inequality is well known and it is obvious that it is an equality if $F = f \nu$ a.e.\ in~$A$ for some constant direction $\nu \in \mathbb{S}^{m-1}$ and some scalar function $f \in L^{1}(A)$ with $f \ge 0$ a.e.\ in~$A$. So let us assume that 
\begin{equation*}
\bigg | \int_A F(x) \, dx \, \bigg | 
= \int_A | F(x) | \, dx.
\end{equation*}
If $\int_A F(x) \, dx=0$, then also $\int_A | F(x) | \, dx=0$. Thus $F=0$ a.e.\ in~$A$ and there is nothing to prove. If $\int_A F(x) \, dx\ne0$ instead, then we can write
\begin{equation*}
\int_A |F(x)| - F(x) \cdot \nu \, dx = 0,
\end{equation*}
with
\begin{equation*}
\nu = \frac{\int_A F(x) \, dx}{| \int_A F(x) \, dx \,|}\in\mathbb{S}^{m-1}.
\end{equation*}
Therefore, we obtain $|F(x)| = F(x) \cdot \nu$ for a.e.~$x \in A$, so that $\frac{F(x)}{|F(x)|} \cdot \nu = 1$ for a.e.~$x \in A$ such that $|F(x)|\ne0$. This implies that $F = f \nu$ a.e.\ in~$A$ with $f = |F|\in L^{1}(A)$ and the conclusion follows.
\end{proof}

As an immediate consequence of \cref{res:easy_norm_ineq}, we have the following result.

\begin{corollary}
\label{lem:strict_ineq_W_norm}
Let $\alpha \in (0, 1)$. If $f \in W^{\alpha, 1}(\R^{n})$, then
\begin{equation}\label{eq:old_ineq}
\|\nabla^{\alpha} f\|_{L^{1}(\R^{n};\, \R^{n})} \le \mu_{n, \alpha} [f]_{W^{\alpha, 1}(\R^{n})},
\end{equation}
with equality if and only if $f=0$ a.e.\ in~$\R^n$.
\end{corollary}

\begin{proof}
Inequality~\eqref{eq:old_ineq} was proved in~\cite{CS18}*{Theorem~3.18}. Note that, given $f \in L^{1}(\R^{n})$, $[f]_{W^{\alpha,1}(\R^n)}=0$ if and only if $f=0$ a.e.\ and thus, in this case, \eqref{eq:old_ineq} is trivially an equality.
If~\eqref{eq:old_ineq} holds as an equality and $f$ is not equivalent to the zero function, then
\begin{equation*}
\int_{\R^n}\bigg(|\nabla^\alpha f (x)|-\mu_{n,\alpha}\int_{\R^n}\frac{|f(y)-f(x)|}{|y-x|^{n+\alpha}}\,dy\bigg)\,dx=0
\end{equation*} 
and thus
\begin{equation}\label{eq:old_equality_start}
\bigg|\int_{\R^n}\frac{(f(y)-f(x))\cdot(y-x)}{|y-x|^{n+\alpha+1}}\,dy\,\bigg|
=\int_{\R^n}\frac{|f(y)-f(x)|}{|y-x|^{n+\alpha}}\,dy
\end{equation}
for all $x\in U$, for some measurable set $U\subset\R^n$ such that $\Leb{n}(\R^n\setminus U)=0$. 
Now let $x\in U$ be fixed. By \cref{res:easy_norm_ineq} (applied with $A=\R^n$), \eqref{eq:old_equality_start} implies that the (non-identically zero) vector field
\begin{equation*}
y\mapsto (f(y) - f(x))\,(y - x),
\quad
y\in\R^n,
\end{equation*}
has constant direction for all $y\in V_x$, for some measurable set $V_x\subset\R^n$ such that \mbox{$\Leb{n}(\R^n\setminus V_x)=0$}. Thus, given $y,y'\in V_x$, the two vectors $y-x$ and $y'-x$ are linearly dependent, so that the three points~$x$, $y$ and~$y'$ are collinear. If $n\ge2$, then this immediately gives $\Leb{n}(V_x)=0$, a contradiction, so that~\eqref{eq:old_ineq} must be strict. If instead $n=1$, then we know that
\begin{equation}\label{eq:constant_sign}
x\in U
\implies
y\mapsto(f(y) - f(x))\,(y - x)\
\text{has constant sign for all $y\in V_x$}.
\end{equation}
We claim that~\eqref{eq:constant_sign} implies that the function~$f$ is (equivalent to) a (non-constant) monotone function. If so, then $f\notin L^1(\R)$, in contrast with the fact that $f\in W^{\alpha,1}(\R)$, so that~\eqref{eq:old_ineq} must be strict and the proof is concluded.
To prove the claim, we argue as follows. Fix $x\in U$ and assume that
\begin{equation}\label{eq:mon_tot}
(f(y) - f(x))\,(y - x)>0
\end{equation}
for all $y\in V_x$ without loss of generality. Now pick $x'\in U\cap V_x$ such that $x'>x$. Then, choosing $y=x'$ in~\eqref{eq:mon_tot}, we get $(f(x') - f(x))\,(x' - x)>0$ and thus $f(x')>f(x)$. Similarly, if $x'\in U\cap V_x$ is such that $x'<x$, then $f(x')<f(x)$. Hence
\begin{equation*}
\esssup_{z<x} f(z) \le f(x) \le \essinf_{z>x} f(z)
\end{equation*}
for all $x\in U$ (where $\esssup$ and $\essinf$ refer to the \emph{essential supremum} and the \emph{essential infimum} respectively) and thus~$f$ must be equivalent to a (non-constant) non-decreasing function.
\end{proof}

Given an open set $\Omega\subset\R^n$ and a measurable set $E\subset\R^n$, we define
\begin{equation*}
\tilde P_{\alpha}(E; \Omega) 
:= \int_{\Omega} \int_{\Omega} \frac{|\chi_{E}(y) - \chi_{E}(x)|}{|y - x|^{n + \alpha}} \, dx \, dy 
+ \int_{\R^{n} \setminus \Omega} \int_{\Omega} \frac{|\chi_{E}(y) - \chi_{E}(x)|}{|y - x|^{n + \alpha}} \, dx \, dy.
\end{equation*}
It is obvious to see that
\begin{equation*}
\tilde P_\alpha(E;\Omega)
\le P_\alpha(E;\Omega)
\le 2\tilde P_\alpha(E;\Omega),
\end{equation*} 
where $P_\alpha$ is the fractional perimeter introduced in~\eqref{intro_eq:def_frac_relative_perim}. Arguing as in the proof of~\cite{CS18}*{Proposition~4.8}, it is immediate to see that 
\begin{equation}\label{eq:ineq_perim_tilde}
\|\nabla^{\alpha} \chi_{E}\|_{L^{1}(\Omega;\, \R^{n})} 
\le \mu_{n, \alpha}\tilde P_{\alpha}(E;\Omega),
\end{equation}
an inequality stronger than that in~\eqref{intro_eq:bound_D_alpha_P}. In analogy with \cref{lem:strict_ineq_W_norm}, we have the following result.

\begin{corollary}\label{res:strict_per_tilde}
Let $\alpha\in(0,1)$, $\Omega\subset\R^n$ be an open set and $E\subset\R^n$ be a measurable set such that $\tilde P_{\alpha}(E;\Omega)<+\infty$.
\begin{enumerate}[(i)]

\item\label{item:strict_ineq_big2} 
If $n\ge2$, $\Leb{n}(E)>0$ and $\Leb{n}(\R^n\setminus E)>0$, then inequality~\eqref{eq:ineq_perim_tilde} is strict.

\item\label{item:strict_ineq_1}
If $n=1$, then~\eqref{eq:ineq_perim_tilde} is an equality if and only if the following hold:
\begin{enumerate}[(a)]

\item for a.e.\ $x\in\Omega\cap E$, $\Leb{1}((-\infty,x)\setminus E)=0$ vel $\Leb{1}((x,+\infty)\setminus E)=0$;

\item for a.e.\ $x\in\Omega\setminus E$, $\Leb{1}((-\infty,x) \cap E)=0$ vel $\Leb{1}((x,+\infty) \cap E)=0$.

\end{enumerate}
\end{enumerate}
\end{corollary}

\begin{proof}
We prove the two statements separately.

\smallskip

\textit{Proof of~\eqref{item:strict_ineq_big2}}.
Assume $n\ge2$. Since $\Leb{n}(E)>0$, for a given $x\in \Omega\setminus E$ the map
\begin{equation*}
y\mapsto(y-x),
\quad
\text{for } y\in E,
\end{equation*}
does not have constant orientation. Similarly, since $\Leb{n}(\R^n\setminus E)>0$, for a given $x\in\Omega\cap E$ also the map
\begin{equation*}
y\mapsto(y-x),
\quad
\text{for } y\in\R^n\setminus E,
\end{equation*}
does not have constant orientation. Hence, by \cref{res:easy_norm_ineq}, we must have
\begin{equation*}
\bigg|\int_E\frac{y-x}{|y-x|^{n+\alpha+1}}\,dy\,\bigg|
<\int_E\frac{dy}{|y-x|^{n+\alpha}},
\quad
\text{for } x\in \Omega\setminus E,
\end{equation*}
and, similarly,
\begin{equation*}
\bigg|\int_{\R^n\setminus E}\frac{y-x}{|y-x|^{n+\alpha+1}}\,dy\,\bigg|
<\int_{\R^n\setminus E}\frac{dy}{|y-x|^{n+\alpha}},
\quad
\text{for } x\in \Omega\cap E.
\end{equation*}
We thus get
\begin{align*}
\|\nabla^{\alpha} \chi_{E}\|_{L^{1}(\Omega;\, \R^{n})}
&=\mu_{n,\alpha}\int_{\Omega}\bigg|\int_{\R^n}\frac{(\chi_E(y)-\chi_E(x))\cdot(y-x)}{|y-x|^{n+\alpha+1}}\,dy\,\bigg|\,dx\\
&=\mu_{n,\alpha}\int_{\Omega\setminus E}\bigg|\int_E\frac{y-x}{|y-x|^{n+\alpha}}\,dy\,\bigg|\,dx+
\mu_{n,\alpha}\int_{\Omega\cap E}\bigg|\int_{\R^n\setminus E}\frac{y-x}{|y-x|^{n+\alpha}}\,dy\,\bigg|\,dx\\
&<\mu_{n,\alpha}\int_{\Omega\setminus E}\int_E\frac{dy\,dx}{|y-x|^{n+\alpha}}+
\mu_{n,\alpha}\int_{\Omega\cap E}\int_{\R^n\setminus E}\frac{dy\,dx}{|y-x|^{n+\alpha}}
=\mu_{n,\alpha}\tilde P_{\alpha}(E;\Omega),
\end{align*}
proving~\eqref{item:strict_ineq_big2}. 

\smallskip

\textit{Proof of~\eqref{item:strict_ineq_1}}. Assume $n=1$. We argue as in the proof of~\cite{CS18}*{Proposition~4.12}. Let
\begin{equation*}
f_E(y,x):=\frac{\chi_E(y)-\chi_E(x)}{|y-x|^{1+\alpha}},
\qquad
\text{for } x,y\in\R,\ y\ne x.
\end{equation*}
Then we can write
\begin{align*}
\tilde P_\alpha(E;\Omega)
&=\int_\Omega\int_\R|f_E(y,x)|\,dy\,dx\\
&=\int_\Omega\bigg(\int_{-\infty}^x|f_E(y,x)|\,dy+\int_x^{+\infty}|f_E(y,x)|\,dy\bigg)\,dx
\end{align*}
and
\begin{align*}
\|\nabla^\alpha\chi_E\|_{L^1(\Omega;\,\R)}
&=\mu_{1,\alpha}\int_\Omega\bigg|\int_\R f_E(y,x)\,\sgn(y-x)\,dy\,\bigg|\,dx\\
&=\mu_{1,\alpha}\int_\Omega\bigg|\int_{-\infty}^x f_E(y,x)\,dy-\int_x^{+\infty} f_E(y,x)\,dy\,\bigg|\,dx.
\end{align*}
Hence~\eqref{eq:ineq_perim_tilde} is an equality if and only if
\begin{equation}\label{eq:eq_case_1}
\bigg|\int_{-\infty}^x f_E(y,x)\,dy-\int_x^{+\infty} f_E(y,x)\,dy\,\bigg|
=
\int_{-\infty}^x|f_E(y,x)|\,dy+\int_x^{+\infty}|f_E(y,x)|\,dy
\end{equation}
for a.e.~$x\in\Omega$. Observing that
\begin{align*}
\bigg|\int_{-\infty}^x f_E(y,x)\,dy-\int_x^{+\infty} f_E(y,x)\,dy\,\bigg|
&\le
\bigg|\int_{-\infty}^x f_E(y,x)\,dy\,\bigg|
+\bigg|\int_x^{+\infty} f_E(y,x)\,dy\,\bigg|\\
&\le
\int_{-\infty}^x|f_E(y,x)|\,dy+\int_x^{+\infty}|f_E(y,x)|\,dy
\end{align*}
for a.e.~$x\in\Omega$, we deduce that~\eqref{eq:ineq_perim_tilde} is an equality if and only if 
\begin{align}
\bigg|\int_{-\infty}^x f_E(y,x)\,dy-\int_x^{+\infty} f_E(y,x)\,dy\,\bigg|
&=
\bigg|\int_{-\infty}^x f_E(y,x)\,dy\,\bigg|
+\bigg|\int_x^{+\infty} f_E(y,x)\,dy\,\bigg|
\label{eq:eq_case_1_bis_A}\\
&=
\int_{-\infty}^x|f_E(y,x)|\,dy+\int_x^{+\infty}|f_E(y,x)|\,dy
\label{eq:eq_case_1_bis_B}
\end{align}
for a.e.~$x\in\Omega$. Now, on the one hand, squaring both sides of~\eqref{eq:eq_case_1_bis_A} and simplifying, we get that~\eqref{eq:ineq_perim_tilde} is an equality if and only if
\begin{equation}\label{eq:eq_case_1_product}
\bigg(\int_{-\infty}^x f_E(y,x) \,dy\bigg)\bigg(\int_x^{+\infty}f_E(y,x)\,dy\bigg)=0
\end{equation}
for a.e.~$x\in\Omega$. On the other hand, we can rewrite~\eqref{eq:eq_case_1_bis_B} as
\begin{align*}
0\le
\int_{-\infty}^x|f_E(y,x)|\,dy
-\bigg|\int_{-\infty}^x f_E(y,x)\,dy\,\bigg|
=
\bigg|\int_x^{+\infty} f_E(y,x)\,dy\,\bigg|
-\int_x^{+\infty}|f_E(y,x)|\,dy\le0
\end{align*}
for a.e.~$x\in\Omega$, so that we must have
\begin{equation*}
\bigg|\int_{-\infty}^x f_E(y,x)\,dy\,\bigg|
=
\int_{-\infty}^x|f_E(y,x)|\,dy
\end{equation*}
and
\begin{equation*}
\bigg|\int_x^{+\infty} f_E(y,x)\,dy\,\bigg|
=\int_x^{+\infty}|f_E(y,x)|\,dy
\end{equation*}
for a.e.~$x\in\Omega$. Hence~\eqref{eq:eq_case_1_product} can be equivalently rewritten as
\begin{equation}\label{eq:eq_case_1_product_bis}
\bigg(\int_{-\infty}^x |f_E(y,x)| \,dy\bigg)\bigg(\int_x^{+\infty}|f_E(y,x)|\,dy\bigg)=0
\end{equation}
for a.e.~$x\in\Omega$. Thus~\eqref{eq:ineq_perim_tilde} is an equality if and only if at least one of the two integrals in the left-hand side of~\eqref{eq:eq_case_1_product_bis} is zero, and the reader can check that~\eqref{item:strict_ineq_1} readily follows.
\end{proof}

\begin{remark}[Half-lines in \cref{res:strict_per_tilde}\eqref{item:strict_ineq_1}]
In the case $n=1$, it is worth to stress that~\eqref{eq:ineq_perim_tilde} is always an equality when the set $E\subset\R$ is (equivalent to) an half-line, i.e.,
\begin{equation*}
\|\nabla^{\alpha} \chi_{(a,+\infty)}\|_{L^{1}(\Omega;\,\R)} 
=
\mu_{1, \alpha}\tilde P_{\alpha}((a,+\infty);\Omega)
\end{equation*}
for any $\alpha\in(0,1)$, any $a\in\R$ and any open set $\Omega\subset\R$ such that \mbox{$\tilde P_{\alpha}((a,+\infty);\Omega)<+\infty$}.  However, the equality cases in~\eqref{eq:ineq_perim_tilde} are considerably richer. Indeed, on the one side,
\begin{equation*}
\|\nabla^{\alpha} \chi_{(-5,-4)\cup(-1,+\infty)}\|_{L^{1}((0,1);\,\R)} 
=
\mu_{1, \alpha}\tilde P_{\alpha}((-5,-4)\cup(-1,+\infty);(0,1))
\end{equation*}
and, on the other side, 
\begin{equation*}
\|\nabla^{\alpha} \chi_{(-5,-4)\cup(0,+\infty)}\|_{L^{1}((-1,1);\,\R)} 
<
\mu_{1, \alpha}\tilde P_{\alpha}((-5,-4)\cup(0,+\infty);(-1,1))
\end{equation*}
for any $\alpha\in(0,1)$. We leave the simple computations to the interested reader.
\end{remark}

\section{Estimates and representation formulas for the fractional \texorpdfstring{$\alpha$}{alpha}-gradient}
\label{sec:estimates_representations}

\subsection{Integrability properties of the fractional \texorpdfstring{$\alpha$}{alpha}-gradient}

We begin with the following technical local estimate on the $W^{\alpha,1}$-seminorm of a function in $BV_{\loc}$.

\begin{lemma}\label{result:simple_est_Sobolev_seminorm_BV_loc}
Let $\alpha\in(0,1)$ and let $f\in BV_{\loc}(\R^n)$. Then $f\in W^{\alpha,1}_{\loc}(\R^n)$ with
\begin{equation}\label{eq:simple_est_Sobolev_seminorm_BV_loc}
[f]_{W^{\alpha,1}(B_R)}
\le\frac{n\omega_n(2R)^{1-\alpha}}{1-\alpha}\,|Df|(B_{3R})
\end{equation}
for all $R>0$.
\end{lemma}

\begin{proof}
Fix $R>0$ and let $f\in BV_{\loc}(\R^n)$ be such that $f\in C^1(B_{3R})$. We can estimate
\begin{align*}
[f]_{W^{\alpha,1}(B_R)}
&=\int_{B_R}\int_{B_R}\frac{|f(y)-f(x)|}{|y-x|^{n+\alpha}}\,dy\,dx\\
&=\int_{B_R}\int_{B_R\cap\set*{|y-x|<2R}}\frac{|f(y)-f(x)|}{|y-x|^{n+\alpha}}\,dy\,dx\\
&\le\int_{\set*{|h|<2R}}\frac{1}{|h|^{n+\alpha}}\int_{B_R}|f(x+h)-f(x)|\,dx\,dh.
\end{align*}
Since
\begin{align*}
\int_{B_R}|f(x+h)-f(x)|\,dx
&\le\int_{B_R}\int_0^1|\nabla f(x+th)\cdot h|\,dt\,dx\\
&\le|h|\int_0^1\int_{B_R}|\nabla f(x+th)|\,dx\,dt\\
&\le|h|\int_{B_{R+|h|}}|\nabla f(z)|\,dz
\end{align*}
for all $h\in\R^n$, we have
\begin{align*}
[f]_{W^{\alpha,1}(B_R)}
&\le\int_{\set*{|h|<2R}}\frac{1}{|h|^{n+\alpha-1}}\int_{B_{R+|h|}}|\nabla f(z)|\,dz\,dh\\
&\le\int_{\set*{|h|<2R}}\frac{|Df|(B_{3R})}{|h|^{n+\alpha-1}}\,dh\\
&=\frac{n\omega_n(2R)^{1-\alpha}}{1-\alpha}\,|Df|(B_{3R})
\end{align*}
proving~\eqref{eq:simple_est_Sobolev_seminorm_BV_loc} for all $f\in BV_{\loc}(\R^n)\cap C^1(B_{3R})$. Now fix $R>0$ and let $f\in BV_{\loc}(\R^n)$. 
By~\cite{EG15}*{Theorem~5.3}, there exists $(f_k)_{k\in\N}\subset BV(B_{3R})\cap C^\infty(B_{3R})$ such that $|Df_k|(B_{3R})\to|Df|(B_{3R})$ and $f_k\to f$ a.e.\ in~$B_{3R}$ as $k\to+\infty$. The conclusion thus follows by a simple application of Fatou's Lemma.
\end{proof}

In the following result, we collect several local integrability estimates involving the fractional $\alpha$-gradient of a function satisfying various regularity assumptions.

\begin{proposition}\label{lem:nabla_alpha_f_L1_and_L1_loc}
The following statements hold.
\begin{enumerate}[(i)]
\item\label{item:lem:nabla_alpha_f_loc_1} If $f\in BV(\R^n)$, then $f\in BV^\alpha(\R^n)$ for all $\alpha\in(0,1)$ with $D^\alpha f=\nabla^\alpha f\Leb{n}$ and
\begin{equation}\label{eq:weak_frac_grad_repr}
\nabla^{\alpha} f = I_{1-\alpha}D f
\quad
\text{a.e.\ in $\R^{n}$.}
\end{equation}
In addition, for any bounded open set $U\subset\R^n$, we have
\begin{equation} \label{eq:weak_frac_grad_loc_estimate} 
\|\nabla^{\alpha} f\|_{L^{1}(U;\, \R^{n})} 
\le C_{n, \alpha, U} \, |D f|(\R^{n})
\end{equation}
for all $\alpha\in(0,1)$, where $C_{n, \alpha, U}$ is as in~\eqref{eq:constant_estimate}. Finally, given an open set $A\subset\R^n$, we have
\begin{equation}\label{eq:Davila_estimate}
\|\nabla^\alpha f\|_{L^1(A;\,\R^n)}
\le\frac{n\omega_n\,\mu_{n,\alpha}}{n+\alpha-1}\left(\frac{|Df|(\closure[-.5]{A_r})}{1-\alpha}\,r^{1-\alpha}+\frac{n+2\alpha-1}{\alpha}\,\|f\|_{L^1(\R^n)}\,r^{-\alpha}\right)
\end{equation} 
for all $r>0$ and $\alpha\in(0,1)$, where $A_r:=\set*{x\in\R^n : \dist(x,A)<r}$.
In particular, we have
\begin{equation} \label{eq:Davila_estimate_opt}
\|\nabla^\alpha f\|_{L^1(\R^n;\,\R^n)}
\le\frac{n\omega_n\,\mu_{n,\alpha} (n + 2 \alpha - 1)^{1 - \alpha}}{\alpha (1 - \alpha)(n+\alpha-1)}\,
\|f\|_{L^{1}(\R^{n})}^{1 - \alpha}\,
[f]_{BV(\R^n)}^{\alpha}.
\end{equation}

\item\label{item:lem:nabla_alpha_f_loc_2} If $f \in L^{\infty}(\R^{n})\cap W^{\alpha, 1}_{\loc}(\R^{n})$, then the weak fractional $\alpha$-gradient $D^{\alpha} f \in \M_{\loc}(\R^n; \R^n)$ exists and satisfies $D^\alpha f=\nabla^\alpha f\Leb{n}$ with $\nabla^{\alpha} f \in L^{1}_{\loc}(\R^{n}; \R^{n})$ and
\begin{equation}\label{eq:nabla_alpha_f_L1_loc}
\begin{split}
\| \nabla^{\alpha} f \|_{L^{1}(B_{R};\, \R^{n})} 
&\le \mu_{n, \alpha} \int_{B_{R}} \int_{\R^{n}} \frac{|f(x) - f(y)|}{|x - y|^{n + \alpha}} \, dx \, dy\\
&\le \mu_{n, \alpha} \left ( [f]_{W^{\alpha, 1}(B_{R})} 
+ P_{\alpha}(B_{R})\, \|f\|_{L^{\infty}(\R^{n})} \right )
\end{split}
\end{equation}
for all $R>0$ and $\alpha\in(0,1)$.

\item\label{item:lem:nabla_alpha_f_loc_3} If $f \in L^{\infty}(\R^{n})\cap BV_{\loc}(\R^{n})$, then the weak fractional $\alpha$-gradient $D^{\alpha} f \in \M_{\rm loc}(\R^n; \R^n)$ exists and satisfies $D^\alpha f=\nabla^\alpha f\Leb{n}$ with $\nabla^{\alpha} f \in L^{1}_{\loc}(\R^{n}; \R^{n})$ and
\begin{equation}\label{eq:nabla_alpha_f_L1_loc_BV_loc}
\| \nabla^{\alpha} f \|_{L^{1}(B_{R};\, \R^{n})} 
\le\mu_{n, \alpha}\left(\frac{n\omega_n(2R)^{1-\alpha}}{1-\alpha} \, |Df|(B_{3R}) 
+ \frac{2^{\alpha+1}(n\omega_n)^2R^{n-\alpha}}{\alpha\,\Gamma(1-\alpha)^{-1}}\, \|f\|_{L^{\infty}(\R^{n})}\right). 
\end{equation} 
for all $R>0$ and $\alpha\in(0,1)$.
\end{enumerate}
\end{proposition}

\begin{proof}
We prove the three statements separately.

\smallskip

\textit{Proof of~(\ref{item:lem:nabla_alpha_f_loc_1})}. Thanks to~\cite{CS18}*{Theorem~3.18}, we just need to prove~\eqref{eq:weak_frac_grad_loc_estimate} and~\eqref{eq:Davila_estimate}. 

We prove~\eqref{eq:weak_frac_grad_loc_estimate}. By~\eqref{eq:weak_frac_grad_repr}, by Tonelli's Theorem and by~\cite{CS18}*{Lemma~2.4}, we get
\begin{align*} 
\int_{U} |\nabla^{\alpha} f|\,dx 
\le\int_{U} I_{1-\alpha}|Df|\,dx
\le C_{n,\alpha,U}\,|D f|(\R^{n}), 
\end{align*}
where $C_{n, \alpha, U}$ is defined as in~\eqref{eq:constant_estimate}.

We now prove~\eqref{eq:Davila_estimate} in two steps.

\smallskip

\textit{Proof of~\eqref{eq:Davila_estimate}, Step~1}. Assume $f\in C^\infty_c(\R^n)$ and fix $r>0$. We have
\begin{align*}
\int_A|\nabla^\alpha f|\,dx
&=\int_A|I_{1-\alpha}\nabla f|\,dx\\
&\le\frac{\mu_{n,\alpha}}{n+\alpha-1}\left(\int_A\int_{\{|h|\le r\}}\frac{|\nabla f(x+h)|}{|h|^{n+\alpha-1}}\,dh\,dx
+\int_A\abs*{\int_{\{|h|>r\}}\frac{\nabla f(x+h)}{|h|^{n+\alpha-1}}\,dh\,}\,dx\right).
\end{align*}
We estimate the two double integrals appearing in the right-hand side separately. By Tonelli's Theorem, we have
\begin{align*}
\int_A\int_{\{|h|\le r\}}\frac{|\nabla f(x+h)|}{|h|^{n+\alpha-1}}\,dh\,dx
&=\int_{\{|h|\le r\}}\int_A|\nabla f(x+h)|\,dx\,\frac{dh}{|h|^{n+\alpha-1}}\\
&\le\|\nabla f\|_{L^1(\closure[-.5]{A_r};\,\R^n)}\int_{\{|h|\le r\}}\frac{dh}{|h|^{n+\alpha-1}}\\
&=n\omega_n\,\frac{r^{1-\alpha}}{1-\alpha}\,\|\nabla f\|_{L^1(\closure[-.5]{A_r};\,\R^n)}.
\end{align*}
Concerning the second double integral, integrating by parts we get
\begin{align*}
\int_{\{|h|>r\}}\frac{\nabla f(x+h)}{|h|^{n+\alpha-1}}\,dh
&=(n+\alpha-1)\int_{\{|h|>r\}}\frac{h f(x+h)}{|h|^{n+\alpha+1}}\,dh\\
&\quad -\int_{\{|h|=r\}}\frac{h}{|h|}\frac{f(x+h)}{|h|^{n+\alpha-1}}\,d\Haus{n-1}(h)
\end{align*}
for all $x\in A$. Hence, we can estimate
\begin{align*}
\int_A\abs*{\int_{\{|h|>r\}}\frac{\nabla f(x+h)}{|h|^{n+\alpha-1}}\,dh\,}\,dx
&\le(n+\alpha-1)\int_A\int_{\{|h|>r\}}\frac{|f(x+h)|}{|h|^{n+\alpha}}\,dh\,dx\\
&\quad+\int_A\int_{\{|h|=r\}}\frac{|f(x+h)|}{|h|^{n+\alpha-1}}\,d\Haus{n-1}(h)\,dx\\
&\le n\omega_n\|f\|_{L^1(\R^n)}\,r^{-\alpha}\left(\frac{n+\alpha-1}{\alpha}
+1\right)\\
&=n\omega_n\left(\frac{n+2\alpha-1}{\alpha}\right)\|f\|_{L^1(\R^n)}\,r^{-\alpha}.
\end{align*}
Thus~\eqref{eq:Davila_estimate} follows for all $f\in C^\infty_c(\R^n)$ and~$r>0$. 

\smallskip

\textit{Proof of~\eqref{eq:Davila_estimate}, Step~2}. Let $f\in BV(\R^n)$ and fix $r>0$. Combining~\cite{EG15}*{Theorem~5.3} with a standard cut-off approximation argument, we find $(f_k)_{k\in\N}\subset C^\infty_c(\R^n)$ such that $f_k\to f$ in~$L^1(\R^n)$ and $|Df_k|(\R^n)\to|Df|(\R^n)$ as $k\to+\infty$. By Step~1, we have that
\begin{equation}\label{eq:Davila_estimate_smooth}
\|\nabla^\alpha f_k\|_{L^1(A;\,\R^n)}
\le\frac{n\omega_n\,\mu_{n,\alpha}}{n+\alpha-1}\left(\frac{|Df_k|(\closure[-.5]{A_r})}{1-\alpha}\,r^{1-\alpha}+\frac{n+2\alpha-1}{\alpha}\,\|f_k\|_{L^1(\R^n)}\,r^{-\alpha}\right)
\end{equation}
for all $k\in\N$. We claim that
\begin{equation}\label{eq:claim_weak_conv_BV}
(\nabla^\alpha f_k)\,\Leb{n}\weakto(\nabla^\alpha f)\,\Leb{n}
\quad
\text{as $k\to+\infty$}.
\end{equation}
Indeed, if $\phi\in\Lip_c(\R^n;\R^n)$, then $\div^\alpha\phi\in L^\infty(\R^n)$ by~\eqref{eq:frac_div_repr_Lip_estimate} and thus
\begin{align*}
\abs*{\int_{\R^n}\phi\cdot\nabla^\alpha f_k\,dx-\int_{\R^n}\phi\cdot\nabla^\alpha f\,dx}
&=\abs*{\int_{\R^n}f_k\,\div^\alpha\phi\,dx-\int_{\R^n}f\,\div^\alpha\phi\,dx}\\
&\le\|\div^\alpha\phi\|_{L^\infty(\R^n;\,\R^n)}\,\|f_k-f\|_{L^1(\R^n)}
\end{align*}
for all $k\in\N$, so that
\begin{equation*}
\lim_{k\to+\infty}\int_{\R^n}\phi\cdot\nabla^\alpha f_k\,dx
=\int_{\R^n}\phi\cdot\nabla^\alpha f\,dx.
\end{equation*}
Now fix $\phi\in C_c^0(\R^n;\R^n)$. Let $U\subset\R^n$ be a bounded open set such that $\supp\phi\subset U$. For each $\eps>0$ sufficiently small, pick $\psi_\eps\in\Lip_c(\R^n;\R^n)$ such that $\|\phi-\psi_\eps\|_{L^\infty(\R^n;\,\R^n)}<\eps$ and $\supp\psi_\eps\subset U$. Then
\begin{align*}
\bigg|\int_{\R^n}\phi\cdot\nabla^\alpha f_k\,dx&-\int_{\R^n}\phi\cdot\nabla^\alpha f\,dx\bigg|
\le\abs*{\int_{\R^n}\psi_\eps\cdot\nabla^\alpha f_k\,dx-\int_{\R^n}\psi_\eps\cdot\nabla^\alpha f\,dx}\\
&\quad+\|\psi_\eps-\phi\|_{L^\infty(\R^n;\,\R^n)}
\left(
\|\nabla^\alpha f_k\|_{L^1(U;\,\R^n)}
+\|\nabla^\alpha f\|_{L^1(U;\,\R^n)}
\right)\\
&\le\abs*{\int_{\R^n}\psi_\eps\cdot\nabla^\alpha f_k\,dx-\int_{\R^n}\psi_\eps\cdot\nabla^\alpha f\,dx}\\
&\quad+\eps\,C_{n,\alpha,U}
\big(|Df_k|(\R^n)+|Df|(\R^n)\big),
\end{align*}
so that
\begin{equation*}
\lim_{k\to+\infty}\abs*{\int_{\R^n}\phi\cdot\nabla^\alpha f_k\,dx-\int_{\R^n}\phi\cdot\nabla^\alpha f\,dx}
\le 2\eps\, C_{n,\alpha,U}|Df|(\R^n).
\end{equation*}
Thus,~\eqref{eq:claim_weak_conv_BV} follows passing to the limit as $\eps\to0^+$.
Thanks to~\eqref{eq:claim_weak_conv_BV}, by~\cite{M12}*{Proposition~4.29} we get that 
\begin{equation*}
\|\nabla^\alpha f\|_{L^1(A;\,\R^n)}
\le\liminf_{k\to+\infty}\|\nabla^\alpha f_k\|_{L^1(A;\,\R^n)}.
\end{equation*}
Since
\begin{equation*}
|Df|(U)\le\liminf_{k\to+\infty}|Df_k|(U)
\end{equation*}
for any open set $U\subset\R^n$ by~\cite{EG15}*{Theorem~5.2}, we can estimate
\begin{align*}
\limsup_{k\to+\infty}|Df_k|(\closure{A_r})
&\le\lim_{k\to+\infty}|Df_k|(\R^n)
-\liminf_{k\to+\infty}|Df_k|(\R^n\setminus A_r)\\
&\le|Df|(\R^n)-|Df|(\R^n\setminus A_r)\\
&=|Df|(\closure{A_r}).
\end{align*}
Thus,~\eqref{eq:Davila_estimate} follows taking limits as $k\to+\infty$ in~\eqref{eq:Davila_estimate_smooth}. Finally, \eqref{eq:Davila_estimate_opt} is easily deduced by optimising the right-hand side of~\eqref{eq:Davila_estimate} in the case $A=\R^n$ with respect to $r > 0$.

\smallskip

\textit{Proof of~(\ref{item:lem:nabla_alpha_f_loc_2})}. Assume $f \in L^{\infty}(\R^{n})\cap W^{\alpha, 1}_{\loc}(\R^{n})$. Given $R>0$,  we can estimate
\begin{align*}
\int_{B_{R}} |\nabla^{\alpha} f(x)| \, dx 
& \le \mu_{n, \alpha} \int_{B_{R}} \int_{\R^{n}} \frac{|f(x) - f(y)|}{|x - y|^{n + \alpha}} \, dx \, dy \\
& = \mu_{n, \alpha} \int_{B_{R}} \int_{B_{R}} \frac{|f(x) - f(y)|}{|x - y|^{n + \alpha}} \, dx \, dy + \mu_{n, \alpha} \int_{B_{R}} \int_{\R^{n} \setminus B_{R}} \frac{|f(x) - f(y)|}{|x - y|^{n + \alpha}} \, dx \, dy \\
& \le \mu_{n, \alpha} [f]_{W^{\alpha, 1}(B_{R})} + 2 \mu_{n, \alpha} \| f\|_{L^{\infty}(\R^{n})} \int_{B_{R}} \int_{\R^{n} \setminus B_{R}} \frac{1}{|x - y|^{n + \alpha}} \, dx \, dy\\
&=\mu_{n, \alpha} [f]_{W^{\alpha, 1}(B_{R})} + \mu_{n, \alpha} \| f\|_{L^{\infty}(\R^{n})}P_\alpha(B_R)
\end{align*}
and~\eqref{eq:nabla_alpha_f_L1_loc} follows. To prove that $D^\alpha f=\nabla^\alpha f\Leb{n}$, we argue as in the proof of~\cite{CS18}*{Proposition~4.8}. Let $\phi\in\Lip_c(\R^n;\R^n)$. Since $f\in L^\infty(\R^n)$, we have 
\begin{equation*}
x\mapsto|f(x)|\int_{\R^n}\frac{|\phi(y)-\phi(x)|}{|y-x|^{n+\alpha}}\,dy\in L^1(\R^n).
\end{equation*}
Hence, by the definition of~$\div^\alpha$ on $\Lip_c$-regular vector fields (see~\cite{CS18}*{Section~2.2}) and by Lebesgue's Dominated Convergence Theorem, we have
\begin{equation*}
\int_{\R^n}f\,\div^\alpha\phi\,dx
=\lim_{\eps\to0^+}\int_{\R^n}f(x)\int_{\set*{|y-x|>\eps}}\frac{(y-x)\cdot\phi(y)}{|y-x|^{n+\alpha+1}}\,dy\,dx.
\end{equation*}
Since
\begin{equation*}
\begin{split}
\int_{\R^n}\int_{\set*{|y-x|>\eps}}\frac{|f(x)|\,|\phi(y)|}{|y-x|^{n+\alpha}}\,dy\,dx
&\le\|f\|_{L^\infty(\R^n)}\int_{\R^n}|\phi(y)|\int_{\set*{|y-x|>\eps}}|y-x|^{-n-\alpha}\,dx\,dy\\
&\le\frac{n\omega_n}{\alpha\eps^\alpha}\|f\|_{L^\infty(\R^n)}\|\phi\|_{L^1(\R^n;\,\R^n)}
\end{split}
\end{equation*}
for all $\eps>0$, by Fubini's Theorem we can compute
\begin{equation*}
\begin{split}
\int_{\R^n}f(x)\int_{\set*{|y-x|>\eps}}\frac{(y-x)\cdot\phi(y)}{|y-x|^{n+\alpha+1}}\,dy\,dx
&=-\int_{\R^n}\phi(y)\int_{\set*{|x-y|>\eps}}\frac{(x-y)\,f(x)}{|x-y|^{n+\alpha+1}}\,dx\,dy\\
&=-\int_{\R^n}\phi(y)\int_{\set*{|x-y|>\eps}}\frac{(x-y)\,(f(x)-f(y))}{|x-y|^{n+\alpha+1}}\,dx\,dy.
\end{split}
\end{equation*}
Since 
\begin{align*}
|\phi(y)|\,\abs*{\int_{\set*{|x-y|>\eps}}\frac{(x-y)\,(f(x)-f(y))}{|x-y|^{n+\alpha+1}}\,dx}
\le|\phi(y)|\int_{\R^n}\frac{|f(x)-f(y)|}{|x-y|^{n+\alpha}}\,dx
\end{align*}
for all $y\in\R^n$ and $\eps>0$, and 
\begin{equation*}
y\mapsto\int_{\R^n}\frac{|f(x)-f(y)|}{|x-y|^{n+\alpha}}\,dx\in L^1_{\loc}(\R^n)
\end{equation*}
by~\eqref{eq:nabla_alpha_f_L1_loc}, again by Lebesgue's Dominated Convergence Theorem we conclude that
\begin{equation*}
\begin{split}
\int_{\R^n}f(x)\,\div^\alpha\phi(x) \,dx
&=-\lim_{\eps\to0}\int_{\R^n}\phi(y)\int_{\set*{|x-y|>\eps}}\frac{(x-y)\,(f(x)-f(y))}{|x-y|^{n+\alpha+1}}\,dx\,dy\\
&=-\int_{\R^n}\phi(y)\lim_{\eps\to0}\int_{\set*{|x-y|>\eps}}\frac{(x-y)\,(f(x)-f(y))}{|x-y|^{n+\alpha+1}}\,dx\,dy\\
&=-\int_{\R^n}\phi(y) \cdot\nabla^\alpha f(y)\,dy
\end{split}
\end{equation*}
for all $\phi\in\Lip_c(\R^n;\R^n)$. Thus $D^\alpha f\in\mathscr{M}_{\loc}(\R^n;\R^n)$ is well defined and $D^\alpha f=\nabla^\alpha f\Leb{n}$.

\smallskip

\textit{Proof of~(\ref{item:lem:nabla_alpha_f_loc_3})}. Assume $f\in L^{\infty}(\R^{n})\cap BV_{\loc}(\R^{n})$. By \cref{result:simple_est_Sobolev_seminorm_BV_loc}, we know that $f \in L^{\infty}(\R^{n})\cap W^{\alpha, 1}_{\loc}(\R^{n})$ for all $\alpha\in(0,1)$, so that $D^{\alpha} f \in \M_{\rm loc}(\R^{n}; \R^{n})$ exists by~\eqref{item:lem:nabla_alpha_f_loc_2}. Hence, inserting~\eqref{eq:simple_est_Sobolev_seminorm_BV_loc} in~\eqref{eq:nabla_alpha_f_L1_loc}, we find
\begin{equation*}
\| \nabla^{\alpha} f \|_{L^{1}(B_{R};\, \R^{n})} 
\le\mu_{n, \alpha}\left( \frac{n\omega_n(2R)^{1-\alpha}}{1-\alpha} \, |Df|(B_{3R}) 
+ P_{\alpha}(B_1)\,R^{n-\alpha}\, \|f\|_{L^{\infty}(\R^{n})}\right). 
\end{equation*} 
Since for all $x\in B_1$ we have
\begin{equation*}
\int_{\R^n\setminus B_1}\frac{dy}{|y-x|^{n+\alpha}}
=\int_{\R^n\setminus B_1(-x)}\frac{dz}{|z|^{n+\alpha}}
\le\int_{\R^n\setminus B_{1-|x|}}\frac{dz}{|z|^{n+\alpha}}
=\frac{n\omega_n}{\alpha(1-|x|)^\alpha},
\end{equation*}
being~$\Gamma$ log-convex on $(0,+\infty)$ (see~\cite{A64}), we can estimate
\begin{equation*}
\begin{split}
P_\alpha(B_1)
&=2\int_{B_1}\int_{\R^n\setminus B_1}\frac{dy\,dx}{|y-x|^{n+\alpha}}
\le\frac{2n\omega_n}{\alpha}\int_{B_1}\frac{dx}{(1-|x|)^\alpha}\\
&=\frac{2(n\omega_n)^2}{\alpha}\int_0^1\frac{t^{n-1}}{(1-t)^\alpha}\,dt
=\frac{2(n\omega_n)^2}{\alpha}\,\frac{\Gamma(n)\,\Gamma(1-\alpha)}{\Gamma(n+1-\alpha)}\\
&\le\frac{2(n\omega_n)^2}{\alpha}\,\frac{(n+1)^\alpha}{n}\Gamma(1-\alpha) \\
&\le \frac{2^{\alpha+1}(n\omega_n)^2}{\alpha}\,\Gamma(1-\alpha),
\end{split}
\end{equation*}
so that
\begin{equation*}
\| \nabla^{\alpha} f \|_{L^{1}(B_{R};\, \R^{n})} 
\le\mu_{n, \alpha}\left( \frac{n\omega_n(2R)^{1-\alpha}}{1-\alpha} \, |Df|_{BV(B_{3R})} 
+ \frac{2^{\alpha+1}(n\omega_n)^2R^{n-\alpha}}{\alpha\,\Gamma(1-\alpha)^{-1}}\, \|f\|_{L^{\infty}(\R^{n})}\right), 
\end{equation*} 
proving~\eqref{eq:nabla_alpha_f_L1_loc_BV_loc}.
\end{proof}

Note that \cref{lem:nabla_alpha_f_L1_and_L1_loc}\eqref{item:lem:nabla_alpha_f_loc_1}, in particular,  applies to any $f\in W^{1,1}(\R^n)$. In the following result, we prove that a similar result holds also for any $f\in W^{1,p}(\R^n)$ with $p\in(1,+\infty)$.

\begin{proposition}[$W^{1,p}(\R^n)\subset S^{\alpha,p}(\R^n)$ for $p\in(1,+\infty)$]\label{res:nabla_alpha_Lp}
Let $\alpha\in(0,1)$ and $p\in(1,+\infty)$. If $f\in W^{1,p}(\R^n)$, then $f\in S^{\alpha,p}(\R^n)$ with
\begin{equation}\label{eq:Davila_estimate_p}
\|\nabla^\alpha_w f\|_{L^p(A;\,\R^n)}
\le\frac{n\omega_n\mu_{n,\alpha}}{n+\alpha-1}\left(\frac{\|\nabla_w f\|_{L^p(\closure[-.5]{A_r};\,\R^n)}}{1-\alpha}\,r^{1-\alpha}+\frac{n+2\alpha-1}{\alpha}\,\|f\|_{L^p(\R^n)}\,r^{-\alpha}\right)
\end{equation} 
for any $r>0$ and any open set $A\subset\R^n$, where $A_r:=\set*{x\in\R^n : \dist(x,A)<r}$. In particular, we have
\begin{equation}\label{eq:W_1_p_inclusion_S_alpha_p}
\|\nabla^\alpha_w f\|_{L^p(\R^n;\,\R^n)}
\le\frac{(n+2\alpha-1)^{1-\alpha}}{n+\alpha-1}\frac{n\omega_n\mu_{n,\alpha}}{\alpha(1-\alpha)}\,\|\nabla_w f\|_{L^p(\R^n;\,\R^n)}^\alpha\| f\|_{L^p(\R^n)}^{1-\alpha}.
\end{equation}
In addition, if $p\in\big(1,\frac{n}{1-\alpha}\big)$ and $q=\frac{np}{n-(1-\alpha)p}$, then
\begin{equation}\label{eq:nabla_alpha_Lp_repres}
\nabla^\alpha_w f = I_{1-\alpha}\nabla_w f
\quad
\text{a.e.\ in~$\R^n$}
\end{equation}
and $\nabla^\alpha_w f\in L^q(\R^n;\R^n)$.
\end{proposition}

\begin{proof}
We argue as in the proof of \cref{lem:nabla_alpha_f_L1_and_L1_loc}\eqref{item:lem:nabla_alpha_f_loc_1}. 

\smallskip

\textit{Proof of~\eqref{eq:Davila_estimate_p}}.
The proof of~\eqref{eq:Davila_estimate_p} for all $f\in C^\infty_c(\R^n)$ is very similar to that of~\eqref{eq:Davila_estimate} and is thus left to the reader. 
Now let $f\in W^{1,p}(\R^n)$ and fix an open set $A\subset\R^n$ and~$r>0$. Combining~\cite{EG15}*{Theorem~4.2} with a standard cut-off approximation argument, we find $(f_k)_{k\in\N}\subset C^\infty_c(\R^n)$ such that $f_k\to f$ in~$W^{1,p}(\R^n)$ as $k\to+\infty$. 
We thus have that
\begin{equation}\label{eq:Davila_estimate_smooth_p}
\|\nabla^\alpha f_k\|_{L^p(A;\,\R^n)}
\le\frac{n\omega_n\,\mu_{n,\alpha}}{n+\alpha-1}\left(\frac{\|\nabla f_k\|_{L^p(\closure[-.5]{A_r};\,\R^n)}}{1-\alpha}\,r^{1-\alpha}+\frac{n+2\alpha-1}{\alpha}\,\|f_k\|_{L^p(\R^n)}\,r^{-\alpha}\right)
\end{equation}
for all $k\in\N$. Hence, choosing $A=\R^n$, we get that the sequence $(\nabla^\alpha f_k)_{k\in\N}$ is uniformly bounded in~$L^p(\R^n;\R^n)$. Up to pass to a subsequence (which we do not relabel for simplicity), there exists $g\in L^p(\R^n;\R^n)$ such that $\nabla^\alpha f_k\weakto g$ in $L^p(\R^n;\R^n)$ as $k\to+\infty$. Given $\phi\in C^\infty_c(\R^n;\R^n)$, we have
\begin{equation*}
\int_{\R^n} f_k\,\div^\alpha\phi\,dx
=-\int_{\R^n}\phi\cdot\nabla^\alpha f_k\,dx
\end{equation*}
for all $k\in\N$. Passing to the limit as $k\to+\infty$, by~\cref{prop:frac_div_repr} we get that
\begin{equation*}
\int_{\R^n} f\,\div^\alpha\phi\,dx
=-\int_{\R^n}\phi\cdot g\,dx
\end{equation*}
for any $\phi\in C^\infty_c(\R^n;\R^n)$, so that $g=\nabla^\alpha_w f$ and hence $f\in S^{\alpha,p}(\R^n)$ according to~\cite{CS18}*{Definition~3.19}. We thus have that
\begin{equation*}
\|\nabla^\alpha_w f\|_{L^p(A;\,\R^n)}\le\liminf_{k\to+\infty}\|\nabla^\alpha f_k\|_{L^p(A;\,\R^n)}
\end{equation*} 
for any open set $A\subset\R^n$, since
\begin{equation*}
\int_{\R^n}\phi\cdot\nabla^\alpha_w f\,dx
=\lim_{k\to+\infty}\int_{\R^n}\phi\cdot\nabla^\alpha f_k\,dx
\le\|\phi\|_{L^{\frac{p}{p-1}}(A;\,\R^n)}\liminf_{k\to+\infty}\|\nabla^\alpha f_k\|_{L^p(A;\,\R^n)}
\end{equation*}
for all $\phi\in C^\infty_c(A;\R^n)$. Therefore, \eqref{eq:Davila_estimate_p} follows by taking limits as~$k\to+\infty$ in~\eqref{eq:Davila_estimate_smooth_p}.

\smallskip

\textit{Proof of~\eqref{eq:W_1_p_inclusion_S_alpha_p}}. Inequality~\eqref{eq:W_1_p_inclusion_S_alpha_p} follows by applying~\eqref{eq:Davila_estimate_p} with $A=\R^n$ and minimising the right-hand side with respect to~$r>0$.

\smallskip

\textit{Proof of~\eqref{eq:nabla_alpha_Lp_repres}}.
Now assume $p\in\big(1,\frac{n}{1-\alpha}\big)$ and let $q=\frac{np}{n-(1-\alpha)p}$. Let $\phi\in C^\infty_c(\R^n;\R^n)$ be fixed. Recalling inequality~\eqref{eq:Riesz_potential_boundedness}, since $\phi\in L^{\frac{q}{q-1}}(\R^n;\R^n)$ we have that
\begin{equation*}
|\phi|\,I_{1-\alpha}|f|\in L^1(\R^n),
\quad
|\phi|\,I_{1-\alpha}|\nabla_w f|\in L^1(\R^n).
\end{equation*}
In particular, Fubini's Theorem implies that
\begin{equation*}
f\,I_{1-\alpha}\phi\in L^1(\R^n; \R^n),
\quad
I_{1-\alpha}\phi\cdot\nabla_w f\in L^1(\R^n).
\end{equation*}
Since $\div^{\alpha} \phi \in L^{\frac{p}{p-1}}(\R^n)$ by \cref{prop:frac_div_repr}, we also get that 
\begin{equation*}
f\,\div I_{1-\alpha}\phi=f\,\div^\alpha\phi\in L^1(\R^n).
\end{equation*}
Therefore, observing that $I_{1-\alpha}\phi\in\Lip_b(\R^n;\R^n)$ because $\nabla I_{1 - \alpha} \phi = \nabla^{\alpha} \phi\in L^{\infty}(\R^{n}; \R^{n^2})$ again by \cref{prop:frac_div_repr} and performing a standard cut-off approximation argument, we can integrate by parts and obtain
\begin{equation*}
\int_{\R^n}\phi\cdot I_{1-\alpha}\nabla_w f\,dx
=\int_{\R^n}I_{1-\alpha}\phi\cdot\nabla_w f\,dx
=-\int_{\R^n}f\,\div I_{1-\alpha}\phi\,dx
=-\int_{\R^n}f\,\div^\alpha\phi\,dx.
\end{equation*}
Therefore
\begin{equation*}
\int_{\R^n}\phi\cdot I_{1-\alpha}\nabla_w f\,dx
=-\int_{\R^n}f\,\div^\alpha\phi\,dx
\end{equation*}
for all $\phi\in C^\infty_c(\R^n;\R^n)$, proving~\eqref{eq:nabla_alpha_Lp_repres}. In particular, notice that $\nabla^\alpha_w f\in L^q(\R^n;\R^n)$ by inequality~\eqref{eq:Riesz_potential_boundedness}. The proof is complete.
\end{proof}

For the case $p=+\infty$, we have the following immediate consequence of \cref{res:leibniz_frac_nabla_Lip_b} and \cref{res:Lip_b_int_by_parts}.

\begin{corollary}[$W^{1,\infty}(\R^n)\subset S^{\alpha,\infty}(\R^n)$]
Let $\alpha\in(0,1)$. If $f\in W^{1,\infty}(\R^n)$, then $f\in S^{\alpha,\infty}(\R^n)$ with
\begin{equation}\label{eq:W_1_infty_inclusion_S_alpha_infty}
\|\nabla^\alpha f\|_{L^\infty(\R^n;\,\R^n)}	
\le2^{1-\alpha}\,\frac{n\omega_n\mu_{n,\alpha}}{\alpha(1-\alpha)}\,\|\nabla_w f\|_{L^\infty(\R^n;\,\R^n)}^\alpha\|f\|_{L^\infty(\R^n)}^{1-\alpha}.
\end{equation}
\end{corollary}

\subsection{Two representation formulas for the \texorpdfstring{$\alpha$}{alpha}-variation}

In this section, we prove two useful representation formulas for the $\alpha$-variation.

We begin with the following weak representation formula for the fractional $\alpha$-variation of functions in~$BV_{\loc}(\R^n) \cap L^\infty(\R^n)$. Here and in the following, we denote by~$f^\star$ the \emph{precise representative} of $f\in L^1_{\loc}(\R^n)$, see~\eqref{eq:def_precise_repres} for the definition.  

\begin{proposition}
\label{prop:general_repr_nabla_alpha}
Let $\alpha \in (0, 1)$ and $f \in BV_{\loc}(\R^n) \cap L^\infty(\R^n)$. Then $\nabla^\alpha f \in L^1_{\loc}(\R^n;\R^n)$ and
\begin{equation} \label{eq:general_repr_formula}
\int_{\R^{n}} \varphi \cdot \nabla^{\alpha} f \, dx 
= \lim_{R \to + \infty} \int_{\R^{n}} \varphi \cdot I_{1 - \alpha} ( \chi_{B_{R}}^\star D f) \, dx
\end{equation}
for all $\phi \in \Lip_{c}(\R^{n}; \R^{n})$.
\end{proposition}

\begin{proof}
By \cref{lem:nabla_alpha_f_L1_and_L1_loc}\eqref{item:lem:nabla_alpha_f_loc_3}, we know that $\nabla^{\alpha} f \in L^{1}_{\rm loc}(\R^{n}; \R^{n})$ for all $\alpha\in(0,1)$. By \cref{res:BV_cut_ball}, we also know that $f\chi_{B_R}\in BV(\R^n)\cap L^\infty(\R^n)$ with $D(\chi_{B_{R}} f) = \chi_{B_{R}}^\star D f + f^\star D \chi_{B_{R}}$ for all $R>0$. Now fix $\phi \in \Lip_c(\R^{n}; \R^{n})$ and take $R > 0$ such that $\supp\phi\subset B_{R/2}$. By~\cite{CS18}*{Theorem 3.18}, we have that 
\begin{align*} 
\int_{\R^{n}} \chi_{B_{R}} f \, \div^{\alpha} \varphi \, dx 
= - \int_{\R^{n}} \varphi \cdot \nabla^{\alpha} (\chi_{B_{R}} f) \, dx 
=-\int_{\R^n}\phi\cdot I_{1-\alpha}D(\chi_{B_R}f)\,dx.
\end{align*}
Moreover, we can split the last integral as
\begin{equation}\label{eq:decomposition}
\int_{\R^n}\phi\cdot I_{1-\alpha}D(\chi_{B_R}f)\,dx
=\int_{\R^n}\phi\cdot I_{1-\alpha}(\chi_{B_R}^\star Df)\,dx
+\int_{\R^n}\phi\cdot I_{1-\alpha}(f^\star D\chi_{B_R})\,dx.
\end{equation}
For all $x\in B_{R/2}$, we can estimate 
\begin{align*}
\abs*{I_{1-\alpha}(f^\star  D\chi_{B_R})(x)}
&=\left | \int_{\partial B_{R}} \frac{f^\star(y)}{|x - y|^{n + \alpha - 1}} \frac{y}{|y|} \, d \Haus{n - 1}(y) \right | \\
& = \frac{1}{R^\alpha} \left | \int_{\partial B_{1}} \frac{f^\star(R y)}{\left |y - \frac{x}{R} \right |^{n + \alpha - 1}} \frac{y}{|y|} \, d \Haus{n - 1}(y) \right | \\
& \le \frac{n \omega_{n}}{R^{\alpha} \left ( 1 - \frac{|x|}{R} \right )^{n + \alpha - 1}} \, \|f\|_{L^{\infty}(\R^{n})}\\
&\le\frac{2^{n+\alpha-1}n\omega_n}{R^\alpha}\, \|f\|_{L^{\infty}(\R^{n})}
\end{align*}
and so, since $\supp\phi\subset B_{R/2}$, we get that
\begin{equation}\label{eq:est_trace}
\abs*{\int_{\R^n}\phi\cdot I_{1-\alpha}(f^\star  D\chi_{B_R})\,dx\,} 
\le\frac{2^{n+\alpha-1}n\omega_n}{R^\alpha}\, \|\phi\|_{L^1(\R^n;\R^n)}\,\|f\|_{L^{\infty}(\R^{n})}.
\end{equation}
Therefore, by~\eqref{eq:frac_div_repr_Lip_L1_estimate}, Lebesgue's Dominated Convergence Theorem, \eqref{eq:decomposition} and~\eqref{eq:est_trace}, we get that
\begin{equation*}
\int_{\R^{n}} f \, \div^{\alpha} \varphi \, dx 
=\lim_{R \to + \infty} \int_{\R^{n}} \chi_{B_{R}} f \, \div^{\alpha} \varphi \, dx
=\lim_{R \to + \infty}\int_{\R^n}\phi\cdot I_{1-\alpha}(\chi_{B_R}^\star Df)\,dx
\end{equation*}
and the conclusion follows.
\end{proof}

In the following result, we show that for all functions in $bv(\R^n)\cap L^\infty(\R^n)$ one can actually pass to the limit as $R\to+\infty$ inside the integral in the right-hand side of~\eqref{eq:general_repr_formula}.

\begin{corollary}
\label{result:repres_formula_nabla_alpha_bv_bounded}
If either $f\in BV(\R^n)$ or $f\in bv(\R^n)\cap L^\infty(\R^n)$, then 
\begin{equation}\label{eq:repres_formula_nabla_alpha_bv_bounded}
\nabla^\alpha f=I_{1-\alpha}Df
\quad
\text{a.e.\ in~$\R^n$}.
\end{equation} 
\end{corollary}

\begin{proof}
If $f\in BV(\R^n)$, then~\eqref{eq:repres_formula_nabla_alpha_bv_bounded} coincides with~\eqref{eq:weak_frac_grad_repr} and there is nothing to prove. So let us assume that $f\in bv(\R^n)\cap L^\infty(\R^n)$. 
Writing $Df = \nu_{f} |D f|$ with $\nu_f\in\mathbb{S}^{n-1}$ $|Df|$-a.e.\ in~$\R^n$, for all $x\in\R^n$ we have
\begin{equation*}
\lim_{R\to+\infty}\chi_{B_R}^\star (y)\,\frac{\nu_f(y)}{|y-x|^{n+\alpha-1}}=\frac{\nu_f(y)}{|y-x|^{n+\alpha-1}}
\quad
\text{for $|Df|$-a.e.~$y \neq x$}.
\end{equation*}
Moreover, for all $x\in\R^n$, we have
\begin{equation*}
\abs*{\chi_{B_R}^\star (y)\,\frac{\nu_f(y)}{|y-x|^{n+\alpha-1}}}
\le\frac{1}{|y-x|^{n+\alpha-1}}\in L^1_y(\R^n,|Df|)
\quad
\forall R>0, 
\end{equation*}
because $I_{1-\alpha}|Df|\in L^1_{\loc}(\R^n)$ by~\cite{CS18}*{Lemma~2.4}. Therefore, by Lebesgue's Dominated Convergence Theorem (applied with respect to the finite measure~$|Df|$), we get that
\begin{equation*}
\lim_{R\to+\infty}I_{1 - \alpha} ( \chi_{B_{R}}^\star D f)(x)=(I_{1 - \alpha} D f)(x)
\quad
\text{for a.e.~$x\in\R^n$}. 
\end{equation*} 
Now let $\phi \in \Lip_{c}(\R^{n}; \R^{n})$. Since
\begin{equation*}
|\phi\cdot I_{1-\alpha}(\chi_{B_R}^\star Df)|\le|\phi|\,I_{1-\alpha}|Df|\in L^1(\R^n)
\quad
\forall R>0, 	
\end{equation*}
again by Lebesgue's Dominated Convergence Theorem we get that
\begin{equation}\label{eq:limit_double_DCT}
\lim_{R\to+\infty} \int_{\R^{n}} \phi \cdot I_{1 - \alpha} ( \chi_{B_{R}}^\star D f) \, dx
=\int_{\R^{n}} \phi \cdot I_{1 - \alpha} D f \, dx.
\end{equation}
The conclusion thus follows by combining~\eqref{eq:general_repr_formula} with~\eqref{eq:limit_double_DCT}. 
\end{proof}

\subsection{Relation between \texorpdfstring{$BV^\beta$}{BVˆbeta} and \texorpdfstring{$BV^{\alpha, p}$}{BVˆalpha,p} for \texorpdfstring{$\beta<\alpha$}{beta<alpha} and \texorpdfstring{$p > 1$}{p>1}}

Let us recall the following result, see~\cite{CS18}*{Lemma~3.28}.

\begin{lemma}\label{result:correspondence_BV_alpha_bv}
Let $\alpha\in(0,1)$. The following properties hold.
\begin{enumerate}[\indent (i)]
\item\label{item:BV_alpha_bv_1} If $f\in BV^\alpha(\R^n)$, then $u:=I_{1-\alpha}f\in bv(\R^n)$ with $Du=D^\alpha f$ in $\M(\R^{n}; \R^{n})$.
\item\label{item:BV_alpha_bv_2} If $u\in BV(\R^n)$, then $f:= (-\Delta)^{\frac{1-\alpha}{2}}u\in BV^\alpha(\R^n)$ with 
\begin{equation*}
\|f\|_{L^1(\R^n)}\le c_{n,\alpha}\|u\|_{BV(\R^n)}
\quad\text{and}\quad
D^\alpha f=D u
\quad\text{in $\M(\R^{n}; \R^{n})$}.
\end{equation*}
As a consequence, the operator $(-\Delta)^{\frac{1-\alpha}{2}}\colon BV(\R^n)\to BV^\alpha(\R^n)$ is continuous.
\end{enumerate}
\end{lemma}

We can thus relate functions with bounded $\alpha$-variation and functions with bounded variation via Riesz potential and the fractional Laplacian. We would like to prove a similar result between functions with bounded $\alpha$-variation and functions with bounded $\beta$-variation, for any couple of exponents $0<\beta<\alpha < 1$.

However, although the standard variation of a function $f\in L^1_{\loc}(\R^n)$ is well defined, it is not clear whether the functional
\begin{equation}\label{eq:functional}
\phi\mapsto\int_{\R^{n}} f \, \div^{\alpha} \phi \, dx
\end{equation}
is well posed for all $\phi \in C^{\infty}_{c}(\R^{n}; \R^{n})$, since $\div^{\alpha} \phi$ does not have compact support. Nevertheless, thanks to \cref{prop:frac_div_repr}, the functional in~\eqref{eq:functional} is well defined as soon as $f\in L^p(\R^n)$ for some $p\in[1,+\infty]$.
Hence, it seems natural to define the space
\begin{equation}\label{eq:def_BV_alpha_p}
BV^{\alpha, p}(\R^{n}) := 
\set*{ f \in L^{p}(\R^{n}) : |D^{\alpha} f|(\R^{n}) < \infty }
\end{equation}
for any $\alpha \in (0, 1)$ and $p \in [1, + \infty]$. In particular,  $BV^{\alpha, 1}(\R^{n}) = BV^{\alpha}(\R^{n})$. Similarly, we let
\begin{equation*}
BV^{1,p}(\R^n):=
\set*{f\in L^p(\R^n) : |Df|(\R^n)<+\infty}
\end{equation*}
for all $p \in [1, + \infty]$. In particular, $BV^{1, 1}(\R^{n}) = BV(\R^{n})$.

A further justification for the definition of these new spaces comes from the following fractional version of the Gagliardo--Nirenberg--Sobolev embedding: if $n \ge 2$ and $\alpha \in (0, 1)$, then $BV^{\alpha}(\R^{n})$ is continuously embedded in $L^{p}(\R^{n})$ for all $p \in \left [1, \frac{n}{n - \alpha} \right ]$, see~\cite{CS18}*{Theorem~3.9}. Hence, thanks to~\eqref{eq:def_BV_alpha_p}, we can equivalently write
\begin{equation*}
BV^{\alpha}(\R^{n}) \subset BV^{\alpha, p}(\R^{n})
\end{equation*}
with continuous embedding for all $n \ge 2$, $\alpha \in (0, 1)$ and $p \in \left [1, \frac{n}{n - \alpha} \right]$.

Incidentally, we remark that the continuous embedding $BV^\alpha(\R^n)\subset L^{\frac n{n-\alpha}}(\R^n)$ for~\mbox{$n\ge2$} and $\alpha\in(0,1)$ can be improved using the main result of the recent work~\cite{Spector18} (see also~\cite{Spector19}). Indeed, if $n\ge2$, $\alpha\in(0,1)$ and $f\in C^\infty_c(\R^n)$, then, by taking $F=\nabla^\alpha f$ in~\cite{Spector18}*{Theorem~1.1}, we have that  
\begin{equation*}
\|f\|_{L^{\frac n{n-\alpha},1}(\R^n)}
\le c_{n,\alpha}\|I_\alpha\nabla^\alpha f\|_{L^{\frac n{n-\alpha},1}(\R^n;\,\R^n)}
\le c_{n,\alpha}'\|\nabla^\alpha f\|_{L^1(\R^n;\,\R^n)}
\end{equation*}
thanks to the boundedness of the Riesz transform $R\colon L^{\frac n{n-\alpha},1}(\R^n)\to L^{\frac n{n-\alpha},1}(\R^n;\R^n)$, where $c_{n,\alpha},c_{n,\alpha}'>0$ are two constants depending only on~$n$ and~$\alpha$, and $L^{\frac n{n-\alpha},1}(\R^n)$ is the Lorentz space of exponents~$\frac n{n-\alpha},1$ (we refer to~\cites{G14-C,G14} for an account on Lorentz spaces and on the properties of Riesz transform). Thus, recalling~\cite{CS18}*{Theorem~3.8}, we readily deduce the continuous embedding $BV^\alpha(\R^n)\subset L^{\frac n{n-\alpha},1}(\R^n)$ for $n\ge2$ and $\alpha\in(0,1)$ by~\cite{G14-C}*{Exercise~1.1.1(b)} and Fatou's Lemma. This suggests that the spaces defined in~\eqref{eq:def_BV_alpha_p} may be further enlarged by considering functions belonging to some Lorentz space, but we do not need this level of generality here.

In the case $n=1$, the space $BV^\alpha(\R)$ does not embed in $L^{\frac1{1-\alpha}}(\R)$ with continuity, see~\cite{CS18}*{Remark~3.10}. However, somehow completing the picture provided by~\cite{Spector18}, we can prove that the space $BV^\alpha(\R)$ continuously embeds in the Lorentz space $L^{\frac{1}{1 - \alpha}, \infty}(\R)$. Although this result is truly interesting only for $n=1$, we prove it below in all dimensions for the sake of completeness.

\begin{theorem}[Weak Gagliardo--Nirenberg--Sobolev inequality]
Let $\alpha\in(0,1)$. There exists a constant $c_{n,\alpha}>0$ such that
\begin{equation}\label{eq:weak_GNS}
\|f\|_{L^{\frac{n}{n-\alpha},\infty}(\R^n)}
\le c_{n,\alpha}|D^\alpha f|(\R^n)
\end{equation}
for all $f\in BV^\alpha(\R^n)$. As a consequence, $BV^{\alpha}(\R^{n})$ is continuously embedded in $L^q(\R^n)$ for any $q\in[1,\frac{n}{n-\alpha})$.
\end{theorem}

\begin{proof}
Let $f\in C^\infty_c(\R^n)$. By~\cite{S19}*{Theorem~3.5} (see also~\cite{CS18}*{Section~3.6}), we have
\begin{align*}
f(x)
=-\div^{-\alpha}\nabla^\alpha f(x)
=-\mu_{n,-\alpha}\int_{\R^n}\frac{(y-x)\cdot\nabla^\alpha f(y)}{|y-x|^{n+1-\alpha}}\,dy,
\quad
x\in\R^n,
\end{align*}
so that
\begin{align*}
|f(x)|
\le\mu_{n,-\alpha}\int_{\R^n}\frac{|\nabla^\alpha f(y)|}{|y-x|^{n-\alpha}}\,dy
=\frac{\mu_{n,-\alpha}}{\mu_{n,1-\alpha}}\,(n-\alpha)\, I_\alpha|\nabla^\alpha f|(x),
\quad
x\in\R^n.
\end{align*}
Since $I_\alpha\colon L^1(\R^n)\to L^{\frac n{n-\alpha},\infty}(\R^n)$ is a continuous operator by Hardy--Littlewood--Sobolev inequality (see~\cite{S70}*{Theorem~1, Chapter~V} or \cite{G14-C}*{Theorem 1.2.3}), we can estimate
\begin{align*}
\|f\|_{L^{\frac{n}{n-\alpha},\infty}(\R^n)}
&\le\frac{n\,\mu_{n,-\alpha}}{\mu_{n,1-\alpha}}\|I_\alpha|\nabla^\alpha f|\|_{L^{\frac{n}{n-\alpha},\infty}(\R^n)}
\le c_{n,\alpha}\||\nabla^\alpha f|\|_{L^1(\R^n)}
=c_{n,\alpha}\,|D^\alpha f|(\R^n),
\end{align*}
where $c_{n,\alpha}>0$ is a constant depending only on~$n$ and~$\alpha$. Thus, inequality~\eqref{eq:weak_GNS} follows for all $f\in C^\infty_c(\R^n)$. Now let $f\in BV^\alpha(\R^n)$. By~\cite{CS18}*{Theorem~3.8}, there exists $(f_k)_{k\in\N}\subset C^\infty_c(\R^n)$ such that $f_k\to f$ a.e.\ in~$\R^n$ and $|D^\alpha f_k|(\R^n)\to |D^\alpha f|(\R^n)$ as $k\to+\infty$. By~\cite{G14-C}*{Exercise~1.1.1(b)} and Fatou's Lemma, we thus get
\begin{align*}
\|f\|_{L^{\frac{n}{n-\alpha},\infty}(\R^n)}
\le\liminf_{k\to+\infty}\|f_k\|_{L^{\frac{n}{n-\alpha},\infty}(\R^n)}
\le c_{n,\alpha}\lim_{k\to+\infty}|D^\alpha f_k|(\R^n)
=c_{n,\alpha}|D^\alpha f|(\R^n)
\end{align*} 
and so~\eqref{eq:weak_GNS} readily follows. Finally, thanks to~\cite{G14-C}*{Proposition~1.1.14}, we obtain the continuous embedding of $BV^{\alpha}(\R^{n})$ in $L^{q}(\R^{n})$ for all $q \in [1, \frac{n}{n - \alpha} )$.
\end{proof}

\begin{remark}[The embedding $BV^{\alpha}(\R)\subset L^{\frac{1}{1 - \alpha}, \infty}(\R)$ is sharp] 
Let $\alpha\in(0,1)$. The continuous embedding $BV^{\alpha}(\R)\subset L^{\frac{1}{1 - \alpha}, \infty}(\R)$ is sharp at the level of Lorentz spaces, in the sense that $BV^\alpha(\R^n)\setminus L^{\frac{1}{1 - \alpha}, q}(\R)\ne\varnothing$ for any $q \in [1, + \infty)$. Indeed, if we let
\begin{equation*}
f_{\alpha}(x) = |x - 1|^{\alpha - 1} \sgn(x - 1) - |x|^{\alpha - 1} \sgn(x),
\qquad
x\in\R\setminus\set{0,1},
\end{equation*}
then $f_{\alpha} \in BV^{\alpha}(\R)$ by~\cite{CS18}*{Theorem 3.26}, and it is not difficult to prove that $f_{\alpha} \in L^{\frac{1}{1 - \alpha}, \infty}(\R)$. However, we can find a constant $c_{\alpha} > 0$ such that
\begin{equation*}
|f_{\alpha}(x)| \ge c_{\alpha} |x|^{\alpha - 1} \chi_{\left ( - \frac{1}{4}, \frac{1}{4} \right)}(x) =: g_{\alpha}(x),
\qquad
x\in\R\setminus\set{0,1},
\end{equation*}
so that $d_{f_{\alpha}} \ge d_{g_{\alpha}}$, where $d_{f_{\alpha}}$ and $d_{g_{\alpha}}$ are the \emph{distribution functions} of $f_{\alpha}$ and $g_{\alpha}$.
A simple calculation shows that
\begin{equation*}
d_{g_{\alpha}}(s) = 
\begin{cases} 
\dfrac{1}{2} & \text{if} \ 0 < s \le c_{\alpha} 4^{1 - \alpha}\\[4mm]
2 \left ( \dfrac{c_{\alpha}}{s} \right )^{\frac{1}{1 - \alpha}} & \text{if} \ \ s> c_{\alpha} 4^{1 - \alpha}, 
\end{cases}
\end{equation*}
so that, by~\cite{G14-C}*{Proposition 1.4.9}, we obtain
\begin{align*}
\|f_{\alpha}\|_{L^{\frac{1}{1 - \alpha}, q}(\R)}^{q} 
& \ge \|g_{\alpha}\|_{L^{\frac{1}{1 - \alpha}, q}(\R)}^{q} 
= \frac{1}{1 - \alpha} \int_{0}^{+\infty} \left [ d_{g_{\alpha}}(s) \right ]^{q(1 - \alpha)} s^{q - 1} \, ds\\
& \ge \frac{2^{q(1 - \alpha)}}{1 - \alpha} \int_{c_{\alpha} 4^{1 - \alpha}}^{+\infty} s^{-q} s^{q - 1} \, ds 
= + \infty
\end{align*}
and thus $f_{\alpha} \notin L^{\frac1{1-\alpha},q}(\R)$ for any $q\in[1,+\infty)$.
\end{remark}

We collect the above continuous embeddings in the following statement.

\begin{corollary}[The embedding $BV^{\alpha}\subset BV^{\alpha, p}$]
\label{cor:BV_alpha_p_emb}
Let $\alpha \in (0, 1)$ and $p \in\left[1, \frac{n}{n - \alpha}\right)$. We have $BV^{\alpha}(\R^{n}) \subset BV^{\alpha, p}(\R^{n})$ with continuous embedding. In addition, if $n \ge 2$, then also $BV^{\alpha}(\R^{n}) \subset BV^{\alpha, \frac{n}{n - \alpha}}(\R^{n})$ with continuous embedding.
\end{corollary}

With \cref{cor:BV_alpha_p_emb} at hands, we are finally ready to investigate the relation between $\alpha$-variation and $\beta$-variation for $0 < \beta < \alpha < 1$.

\begin{lemma}\label{res:BV_beta_bv_alpha}
Let $0<\beta<\alpha<1$. The following hold.
\begin{enumerate}[(i)]
\item\label{item:res:BV_beta_bv_alpha_1} 
If $f\in BV^\beta(\R^n)$, then $u:=I_{\alpha-\beta}f\in BV^{\alpha, p}(\R^n)$ for any $p \in \left ( \frac{n}{n - \alpha + \beta}, \frac{n}{n - \alpha} \right)$ (including $p = \frac{n}{n - \alpha}$ if $n \ge 2$), with $D^\alpha u=D^\beta f$ in $\mathscr{M}(\R^n;\R^n)$.
\item\label{item:res:BV_beta_bv_alpha_2} 
If $u\in BV^\alpha(\R^n)$, then $f:=(-\Delta)^{\frac{\alpha-\beta}{2}}u\in BV^\beta(\R^n)$ with
\begin{equation*}
\|f\|_{L^1(\R^n)}\le c_{n,\alpha,\beta}\,\|u\|_{BV^\alpha(\R^n)}
\quad\text{and}\quad
D^\beta f=D^\alpha u
\quad\text{in}\ \mathscr{M}(\R^n;\R^n).
\end{equation*}
\end{enumerate}
As a consequence, the operator $(-\Delta)^{\frac{\alpha-\beta}{2}}\colon BV^\alpha(\R^n)\to BV^\beta(\R^n)$ is continuous.
\end{lemma}

\begin{proof}
We begin with the following observation. Let $\phi\in C^\infty_c(\R^n;\R^n)$ and let $U\subset\R^n$ be a bounded open set such that $\supp\phi\subset U$. By \cref{prop:frac_div_repr} and the \emph{semigroup property}~\eqref{eq:Riesz_potential_semigroup} of the Riesz potential, we can write
\begin{equation*}
\div^\beta\phi
=I_{1-\beta}\div\phi
=I_{\alpha-\beta}I_{1-\alpha}\div\phi
=I_{\alpha-\beta}\div^\alpha\phi.
\end{equation*} 
Similarly, we also have
\begin{equation*}
I_{\alpha-\beta}|\div^\alpha\phi|
=I_{\alpha-\beta}|I_{1-\alpha}\div\phi|
\le I_{\alpha-\beta}I_{1-\alpha}|\div\phi|
=I_{1-\beta}|\div\phi|,
\end{equation*}
so that $I_{\alpha-\beta}|\div^\alpha\phi|\in L^\infty(\R^n)$ with
\begin{equation*}
\|I_{\alpha-\beta}|\div^\alpha\phi|\|_{L^\infty(\R^n)}
\le\|I_{1-\beta}|\div\phi|\|_{L^\infty(\R^n)} 
\le C_{n,\beta,U}\|\div\phi\|_{L^\infty(\R^n)} 
\end{equation*}
by~\cite{CS18}*{Lemma~2.4}. We now prove the two statements separately.

\smallskip

\textit{Proof of~(\ref{item:res:BV_beta_bv_alpha_1})}.
Let $f\in BV^\beta(\R^n)$ and $\phi\in C^\infty_c(\R^n;\R^n)$. Thanks to \cref{cor:BV_alpha_p_emb}, if $n \ge 2$, then $f \in BV^{\beta, q}(\R^{n})$ for any $q \in [1, \frac{n}{n - \beta}]$ and so $I_{\alpha - \beta} f \in L^{p}(\R^{n})$ for any $p \in \left ( \frac{n}{n - \alpha + \beta}, \frac{n}{n - \alpha} \right ]$ by~\eqref{eq:Riesz_potential_boundedness}. If instead $n = 1$, then $f \in BV^{\beta, q}(\R)$ for any $q \in [1, \frac{1}{1 - \beta})$ and so $I_{\alpha - \beta} f \in L^{p}(\R)$ for any $p \in \left ( \frac{1}{1 - \alpha + \beta}, \frac{1}{1 - \alpha} \right )$. Since $f\in L^1(\R^n)$ and $I_{\alpha-\beta}|\div^\alpha\phi|\in L^\infty(\R^n)$, by Fubini's Theorem we have
\begin{equation}\label{eq:BV_beta_bv_alpha_int_by_parts}
\int_{\R^n}f\,\div^\beta\phi\,dx
=\int_{\R^n}f\,I_{\alpha-\beta}\div^\alpha\phi\,dx
=\int_{\R^n}u\,\div^\alpha\phi\,dx,
\end{equation}  
proving that $u:=I_{\alpha-\beta}f\in BV^{\alpha, p}(\R^n)$ for any $p \in \left ( \frac{n}{n - \alpha + \beta}, \frac{n}{n - \alpha} \right )$ (including $p = \frac{n}{n - \alpha}$ if $n \ge 2$), with $D^\alpha u=D^\beta f$ in~$\mathscr{M}(\R^n;\R^n)$.

\smallskip

\textit{Proof of~(\ref{item:res:BV_beta_bv_alpha_2})}.
Let $u\in BV^\alpha(\R^n)$. By~\cite{CS18}*{Theorem~3.32}, we know that $u\in W^{\alpha-\beta,1}(\R^n)$, so that $f:=(-\Delta)^{\frac{\alpha-\beta}{2}}u\in L^1(\R^n)$ with $\|f\|_{L^1(\R^n)}\le c_{n,\alpha,\beta}\,\|u\|_{BV^\alpha(\R^n)}$, see~\cite{CS18}*{Section~3.10}. Then, arguing as before, for any $\phi\in C^\infty_c(\R^n;\R^n)$ we get~\eqref{eq:BV_beta_bv_alpha_int_by_parts}, since we have $I_{\alpha-\beta}f =u$ in~$L^1(\R^n)$ (see \cite{CS18}*{Section 3.10}). The proof is complete.
\end{proof}

\subsection{The inclusion \texorpdfstring{$BV^\alpha\subset W^{\beta,1}$}{BVˆalpha in Wˆbeta,1} for \texorpdfstring{$\beta<\alpha$}{beta<alpha}: a representation formula}
\label{subsec:inclusion_BV_alpha_Sobolev_beta}

In~\cite{CS18}*{Theorem~3.32}, we proved that the inclusion $BV^\alpha\subset W^{\beta,1}$ is continuous for $\beta<\alpha$. In the following result we prove a useful representation formula for the fractional $\beta$-gradient of any $f\in BV^\alpha(\R^n)$, extending the formula obtained in \cref{result:repres_formula_nabla_alpha_bv_bounded}. 

\begin{proposition}\label{res:repres_formula_BV_alpha_beta}
Let $\alpha\in(0,1)$. If $f\in BV^\alpha(\R^n)$, then $f\in W^{\beta,1}(\R^n)$ for all $\beta\in(0,\alpha)$ with
\begin{equation}\label{eq:repres_formula_BV_alpha_beta}
\nabla^\beta f=I_{\alpha-\beta}D^\alpha f
\quad
\text{a.e.\ in $\R^n$}.
\end{equation}
In addition, for any bounded open set $U\subset\R^n$, we have
\begin{equation}\label{eq:weak_frac_grad_loc_estimate_alpha_beta}
\|\nabla^\beta f\|_{L^1(U;\,\R^n)}
\le C_{n,(1 - \alpha+\beta),U}\,|D^\alpha f|(\R^n)
\end{equation}
for all $\beta\in(0,\alpha)$, where~$C_{n,\alpha,U}$ is as in~\eqref{eq:constant_estimate}. 
Finally, given an open set $A\subset\R^n$, we have
\begin{equation}\label{eq:Davila_estimate_alpha_beta}
\|\nabla^\beta f\|_{L^1(A;\,\R^n)}
\le\frac{\mu_{n,1 + \beta - \alpha}}{n+\beta-\alpha}\left(\frac{\omega_{n, 1}|D^\alpha f|(\closure[-.5]{A_r})}{\alpha-\beta}\,r^{\alpha-\beta}+\frac{\omega_{n,\alpha}(n + 2 \beta -\alpha)}{\beta}\,\|f\|_{L^1(\R^n)}\,r^{-\beta}\right)
\end{equation} 
for all $r>0$ and all $\beta\in(0,\alpha)$, where $\omega_{n,\alpha}:=\|\nabla^\alpha\chi_{B_1}\|_{L^{1}(\R^n; \R^{n})}$, $\omega_{n, 1} := |D \chi_{B_{1}}|(\R^{n}) = n \omega_{n}$, and, as above, $A_r:=\set*{x\in\R^n : \dist(x,A)<r}$. In particular, we have
\begin{equation} \label{eq:Davila_estimate_alpha_beta_interp}
\|\nabla^\beta f\|_{L^1(\R^{n};\,\R^n)} 
\le 
\frac{\alpha \mu_{n, 1 + \beta - \alpha} \omega_{n, 1}^{\frac{\beta}{\alpha}} \omega_{n, \alpha}^{1 - \frac{\beta}{\alpha}} (n + 2 \beta - \alpha)^{1 - \frac{\beta}{\alpha}}}
{\beta(n + \beta - \alpha)(\alpha - \beta)} \,
\|f\|_{L^{1}(\R^{n})}^{1 - \frac{\beta}{\alpha}}\, 
|D^{\alpha}f|(\R^{n})^{\frac{\beta}{\alpha}}.
\end{equation}
\end{proposition}

\begin{proof}
Fix $\beta\in(0,\alpha)$. By~\cite{CS18}*{Theorem~3.32} we already know that $f\in W^{\beta,1}(\R^n)$, with $D^\beta f=\nabla^\beta f\Leb{n}$ according to~\cite{CS18}*{Theorem~3.18}. We thus just need to prove~\eqref{eq:repres_formula_BV_alpha_beta}, \eqref{eq:weak_frac_grad_loc_estimate_alpha_beta} and~\eqref{eq:Davila_estimate_alpha_beta}.

We prove~\eqref{eq:repres_formula_BV_alpha_beta}. 
Let $\phi\in C^\infty_c(\R^n;\R^n)$. Note that $I_{\alpha-\beta}\phi\in\Lip_b(\R^n;\R^n)$ is such that $\div I_{\alpha-\beta}\phi=I_{\alpha-\beta}\div\phi$, so that 
\begin{equation*}
I_{1-\alpha}\div I_{\alpha-\beta}\phi
=I_{1-\alpha}I_{\alpha-\beta}\div\phi
=I_{1-\beta}\div\phi
=\div^\beta\phi
\end{equation*}
by the \emph{semigroup property}~\eqref{eq:Riesz_potential_semigroup} of the Riesz potential. Moreover, in a similar way, we have
\begin{equation*}
I_{1-\alpha}|\div I_{\alpha-\beta}\phi|
=I_{1-\alpha}|I_{\alpha-\beta}\div\phi|
\le I_{1-\alpha}I_{\alpha-\beta}|\div\phi|
=I_{1-\beta}|\div\phi|
\in L^1_{\loc}(\R^n).
\end{equation*}
By \cref{res:frac_div_Lip_b}, we thus have that $\div^\alpha I_{\alpha-\beta}\phi=\div^\beta\phi$. Consequently, by \cref{res:Lip_b_test}, we get
\begin{equation*}
\int_{\R^n}f\,\div^\beta\phi\,dx
=\int_{\R^n}f\,\div^\alpha I_{\alpha-\beta}\phi\,dx
=-\int_{\R^n}I_{\alpha-\beta}\phi\cdot dD^\alpha f.
\end{equation*}
Since $|D^\alpha f|(\R^n)<+\infty$, we have $I_{\alpha-\beta}|D^\alpha f|\in L^1_{\loc}(\R^n)$ and thus, by Fubini's Theorem, we get that
\begin{equation*}
\int_{\R^n}I_{\alpha-\beta}\phi\cdot dD^\alpha f
=\int_{\R^n}\phi\cdot I_{\alpha-\beta}D^\alpha f\,dx.
\end{equation*}
We conclude that
\begin{equation*}
\int_{\R^n}f\,\div^\beta\phi\,dx
=-\int_{\R^n}\phi\cdot I_{\alpha-\beta}D^\alpha f\,dx
\end{equation*}
for any $\phi\in C^\infty_c(\R^n;\R^n)$, proving~\eqref{eq:repres_formula_BV_alpha_beta}. 

We prove~\eqref{eq:weak_frac_grad_loc_estimate_alpha_beta}. By~\eqref{eq:repres_formula_BV_alpha_beta}, by Tonelli's Theorem and by~\cite{CS18}*{Lemma~2.4}, we get
\begin{align*} 
\int_U |\nabla^{\beta} f|\,dx
\le\int_U I_{\alpha-\beta}|D^\alpha f|\,dx
\le C_{n,(1-\alpha+\beta),U}|D^\alpha f|(\R^n)
\end{align*}
where $C_{n,\alpha, U}$ is as in~\eqref{eq:constant_estimate}.

We now prove~\eqref{eq:Davila_estimate_alpha_beta} in two steps. We argue as in the proof of~\eqref{eq:Davila_estimate}.

\smallskip

\textit{Proof of~\eqref{eq:Davila_estimate_alpha_beta}, Step~1}. Assume $f\in C^\infty_c(\R^n)$ and fix $r>0$. We have
\begin{align*}
\int_A|\nabla^\beta f| dx
&=\int_A|I_{\alpha-\beta}\nabla^\alpha f|\,dx\\
&\le\frac{\mu_{n,1+\beta-\alpha}}{n+\beta-\alpha}\left(\int_A\int_{\{|h|<r\}}\frac{|\nabla^\alpha f(x+h)|}{|h|^{n+\beta-\alpha}} dh dx
+\int_A\abs*{\int_{\{|h|\ge r\}}\frac{\nabla^\alpha f(x+h)}{|h|^{n+\beta-\alpha}}\,dh } dx\right).
\end{align*}
We estimate the two double integrals appearing in the right-hand side separately. By Tonelli's Theorem, we have
\begin{align*}
\int_A\int_{\{|h|<r\}}\frac{|\nabla^\alpha f(x+h)|}{|h|^{n+\beta-\alpha}}\,dh\,dx
&=\int_{\{|h|<r\}}\int_A|\nabla^\alpha f(x+h)|\,dx\,\frac{dh}{|h|^{n+\beta-\alpha}}\\
&\le|D^\alpha f|(A_r)\int_{\{|h|<r\}}\frac{dh}{|h|^{n+\beta-\alpha}}\\
&=\frac{n\omega_n\,|D^\alpha f|(A_r)}{\alpha-\beta}\,r^{\alpha-\beta}.
\end{align*}
Concerning the second double integral, we apply~\cite{AH96}*{Lemma~3.1.1(c)} to each component of the measure $D^\alpha f\in\mathscr{M}(\R^n;\R^n)$ and get
\begin{align*}
\int_{\{|h|\ge r\}}\frac{\nabla^\alpha f(x+h)}{|h|^{n+\beta-\alpha}}\,dh
&=(n+\beta-\alpha)\int_r^{+\infty}\frac{D^\alpha f(B_\rho(x))}{\rho^{n+\beta-\alpha+1}}\,d\rho
-\frac{D^\alpha f(B_r(x))}{r^{n+\beta-\alpha}}
\end{align*}
for all $x\in A$. Since
\begin{align*}
D^\alpha f(B_\rho(x))
&=\int_{\R^n}\chi_{B_\rho}(y)\,\nabla^\alpha f(x+y)\,dy\\
&=-\int_{\R^n} f(x+y)\,\nabla^\alpha\chi_{B_\rho}(y)\,dy\\
&=-\rho^{n-\alpha}\int_{\R^n}f(x+\rho y)\,\nabla^\alpha\chi_{B_1}(y)\,dy,
\end{align*}
we can compute
\begin{align*}
(n+\beta-\alpha)&\int_r^{+\infty}\frac{D^\alpha f(B_\rho(x))}{\rho^{n+\beta-\alpha+1}}\,d\rho
-\frac{D^\alpha f(B_r(x))}{r^{n+\beta-\alpha}}\\
&= - (n+\beta-\alpha)\int_r^{+\infty}\frac{1}{\rho^{\beta+1}} \int_{\R^n}f(x+\rho y)\,\nabla^\alpha\chi_{B_1}(y)\,dy \,d\rho \\
& \quad + \frac{1}{r^{\beta}} \int_{\R^n}f(x+r y)\,\nabla^\alpha\chi_{B_1}(y)\,dy \\
&=\int_{\R^{n}} \left ( \frac{f(x + r y)}{r^{\beta}} - (n+\beta-\alpha) \int_{r}^{+ \infty}  \frac{f(x + \rho y)}{\rho^{\beta+1}} \, d \rho \right ) \nabla^{\alpha} \chi_{B_{1}}(y) \, dy
\end{align*}
for all $x\in A$. Hence, we have
\begin{align*}
\int_{A}\bigg|\int_{\{|h|>r\}}\frac{\nabla^{\alpha} f(x+h)}{|h|^{n+\beta-\alpha}} &\, dh\, \bigg| \,dx
\le \int_{\R^{n}}\bigg|\int_{\{|h|>r\}}\frac{\nabla^{\alpha} f(x+h)}{|h|^{n+\beta-\alpha}}\,dh\,\bigg|\,dx \\
& \le \int_{\R^{n}} \int_{\R^{n}} \frac{|f(x + r y)|}{r^{\beta}} \, |\nabla^{\alpha} \chi_{B_{1}}(y)| \, dx \, dy \\
& \quad + (n+\beta-\alpha) \int_{\R^{n}} \int_{r}^{+ \infty} \int_{\R^{n}} \frac{|f(x + \rho y)|}{\rho^{\beta + 1}}\, |\nabla^{\alpha} \chi_{B_{1}}(y)| \, dx \, d \rho \, dy\\
&=\frac{\omega_{n,\alpha}(n + 2 \beta -\alpha)}{\beta}\,\|f\|_{L^1(\R^n)}\,r^{-\beta}.
\end{align*}
Thus~\eqref{eq:Davila_estimate} follows for all $f\in C^\infty_c(\R^n)$ and~$r>0$. 

\smallskip

\textit{Proof of~\eqref{eq:Davila_estimate}, Step~2}. Let $f\in BV^\alpha(\R^n)$ and fix $r>0$. By~\cite{CS18}*{Theorem~3.8}, we find $(f_k)_{k\in\N}\subset C^\infty_c(\R^n)$ such that $f_k\to f$ in~$L^1(\R^n)$ and $|D^\alpha f_k|(\R^n)\to|D^\alpha f|(\R^n)$ as $k\to+\infty$. By Step~1, we have that
\begin{equation}\label{eq:Davila_estimate_smooth_alpha_beta}
\|\nabla^\beta f_k\|_{L^1(A;\,\R^n)}
\le\frac{\mu_{n,1+\beta-\alpha}}{n+\beta-\alpha}\left(\frac{n\omega_n|D^\alpha f_k|(\closure[-.5]{A_r})}{\alpha-\beta}\,r^{\alpha-\beta}+\frac{\omega_{n,\alpha}(n + 2 \beta -\alpha)}{\beta}\,\|f_k\|_{L^1(\R^n)}\,r^{-\beta}\right)
\end{equation}
for all $k\in\N$. We have that
\begin{equation}\label{eq:claim_weak_conv_BV_alpha_beta}
(\nabla^\beta f_k)\,\Leb{n}\weakto(\nabla^\beta f)\,\Leb{n}
\quad
\text{as $k\to+\infty$}.
\end{equation}
This can be proved arguing as in the proof of~\eqref{eq:claim_weak_conv_BV} using~\eqref{eq:weak_frac_grad_loc_estimate_alpha_beta}. At this point the proof goes like that of \cref{lem:nabla_alpha_f_L1_and_L1_loc}\eqref{item:lem:nabla_alpha_f_loc_1} and
we thus leave the details to the reader.
\end{proof}

\section{Asymptotic behavior of fractional \texorpdfstring{$\alpha$}{alpha}-variation as \texorpdfstring{$\alpha\to1^-$}{alpha tends to 1-}}
\label{sec:asymptotic_alpha_to_1}

\subsection{Convergence of \texorpdfstring{$\nabla^\alpha$}{nablaˆalpha} and \texorpdfstring{$\div^\alpha$}{divˆalpha} as \texorpdfstring{$\alpha\to1^-$}{alpha tends to 1-}}

We begin with the following simple result about the asymptotic behavior of the constant $\mu_{n,\alpha}$ as $\alpha\to1^-$.

\begin{lemma}\label{lem:mu_alpha_bound} 
Let $n\in\N$. We have
\begin{equation}\label{eq:mu_alpha_above_est}
\frac{\mu_{n,\alpha}}{1-\alpha}
\le\pi^{- \frac{n}{2}} \sqrt{6} \, \frac{\Gamma \left ( \frac{n}{2} + 1\right )}{\Gamma \left ( \frac{3}{2} \right )}
=:C_n
\qquad
\forall\alpha\in(0,1)
\end{equation}
and
\begin{equation}\label{eq:def_mu_alpha_asym} 
\lim_{\alpha\to1^-}\frac{\mu_{n,\alpha}}{1-\alpha}=\omega_n^{-1}.
\end{equation}
\end{lemma}

\begin{proof}
Since $\Gamma(1) = 1$ and $\Gamma(1 + x) = x \, \Gamma(x)$ for $x > 0$ (see~\cite{A64}), we have $\Gamma(x) \sim x^{-1}$ as $x \to 0^+$. Thus as $\alpha\to1^-$ we find
\begin{equation*} 
\mu_{n, \alpha} 
= 2^{\alpha} \pi^{- \frac{n}{2}} \frac{\Gamma\left ( \frac{n + \alpha + 1}{2} \right )}{\Gamma\left ( \frac{1 - \alpha}{2} \right )} 
\sim \pi^{- \frac{n}{2}}\,(1 - \alpha)\,\Gamma\left ( \frac{n}{2} + 1 \right ) 
= \omega_{n}^{-1} (1 - \alpha)
\end{equation*}
and~\eqref{eq:def_mu_alpha_asym} follows. 
Since~$\Gamma$ is log-convex on~$(0,+\infty)$ (see~\cite{A64}), for all $x > 0$ and $a \in (0, 1)$ we have 
\begin{equation*} 
\Gamma(x + a) 
= \Gamma((1 - a)x + a(x + 1)) 
\le \Gamma(x)^{1 - a}\, \Gamma(x + 1)^{a} 
= x^{a}\, \Gamma(x). 
\end{equation*}
For $x = \frac{n}{2}$ and $a = \frac{\alpha + 1}{2}$, we can estimate
\begin{equation*} 
\Gamma\left ( \frac{n + \alpha + 1}{2} \right ) 
\le \left ( \frac{n}{2} \right )^{\frac{\alpha + 1}{2}} \Gamma \left ( \frac{n}{2} \right ) 
\le \max\left \{\frac{n}{2}, 1 \right \} \Gamma \left ( \frac{n}{2} \right ) \le n \Gamma \left ( \frac{n}{2} \right ) = 2 \Gamma \left ( \frac{n}{2} + 1 \right )
\end{equation*}
for all $n\ge1$. For $x = 1 + \frac{1 - \alpha}{2}$ and $a = \frac{\alpha}{2}$, we can estimate
\begin{equation*} 
\Gamma \left ( \frac{3}{2} \right ) 
\le \left (1 +  \frac{1 - \alpha}{2} \right )^{\frac{\alpha}{2}} \Gamma\left ( 1+  \frac{1 - \alpha}{2} \right ) 
\le \sqrt{ \frac{3}{2}}\, \frac{1 - \alpha}{2}\, \Gamma\left (\frac{1 - \alpha}{2} \right ). 
\end{equation*}
We thus get
\begin{equation*} 
\mu_{n, \alpha} (1 - \alpha)^{-1} 
=  2^{\alpha - 1} \, \pi^{- \frac{n}{2}}\, \frac{\Gamma\left ( \frac{n + \alpha + 1}{2} \right )}{\Gamma\left ( \frac{1 - \alpha}{2} + 1 \right )} 
\le \pi^{- \frac{n}{2}} \, \sqrt{ \frac{3}{2}} \, \frac{2 \Gamma \left ( \frac{n}{2} + 1 \right )}{\Gamma \left ( \frac{3}{2} \right )}
\end{equation*}
and~\eqref{eq:mu_alpha_above_est} follows. 
\end{proof}

In the following technical result, we show that the constant $C_{n, \alpha, U}$ defined in~\eqref{eq:constant_estimate} is uniformly bounded as $\alpha \to 1^{-}$ in terms of the volume and the diameter of the bounded open set~$U\subset\R^n$.

\begin{lemma}[Uniform upper bound on~$C_{n, \alpha, U}$ as $\alpha\to1^-$]\label{lem:unif_bound_C_n_alpha_U} 
Let $n \in \N$ and $\alpha \in (\frac{1}{2}, 1)$. Let $U\subset\R^n$ be bounded open set. If $C_{n, \alpha, U}$ is as in~\eqref{eq:constant_estimate}, then
\begin{equation}\label{eq:uniform_estimate_C_n_alpha_U} 
C_{n, \alpha, U} 
\le \frac{n\omega_n C_{n}}{\left ( n - \frac{1}{2} \right )} \left ( \frac{n}{\left ( n - \frac{1}{2} \right )} \, \max\set*{1,\frac{|U|}{\omega_{n}}}^{\frac{1}{n}} +\max\set*{1,\sqrt{\diam(U)}} \right) 
=: \kappa_{n, U}, 
\end{equation}
where $C_{n}$ is as in~\eqref{eq:mu_alpha_above_est}.
\end{lemma}

\begin{proof} 
By~\eqref{eq:mu_alpha_above_est}, for all $\alpha\in(\frac{1}{2},1)$ we have
\begin{equation*} 
\frac{ n \, \mu_{n, \alpha}}{(n + \alpha - 1) (1 - \alpha)} 
\le \frac{n \, C_{n}}{n + \alpha - 1} 
\le \frac{n \, C_{n}}{n - \frac{1}{2}}.
\end{equation*}
Since $t^{1 - \alpha} \le \max\set*{1, \sqrt{t}}$ for any $t\ge0$ and $\alpha\in(\frac{1}{2},1)$, we have 
\begin{equation*}
\omega_n(\diam(U))^{1 - \alpha}\le\omega_n\max\set*{1,\sqrt{\diam(U)}}
\end{equation*}
and
\begin{align*} 
\left(\frac{n \omega_{n}}{n+\alpha-1}\right)^\frac{n + \alpha - 1}{n}|U|^\frac{1 - \alpha}{n} 
&=\frac{n\omega_n}{n+\alpha-1}\left(\frac{|U|(n+\alpha-1)}{n\omega_n}\right)^{\frac{1-\alpha}{n}}
\le \frac{n\omega_n}{\left ( n - \frac{1}{2} \right )} \max\set*{1,\frac{|U|}{\omega_{n}}}^{\frac{1}{n}}.
\end{align*}
Combining these inequalities, we get the conclusion.
\end{proof}

As consequence of \cref{prop:frac_div_repr} and \cref{lem:unif_bound_C_n_alpha_U}, we prove that $\nabla^\alpha$ and $\div^\alpha$ converge pointwise to~$\nabla$ and~$\div$ respectively as $\alpha \to 1^{-}$.

\begin{proposition}\label{res:pointwise_converg}
If $f\in C^1_c(\R^n)$, then for all $x\in\R^n$ we have
\begin{equation}\label{eq:pointwise_converg_riesz}
\lim_{\alpha \to 0^-}I_{\alpha}f(x)=f(x).
\end{equation}
As a consequence, if $f \in C^{2}_{c}(\R^{n})$ and $\phi \in C^{2}_c(\R^{n}; \R^{n})$, then for all $x\in\R^n$ we have
\begin{equation}\label{eq:pointwise_converg_nabla_alpha_div_alpha}
\lim_{\alpha \to 1^{-}}\nabla^{\alpha} f(x)=\nabla f(x),
\qquad
\lim_{\alpha \to 1^{-}}\div^{\alpha} \phi(x)=\div \phi(x).
\end{equation}
\end{proposition}

\begin{proof} 
Let $f \in C^{1}_{c}(\R^{n})$ and fix $x\in\R^n$. Writing~\eqref{eq:frac_nabla_repres} in spherical coordinates, we find
\begin{equation*}
I_{\alpha}f(x) = \frac{\mu_{n, 1-\alpha}}{n - \alpha} \lim_{\delta \to 0} \int_{\partial B_1} \int_{\delta}^{+ \infty}  \rho^{-1+ \alpha} f(x + \rho v)\,d \rho\,  d \Haus{n - 1}(v).
\end{equation*}
Since $f\in C^{1}_{c}(\R^{n})$, for each fixed $v\in\de B_1$ we can integrate by parts in the variable~$\rho$ and get
\begin{align*} 
\int_{\delta}^{+ \infty} \rho^{-1+ \alpha}f(x + \rho v)\, d \rho
&= \left[\frac{\rho^{\alpha}}{\alpha}\, f(x + \rho v)\right]_{\rho=\delta}^{\rho\to+\infty}
- \frac{1}{\alpha} \int_{\delta}^{+\infty} \rho^{\alpha}\, \de_\rho(f(x + \rho v))\, d \rho\\
& = -\frac{\delta^{\alpha}}{\alpha} \, f(x + \delta v)
- \frac{1}{\alpha} \int_{\delta}^{+\infty} \rho^{\alpha}\, \de_\rho(f(x + \rho v))\, d \rho.
\end{align*}
Clearly, we have 
\begin{equation*} 
\lim_{\delta\to0^+}\delta^{\alpha} \int_{\partial B_1} f(x + \delta v) \, d \Haus{n - 1}(v) =0.
\end{equation*}
Thus, by Fubini's Theorem, we conclude that
\begin{equation}\label{eq:C2_frac_grad_repr}
I_{\alpha} f(x) 
= - \frac{\mu_{n, 1-\alpha}}{\alpha(n - \alpha)} \int_{0}^{\infty} \int_{\de B_1} \rho^{\alpha}\, \de_\rho(f(x + \rho v))\, d \Haus{n - 1}(v) \, d \rho.
\end{equation}
Since~$f$ has compact support and recalling~\eqref{eq:def_mu_alpha_asym}, we can pass to the limit in~\eqref{eq:C2_frac_grad_repr} and get
\begin{equation*} 
\lim_{\alpha\to0^+}I_{\alpha} f(x) 
= - \frac{1}{n \omega_{n}} \int_{\de B_1} \int_{0}^{\infty} \de_\rho( f(x + \rho v) ) \, d \rho \, d \Haus{n - 1}(v)
= f(x),
\end{equation*}
proving~\eqref{eq:pointwise_converg_riesz}. The pointwise limits in~\eqref{eq:pointwise_converg_nabla_alpha_div_alpha} immediately follows by \cref{prop:frac_div_repr}.
\end{proof}

In the following crucial result, we improve the pointwise convergence obtained in \cref{res:pointwise_converg} to strong convergence in~$L^p(\R^n)$ for all $p\in[1,+\infty]$.

\begin{proposition}\label{res:L1_converg}
Let $p\in[1,+\infty]$. If $f \in C^{2}_{c}(\R^{n})$ and $\phi \in C^{2}_c(\R^{n}; \R^{n})$, then
\begin{equation*}
\lim_{\alpha \to 1^{-}}\|\nabla^{\alpha} f-\nabla f\|_{L^p(\R^n;\R^n)}=0,
\qquad
\lim_{\alpha \to 1^{-}}\|\div^{\alpha} \phi-\div \phi\|_{L^p(\R^n)}=0.
\end{equation*}
\end{proposition}

\begin{proof}
Let $f \in C^{2}_{c}(\R^{n})$. Since
\begin{equation*}
\int_{B_1}\frac{dy}{|y|^{n+\alpha-1}}=n\omega_n\int_0^1\frac{d\rho}{\rho^\alpha}=\frac{n\omega_n}{1-\alpha},
\end{equation*}
for all $x\in\R^n$ we can write
\begin{equation*}
\frac{n\omega_n\mu_{n,\alpha}}{(1-\alpha)(n+\alpha-1)}\,\nabla f(x)
=\frac{\mu_{n,\alpha}}{n+\alpha-1}\int_{B_1}\frac{\nabla f(x)}{|y|^{n+\alpha-1}}\,dy.
\end{equation*}
Therefore, by~\eqref{eq:frac_nabla_repres}, we have
\begin{align*}
\nabla^\alpha f(x)&-\frac{n\omega_n\mu_{n,\alpha}}{(1-\alpha)(n+\alpha-1)}\,\nabla f(x)\\
&=\frac{\mu_{n,\alpha}}{n+\alpha-1}\left(\int_{B_1}\frac{\nabla f(x+y)-\nabla f(x)}{|y|^{n+\alpha-1}}\,dy
+\int_{\R^n\setminus B_1}\frac{\nabla f(x+y)}{|y|^{n+\alpha-1}}\,dy\right)
\end{align*}
for all $x\in\R^n$. We now distinguish two cases. 

\smallskip

\textit{Case~1: $p\in[1,+\infty)$}. Using the elementary inequality $|v+w|^p\le2^{p-1}(|v|^p+|w|^p)$ valid for all $v,w\in\R^n$, we have
\begin{align*}
\int_{\R^n}\bigg|\,\nabla^\alpha f(x)-&\frac{n\omega_n\mu_{n,\alpha}}{(1-\alpha)(n+\alpha-1)}\,\nabla f(x)\,\bigg|^p dx\\
&\le\frac{2^{p-1}\mu_{n,\alpha}}{n+\alpha-1}\int_{\R^n}\bigg|\int_{B_1}\frac{\nabla f(x+y)-\nabla f(x)}{|y|^{n+\alpha-1}}\,dy\,\bigg|^p dx\\
&\quad+\frac{2^{p-1}\mu_{n,\alpha}}{n+\alpha-1}\int_{\R^n}\bigg|\int_{\R^n\setminus B_1}\frac{\nabla f(x+y)}{|y|^{n+\alpha-1}}\,dy\,\bigg|^p dx.
\end{align*}
We now estimate the two double integrals appearing in the right-hand side separately.

For the first double integral, as in the proof of \cref{res:pointwise_converg}, we pass in spherical coordinates to get
\begin{equation}\label{eq:splitting}
\begin{split}
\int_{B_1}\frac{\nabla f(x+y)-\nabla f(x)}{|y|^{n+\alpha-1}}\,dy
&=\int_{\de B_1}\int_0^1\rho^{-\alpha}\left(\nabla f(x+\rho v)-\nabla f(x)\right)\,d\rho\,d\Haus{n-1}(v)\\
&=\frac{1}{1-\alpha}\int_{\de B_1}\left(\nabla f(x+v)-\nabla f(x)\right)\,d\Haus{n-1}(v)\\
&\quad-\int_{\de B_1}\int_0^1\frac{\rho^{1-\alpha}}{1-\alpha}\,\de_\rho(\nabla f(x+\rho v))\,d\rho\,d\Haus{n-1}(v)	
\end{split}	
\end{equation}
for all $x\in\R^n$. Hence, by~\eqref{eq:def_mu_alpha_asym}, we find
\begin{align*}
\lim_{\alpha\to1^-}\frac{\mu_{n,\alpha}}{(1-\alpha)(n+\alpha-1)}&\int_{\de B_1}\left(\nabla f(x+v)-\nabla f(x)\right)\,d\Haus{n-1}(v)\\
&=\frac{1}{n\omega_n}\int_{\de B_1}\left(\nabla f(x+v)-\nabla f(x)\right)\,d\Haus{n-1}(v)
\end{align*}
and
\begin{align*}
\lim_{\alpha\to1^-}\frac{\mu_{n,\alpha}}{(1-\alpha)(n+\alpha-1)}&\int_{\de B_1}\int_0^1\rho^{1-\alpha}\,\de_\rho(\nabla f(x+\rho v))\,d\rho\,d\Haus{n-1}(v)\\
&=\frac{1}{n\omega_n}\int_{\de B_1}\int_0^1\de_\rho(\nabla f(x+\rho v))\,d\rho\,d\Haus{n-1}(v)\\
&=\frac{1}{n\omega_n}\int_{\de B_1}\left(\nabla f(x+v)-\nabla f(x)\right)\,d\Haus{n-1}(v)
\end{align*}
for all $x\in\R^n$. Therefore, we get
\begin{equation*}
\lim_{\alpha\to1^-}\frac{\mu_{n,\alpha}}{n+\alpha-1}\int_{B_1}\frac{\nabla f(x+y)-\nabla f(x)}{|y|^{n+\alpha-1}}\,dy=0
\end{equation*}
for all $x\in\R^n$. Recalling~\eqref{eq:mu_alpha_above_est}, we also observe that 
\begin{equation*}
\frac{\mu_{n,\alpha}}{n+\alpha-1}\frac{|\nabla f(x+y)-\nabla f(x)|}{|y|^{n+\alpha-1}}\le C_n\frac{\abs*{\nabla f(x+y)-\nabla f(x)}}{|y|^n}
\end{equation*}
for all $\alpha\in(0,1)$, $x\in\R^n$ and $y\in B_1$. Moreover, letting $R>0$ be such that $\supp f\subset B_R$, we can estimate
\begin{equation*}
\int_{B_1}\frac{\abs*{\nabla f(x+y)-\nabla f(x)}}{|y|^n}\,dy
\le n\omega_n \,\|\nabla^2 f\|_{L^\infty\left(\R^n;\,\R^{n^2}\right)}\,\chi_{B_{R+1}}(x)
\end{equation*} 
for all $x\in\R^n$, so that
\begin{equation*}
x\mapsto\left(\int_{B_1}\frac{\abs*{\nabla f(x+y)-\nabla f(x)}}{|y|^n}\,dy\right)^p\in L^1(\R^n).
\end{equation*}
In conclusion, applying Lebesgue's Dominated Convergence Theorem, we find
\begin{equation*}
\lim_{\alpha\to1^-}\frac{\mu_{n,\alpha}}{n+\alpha-1}\int_{\R^n}\abs*{\int_{B_1}\frac{\nabla f(x+y)-\nabla f(x)}{|y|^{n+\alpha-1}}\,dy\,}^p dx=0.
\end{equation*}
For the second double integral, note that
\begin{equation*}
\int_{\R^n\setminus B_1}\frac{\nabla f(x+y)}{|y|^{n+\alpha-1}}\,dy
=\int_{\R^n\setminus B_1}\frac{\nabla(f(x+y)-f(x))}{|y|^{n+\alpha-1}}\,dy
\end{equation*}
for all $x\in\R^n$. Now let $R>0$. Integrating by parts, we have that
\begin{align*}
\int_{B_R\setminus B_1}\frac{\nabla(f(x+y)-f(x))}{|y|^{n+\alpha-1}}\,dy
&=(n+\alpha-1)\int_{B_R\setminus B_1}\frac{y\,(f(x+y)-f(x))}{|y|^{n+\alpha+1}}\,dy\\
&\quad+\frac{1}{R^{n+\alpha-1}}\int_{\de B_R}(f(x+y)-f(x))\,d\Haus{n-1}(y)\\
&\quad-\int_{\de B_1}(f(x+y)-f(x))\,d\Haus{n-1}(y)
\end{align*}
for all $x\in\R^n$. Since
\begin{align*}
\int_{\R^n\setminus B_R}\frac{|f(x+y)-f(x)|}{|y|^{n+\alpha}}\,dy
\le\frac{2n\omega_n}{\alpha R^\alpha}\|f\|_{L^\infty(\R^n)}
\end{align*}
and
\begin{align*}
\frac{1}{R^{n+\alpha-1}}\int_{\de B_R}|f(x+y)-f(x)|\,d\Haus{n-1}(y)
\le\frac{2n\omega_n}{R^\alpha}\|f\|_{L^\infty(\R^n)}
\end{align*}
for all $R>0$, we conclude that
\begin{equation}\label{eq:splitting2}
\begin{split}
\int_{\R^n\setminus B_1}\frac{\nabla f(x+y)}{|y|^{n+\alpha-1}}\,dy
&=\lim_{R\to+\infty}\int_{B_R\setminus B_1}\frac{\nabla f(x+y)}{|y|^{n+\alpha-1}}\,dy\\
&=(n+\alpha-1)\int_{\R^n\setminus B_1}\frac{y\,(f(x+y)-f(x))}{|y|^{n+\alpha+1}}\,dy\\
&\quad-\int_{\de B_1}(f(x+y)-f(x))\,d\Haus{n-1}(y)
\end{split}	
\end{equation}
for all $x\in\R^n$. Hence, by Minkowski's Integral Inequality (see~\cite{S70}*{Section~A.1}, for example), we can estimate
\begin{align*}
\bigg\|\int_{\R^n\setminus B_1}\frac{\nabla f(\cdot+y)}{|y|^{n+\alpha-1}}\,dy\,\bigg\|_{L^p(\R^n;\,\R^n)} 
&\le(n+\alpha-1)\bigg\|\int_{\R^n\setminus B_1}\frac{|f(\cdot+y)-f(\cdot)|}{|y|^{n+\alpha}}\,dy\,\bigg\|_{L^p(\R^n)}\\
&\quad+\bigg\|\int_{\de B_1}|f(\cdot+y)-f(\cdot)|\,d\Haus{n-1}(y)\,\bigg\|_{L^p(\R^n)}\\
&\le\frac{n+2\alpha-1}{\alpha}\,2n\omega_n\|f\|_{L^p(\R^n)}.
\end{align*}
Thus, by~\eqref{eq:def_mu_alpha_asym}, we get that
\begin{equation*}
\lim_{\alpha\to1^-}\frac{\mu_{n,\alpha}}{n+\alpha-1}\int_{\R^n}\abs*{\int_{\R^n\setminus B_1}\frac{\nabla f(x+y)}{|y|^{n+\alpha-1}}\,dy\,}^p\,dx=0.
\end{equation*}

\smallskip

\textit{Case~2: $p=+\infty$}. 
We have
\begin{align*}
\sup_{x\in\R^n}&\bigg|\nabla^\alpha f(x)-\frac{n\omega_n\mu_{n,\alpha}}{(1-\alpha)(n+\alpha-1)}\,\nabla f(x)\bigg|\\
&\le\frac{\mu_{n,\alpha}}{n+\alpha-1}\left(\sup_{x\in\R^n}\abs*{\int_{B_1}\frac{\nabla f(x+y)-\nabla f(x)}{|y|^{n+\alpha-1}}\,dy\,}
+\sup_{x\in\R^n}\abs*{\int_{\R^n\setminus B_1}\frac{\nabla f(x+y)}{|y|^{n+\alpha-1}}\,dy\,}\right).
\end{align*}
Again we estimate the two integrals appearing in the right-hand side separately. We note that
\begin{align*}
\int_{\de B_1}\left(\nabla f(x+v)-\nabla f(x)\right)&\,d\Haus{n-1}(v)
-\int_{\de B_1}\int_0^1\rho^{1-\alpha}\,\de_\rho(\nabla f(x+\rho v))\,d\rho\,d\Haus{n-1}(v)\\
&=\int_{\de B_1}\int_0^1(1-\rho^{1-\alpha})\,\de_\rho(\nabla f(x+\rho v))\,d\rho\,d\Haus{n-1}(v),
\end{align*}
so that we can rewrite~\eqref{eq:splitting} as
\begin{align*}
\int_{B_1}\frac{\nabla f(x+y)-\nabla f(x)}{|y|^{n+\alpha-1}}\,dy
=\frac{1}{1-\alpha}\int_{\de B_1}\int_0^1(1-\rho^{1-\alpha})\,\de_\rho(\nabla f(x+\rho v))\,d\rho\,d\Haus{n-1}(v).
\end{align*}
Hence, we can estimate
\begin{align*}
\sup_{x\in\R^n}&\bigg|\int_{B_1}\frac{\nabla f(x+y)-\nabla f(x)}{|y|^{n+\alpha-1}}\,dy\,\bigg|\\
&\le\frac{1}{1-\alpha}\int_{\de B_1}\int_0^1(1-\rho^{1-\alpha})\sup_{x\in\R^n}|\de_\rho(\nabla f(x+\rho v))|\,d\rho\,d\Haus{n-1}(v)\\
&\le\frac{1}{2-\alpha}\,n\omega_n\,\|\nabla^2 f\|_{L^\infty\left(\R^n;\,\R^{n^2}\right)},
\end{align*}
so that
\begin{align*}
\lim_{\alpha\to1^-}\frac{\mu_{n,\alpha}}{n+\alpha-1}\sup_{x\in\R^n}\abs*{\int_{B_1}\frac{\nabla f(x+y)-\nabla f(x)}{|y|^{n+\alpha-1}}\,dy\,}=0.
\end{align*}
For the second integral, by~\eqref{eq:splitting2} we can estimate
\begin{align*}
&\sup_{x\in\R^n}\bigg|\int_{\R^n\setminus B_1}\frac{\nabla f(x+y)}{|y|^{n+\alpha-1}}\,dy\,\bigg|\,dx \\
&\qquad\le(n+\alpha-1)\sup_{x\in\R^n}\bigg|\int_{\R^n\setminus B_1}\frac{|f(x+y)-f(x)|}{|y|^{n+\alpha}}\,dy\,\bigg|\\
&\qquad\quad+\sup_{x\in\R^n}\bigg|\int_{\de B_1}|f(x+y)-f(x)|\,d\Haus{n-1}(y)\,\bigg|\\
&\qquad\le\frac{n+2\alpha-1}{\alpha}\,2n\omega_n\|f\|_{L^\infty(\R^n)}.
\end{align*}
Thus, by~\eqref{eq:def_mu_alpha_asym}, we get that
\begin{align*}
\lim_{\alpha\to1^-}\frac{\mu_{n,\alpha}}{n+\alpha-1}\sup_{x\in\R^n}\abs*{\int_{\R^n\setminus B_1}\frac{\nabla f(x+y)}{|y|^{n+\alpha-1}}\,dy\,}=0.
\end{align*}

\smallskip

We can now conclude the proof. Again recalling~\eqref{eq:def_mu_alpha_asym}, we thus find that
\begin{align*}
&\lim_{\alpha\to1^-}\|\nabla^{\alpha} f-\nabla f\|_{L^p(\R^n;\R^n)}\\
&\le\lim_{\alpha\to1^-}\left\|\nabla^{\alpha} f-\frac{n\omega_n\mu_{n,\alpha}}{(1-\alpha)(n+\alpha-1)}\,\nabla f\,\right\|_{L^p(\R^n;\R^n)}\\
&\quad+\|\nabla f\|_{L^p(\R^n;\R^n)}\lim_{\alpha\to1^-} \left | \frac{n\omega_n\mu_{n,\alpha}}{(1-\alpha)(n+\alpha-1)}-1\right|=0
\end{align*}
for all $p\in[1,+\infty]$ and the conclusion follows. The $L^p$-convergence of~$\div^\alpha\phi$ to $\div\phi$ as~$\alpha\to1^-$ for all $p\in[1,+\infty]$ follows by a similar argument and is left to the reader.
\end{proof}

\begin{remark}
Note that the conclusion of \cref{res:L1_converg} still holds if instead one assumes that $f\in\mathscr{S}(\R^n)$ and $\phi\in\mathscr{S}(\R^n;\R^n)$, where $\mathscr{S}(\R^n;\R^m)$ is the space of $m$-vector-valued Schwartz functions. We leave the proof of this assertion to the reader.
\end{remark}

\subsection{Weak convergence of \texorpdfstring{$\alpha$}{alpha}-variation as \texorpdfstring{$\alpha\to1^-$}{alpha tends to 1-}}

In \cref{res:weak_conv_nabla_alpha_to_1} below, we prove that the fractional $\alpha$-variation weakly converges to the standard variation as $\alpha\to1^-$ for functions either in~$BV(\R^{n})$ or in~$BV_{\loc}(\R^{n})\cap L^{\infty}(\R^{n})$. In the proof of \cref{res:weak_conv_nabla_alpha_to_1}, we are going to use the following technical result.

\begin{lemma}\label{res:nabla_alpha_f_L1_loc_uniform}
There exists a dimensional constant $c_n>0$ with the following property. If $f \in L^{\infty}(\R^{n})\cap BV_{\loc}(\R^{n})$, then
\begin{equation}\label{eq:nabla_alpha_f_L1_loc_uniform}
\| \nabla^{\alpha} f \|_{L^{1}(B_{R};\, \R^{n})} 
\le c_n\left(R^{1-\alpha}|Df|(B_{3R}) 
+ R^{n-\alpha}\,\|f\|_{L^{\infty}(\R^{n})}\right) 
\end{equation} 
for all $R>0$ and $\alpha\in(\frac{1}{2},1)$.
\end{lemma}

\begin{proof}
Since $\Gamma(x)\sim x^{-1}$ as $x\to0^+$ (see~\cite{A64}), inequality~\eqref{eq:nabla_alpha_f_L1_loc_uniform} follows immediately  combining~\eqref{eq:nabla_alpha_f_L1_loc_BV_loc} with \cref{lem:mu_alpha_bound}.
\end{proof}

\begin{theorem}\label{res:weak_conv_nabla_alpha_to_1}
If either $f \in BV(\R^{n})$ or $f \in BV_{\loc}(\R^{n})\cap L^{\infty}(\R^{n})$, then 
\begin{equation*}
D^\alpha f\weakto Df
\quad\text{as $\alpha\to1^-$}. 
\end{equation*}
\end{theorem}

\begin{proof}
We divide the proof in two steps.

\smallskip

\textit{Step~1}. Assume $f \in BV(\R^{n})$. By~\cite{CS18}*{Theorem~3.18}, we have
 \begin{equation*}
\int_{\R^n}\phi\cdot\nabla^\alpha f\,dx
=-\int_{\R^n} f\,\div^\alpha\phi\,dx	
\end{equation*}
for all $\phi\in\Lip_c(\R^n;\R^n)$. Thus, given $\phi\in C^2_c(\R^n;\R^n)$, recalling \cref{res:pointwise_converg} and the estimates~\eqref{eq:frac_div_repr_Lip_estimate} and~\eqref{eq:uniform_estimate_C_n_alpha_U}, by Lebesgue's Dominated Convergence Theorem we get that
\begin{equation*}
\lim_{\alpha\to1^-}\int_{\R^n}\phi\cdot\nabla^\alpha f\,dx
=-\lim_{\alpha\to1^-}\int_{\R^n} f\,\div^\alpha\phi\,dx
=-\int_{\R^n} f\,\div\phi\,dx
=\int_{\R^n} \phi\cdot dDf.
\end{equation*}
To achieve the same limit for any $\phi\in C_c^0(\R^n;\R^n)$, one just need  to exploit~\eqref{eq:weak_frac_grad_loc_estimate} and the uniform estimate~\eqref{eq:uniform_estimate_C_n_alpha_U} in \cref{lem:unif_bound_C_n_alpha_U}, and  argue as in Step~2 of the proof of~\eqref{eq:Davila_estimate}. 
We leave the details to the reader.

\smallskip

\textit{Step~2}. Assume $f \in BV_{\loc}(\R^{n})\cap L^{\infty}(\R^{n})$. By \cref{lem:nabla_alpha_f_L1_and_L1_loc}\eqref{item:lem:nabla_alpha_f_loc_3}, we know that $D^\alpha f=\nabla^\alpha f\Leb{n}$ with $\nabla^\alpha f\in L^1_{\loc}(\R^n;\R^n)$. By \cref{res:L1_converg}, we get that
\begin{equation*}
\lim_{\alpha\to1^-}\abs*{\int_{\R^n}\phi\cdot\nabla^\alpha f\,dx-\int_{\R^n}\phi\cdot dDf\,}
\le\|f\|_{L^\infty(\R^n)}\lim_{\alpha\to1^-}\|\div^\alpha\phi-\div\phi\|_{L^1(\R^n;\,\R^n)}=0
\end{equation*}
for all $\phi\in C^2_c(\R^n;\R^n)$. 
To achieve the same limit for any $\phi\in C_c^0(\R^n;\R^n)$, one just need to exploit~\eqref{eq:nabla_alpha_f_L1_loc_uniform} and argue as in Step~1.
We leave the details to the reader.
\end{proof}

We are now going to improve the weak convergence of the fractional $\alpha$-variation obtained in \cref{res:weak_conv_nabla_alpha_to_1} by establishing the weak convergence also of the total fractional $\alpha$-variation as~$\alpha\to1^-$, see \cref{res:weak_conv_tot_var} below. To do so, we need the following preliminary result.

\begin{lemma}\label{res:weak_conv_riesz}
Let $\mu\in\mathscr{M}(\R^n; \R^m)$. We have $(I_{\alpha}\mu)\Leb{n}\weakto\mu$ as $\alpha\to0^+$.
\end{lemma}

\begin{proof}
Since Riesz potential is a linear operator and thanks to Hahn--Banach Decomposition Theorem, without loss of generality we can assume that $\mu$ is a nonnegative finite Radon measure.

Let now $\phi\in C_c^1(\R^n)$ and let $U\subset\R^n$ be a bounded open set such that $\supp\phi\subset U$. We have that $\|I_\alpha|\phi|\|_{L^\infty(\R^n)}\le\kappa_{n,U}\|\phi\|_{L^\infty(\R^n)}$ for all $\alpha\in(0,\frac{1}{2})$ by~\cite{CS18}*{Lemma~2.4} and \cref{lem:unif_bound_C_n_alpha_U}. Thus, by~\eqref{eq:pointwise_converg_riesz}, Fubini's Theorem and Lebesgue's Dominated Convergence Theorem, we get that
\begin{equation*}
\lim_{\alpha\to0^+}\int_{\R^n} \phi\,I_\alpha\mu\,dx
=\lim_{\alpha\to0^+}\int_{\R^n} I_\alpha\phi\,d\mu
=\int_{\R^n}\phi\,d\mu.
\end{equation*}
To achieve the same limit for any $\phi\in C_c^0(\R^n)$, one just need to exploit~\cite{CS18}*{Lemma~2.4} and~\eqref{eq:uniform_estimate_C_n_alpha_U} and argue as in Step~2 of the proof of~\eqref{eq:Davila_estimate}.
We leave the details to the reader.
\end{proof}

\begin{theorem}\label{res:weak_conv_tot_var}
If either $f\in BV(\R^n)$ or $f\in bv(\R^n)\cap L^\infty(\R^n)$, then 
\begin{equation}\label{eq:weak_conv_tot_var}
|D^\alpha f|\weakto|Df|
\quad
\text{as $\alpha\to1^-$}.	
\end{equation}
Moreover, if $f\in BV(\R^n)$, then also
\begin{equation}\label{eq:conv_tot_mass}
\lim_{\alpha\to1^-}|D^\alpha f|(\R^n)
=|Df|(\R^n).
\end{equation}
\end{theorem}

\begin{proof}
We prove~\eqref{eq:weak_conv_tot_var} and~\eqref{eq:conv_tot_mass} separately.

\smallskip

\textit{Proof of~\eqref{eq:weak_conv_tot_var}}. 
By \cref{res:weak_conv_nabla_alpha_to_1}, we know that $D^\alpha f\weakto Df$ as $\alpha\to1^-$. By~\cite{M12}*{Proposition~4.29}, we thus have that
\begin{equation}\label{eq:always_liminf_open_sets}
|Df|(A)\le\liminf_{\alpha\to1^-}|D^\alpha f|(A)
\end{equation} 
for any open set $A\subset\R^n$. 
Now let $K\subset\R^n$ be a compact set.
By the representation formula~\eqref{eq:repres_formula_nabla_alpha_bv_bounded} in \cref{result:repres_formula_nabla_alpha_bv_bounded}, we can estimate
\begin{equation*}
|D^\alpha f|(K)
=\|\nabla^\alpha f\|_{L^1(K;\,\R^n)}
\le\|I_{1-\alpha}|Df|\|_{L^1(K)}
=(I_{1-\alpha}|Df|\,\Leb{n})(K).
\end{equation*}
Since $|Df|(\R^n)<+\infty$, by \cref{res:weak_conv_riesz} and \cite{M12}*{Proposition~4.26} we can conclude that
\begin{equation*}
\limsup_{\alpha \to 1^{-}} |D^{\alpha} f|(K) \le \limsup_{\alpha \to 1^{-}} (I_{1-\alpha}|Df|\,\Leb{n})(K) \le |Df|(K),
\end{equation*}
and so~\eqref{eq:weak_conv_tot_var} follows, thanks again to \cite{M12}*{Proposition~4.26}.

\smallskip

\textit{Proof of~\eqref{eq:conv_tot_mass}}.
Now assume $f\in BV(\R^n)$. By~\eqref{eq:Davila_estimate} applied with $A=\R^n$ and $r=1$, we have
\begin{equation*}
|D^\alpha f|(\R^n)
\le\frac{n\omega_n\,\mu_{n,\alpha}}{n+\alpha-1}\left(\frac{|Df|(\R^n)}{1-\alpha}+\frac{n+2\alpha-1}{\alpha}\,\|f\|_{L^1(\R^n)}\right).
\end{equation*} 
By~\eqref{eq:def_mu_alpha_asym}, we thus get that
\begin{equation}\label{eq:limsup_tot_mass}
\limsup_{\alpha\to1^-}|D^\alpha f|(\R^n)
\le|D f|(\R^n).
\end{equation}
Thus~\eqref{eq:conv_tot_mass} follows by combining~\eqref{eq:always_liminf_open_sets} for $A=\R^n$ with~\eqref{eq:limsup_tot_mass}.
\end{proof}

\begin{remark}
We notice that \cref{res:weak_conv_nabla_alpha_to_1} and \cref{res:weak_conv_tot_var}, in the case $f=\chi_E\in BV(\R^n)$ with $E\subset\R^n$ bounded, and \cref{res:weak_conv_alpha_p}, were already announced in~\cite{S16}*{Theorems~16 and~17}.
\end{remark}

Note that \cref{res:weak_conv_nabla_alpha_to_1} and \cref{res:weak_conv_tot_var} in particular apply to any $f\in W^{1,1}(\R^n)$. In the following result, by exploiting \cref{res:nabla_alpha_Lp}, we prove that a stronger property holds for any $f\in W^{1,p}(\R^n)$ with $p\in[1,+\infty)$.

\begin{theorem}\label{res:weak_conv_alpha_p}
Let $p\in[1,+\infty)$. If $f\in W^{1,p}(\R^n)$, then
\begin{equation}\label{eq:strong_conv_alpha_p}
\lim_{\alpha\to1^-}\|\nabla^\alpha_w f-\nabla_w f\|_{L^p(\R^n;\,\R^n)}=0.
\end{equation}
\end{theorem}

\begin{proof}
By \cref{res:nabla_alpha_Lp} we know that $f\in S^{\alpha,p}(\R^n)$ for any $\alpha \in (0, 1)$. We now assume $p\in(1,+\infty)$ and divide the proof in two steps.

\smallskip

\textit{Step~1}. We claim that 
\begin{equation}\label{eq:conv_tot_mass_p}
\lim_{\alpha\to1^-}\|\nabla^\alpha_w f\|_{L^p(\R^n;\,\R^n)}=\|\nabla_w f\|_{L^p(\R^n;\,\R^n)}.
\end{equation}
Indeed, on the one hand, by \cref{res:L1_converg}, we have
\begin{equation}\label{eq:int_by_parts_p_step_1}
\int_{\R^n}\phi\cdot\nabla_w f\,dx
=-\int_{\R^n}f\,\div\phi\,dx
=-\lim_{\alpha\to1^-}\int_{\R^n}f\,\div^\alpha\phi\,dx
=\lim_{\alpha\to1^-}\int_{\R^n}\phi\cdot\nabla_w^\alpha f\,dx
\end{equation}
for all $\phi\in C^\infty_c(\R^n;\R^n)$, so that
\begin{equation*}
\int_{\R^n}\phi\cdot\nabla_w f\,dx
\le\|\phi\|_{L^{\frac{p}{p-1}}(\R^n;\,\R^n)}\liminf_{\alpha\to1^-}\|\nabla_w^\alpha f\|_{L^p(\R^n;\,\R^n)}
\end{equation*}
for all $\phi\in C^\infty_c(\R^n;\R^n)$. We thus get that
\begin{equation}\label{eq:liminf_p}
\|\nabla_w f\|_{L^p(\R^n;\R^n)}
\le\liminf_{\alpha\to1^-}\|\nabla^\alpha_w f\|_{L^p(\R^n;\R^n)}.
\end{equation}
On the other hand, applying~\eqref{eq:Davila_estimate_p} with $A=\R^n$ and~$r=1$, we have
\begin{equation*}
\|\nabla^\alpha_w f\|_{L^p(\R^n;\,\R^n)}
\le\frac{n\omega_n\,\mu_{n,\alpha}}{n+\alpha-1}\left(\frac{\|\nabla_w f\|_{L^p(\R^n;\,\R^n)}}{1-\alpha}+\frac{n+2\alpha-1}{\alpha}\,\|f\|_{L^p(\R^n)}\right).
\end{equation*}
By~\eqref{eq:def_mu_alpha_asym}, we conclude that
\begin{equation}\label{eq:limsup_p}
\limsup_{\alpha\to1^-}\|\nabla^\alpha_w f\|_{L^p(\R^n;\,\R^n)}
\le\|\nabla_w f\|_{L^p(\R^n;\,\R^n)}.
\end{equation}
Thus,~\eqref{eq:conv_tot_mass_p} follows by combining~\eqref{eq:liminf_p} and~\eqref{eq:limsup_p}.

\smallskip

\textit{Step~2}. We now claim that
\begin{equation}\label{eq:weak_conv_alpha_p}
\nabla^\alpha_w f\weakto\nabla_w f
\quad\text{in~$L^p(\R^n;\R^n)$ as $\alpha\to1^-$.}	
\end{equation}
Indeed, let $\phi\in L^{\frac{p}{p-1}}(\R^n;\R^n)$. For each $\eps>0$, let $\psi_\eps\in C^\infty_c(\R^n;\R^n)$ be such that $\|\psi_\eps-\phi\|_{L^{\frac{p}{p-1}}(\R^n;\,\R^n)}<\eps$. By~\eqref{eq:int_by_parts_p_step_1} and~\eqref{eq:conv_tot_mass_p}, we can estimate
\begin{align*}
\limsup_{\alpha\to1^-}\bigg|\int_{\R^n}\phi\cdot\nabla^\alpha_w f\,dx&-\int_{\R^n}\phi\cdot\nabla_w f\,dx\,\bigg|
\le \limsup_{\alpha \to 1^{-}} \bigg|\int_{\R^n}\psi_{\eps}\cdot\nabla^\alpha_w f\,dx-\int_{\R^n}\psi_{\eps}\cdot\nabla_w f\,dx\,\bigg|  \\
& \quad + \int_{\R^n} |\phi - \psi_{\eps}| \, |\nabla^\alpha_w f| \,dx + \int_{\R^n}|\phi - \psi_{\eps}| \, |\nabla_w f|\,dx \\
&\le\eps\,\bigg(\lim_{\alpha\to1^-}\|\nabla^\alpha_w f\|_{L^p(\R^n;\,\R^n)}+\|\nabla_w f\|_{L^p(\R^n;\,\R^n)}\bigg)\\
&=2\eps\,\|\nabla_w f\|_{L^p(\R^n;\,\R^n)}
\end{align*}
so that~\eqref{eq:weak_conv_alpha_p} follows passing to the limit as~$\eps\to0^+$.

\smallskip

Since $L^p(\R^n;\R^n)$ is uniformly convex (see~\cite{Brezis11}*{Section~4.3} for example), the limit in~\eqref{eq:strong_conv_alpha_p} follows from~\eqref{eq:conv_tot_mass_p} and~\eqref{eq:weak_conv_alpha_p} by~\cite{Brezis11}*{Proposition~3.32}, and the proof in the case $p\in(1,+\infty)$ is complete.

\smallskip

For the case $p=1$, we argue as follows (we thank Mattia Calzi for this simple argument). Without loss of generality, it is enough to prove the limit in~\eqref{eq:conv_tot_mass_p} with $p=1$ for any given sequence $(\alpha_k)_{k\in\N}$ such that $\alpha_k\to 1^-$ as $k\to+\infty$. By \eqref{eq:conv_tot_mass}, the sequence $(\|\nabla^{\alpha_k} f\|_{L^1(\R^n;\,\R^n)})_{k\in\N}$ is bounded for any $f\in W^{1,1}(\R^n)$ and thus, by Banach--Steinhaus Theorem, the linear operators $\nabla^{\alpha_k}\colon W^{1,1}(\R^n)\to L^1(\R^n;\R^n)$, $k\in\N$, are uniformly bounded (in the operator norm). The conclusion hence follows by exploiting the density of $C^\infty_c(\R^n)$ in $W^{1,1}(\R^n)$ and \cref{res:L1_converg}.
\end{proof}

For the case $p=+\infty$, we have the following result.
The proof is very similar to the one of \cref{res:weak_conv_alpha_p} and is thus left to the reader.

\begin{theorem}\label{res:weak_conv_alpha_infty}
If $f\in W^{1,\infty}(\R^n)$, then
\begin{equation}\label{eq:weak_conv_alpha_infty}
\nabla^\alpha_w f\weakto\nabla_w f
\quad\text{in~$L^\infty(\R^n;\R^n)$ as $\alpha\to1^-$}	
\end{equation}
and
\begin{equation}\label{eq:liminf_tot_mass_infty}
\|\nabla_w f\|_{L^\infty(\R^n;\,\R^n)}\le\liminf_{\alpha\to1^-}\|\nabla^\alpha_w f\|_{L^\infty(\R^n;\,\R^n)}.
\end{equation}
\end{theorem}

\subsection{\texorpdfstring{$\Gamma$}{Gamma}-convergence of \texorpdfstring{$\alpha$}{alpha}-variation as \texorpdfstring{$\alpha\to1^-$}{alpha tends to 1-}}\label{subsec:Gamma_conv}

In this section, we study the $\Gamma$-convergence of the fractional $\alpha$-variation to the standard variation as $\alpha\to1^-$.

We begin with the $\Gamma\mhypen\liminf$ inequality.

\begin{theorem}[$\Gamma\mhypen\liminf$ inequalities as $\alpha\to1^-$]\label{res:Gamma_liminf}
Let $\Omega\subset\R^n$ be an open set.
\begin{enumerate}[(i)]

\item\label{item:Gamma_liminf_L_infty} If $(f_\alpha)_{\alpha\in(0,1)}\subset L^1_{\loc}(\R^n)$ satisfies $\sup_{\alpha\in(0,1)}\|f_\alpha\|_{L^\infty(\R^n)}<+\infty$ and $f_\alpha\to f$ in $L^1_{\loc}(\R^n)$ as $\alpha\to1^-$, then 
\begin{equation}\label{eq:Gamma_liminf}
|Df|(\Omega)\le\liminf\limits_{\alpha\to1^-}|D^\alpha f_\alpha|(\Omega).
\end{equation}

\item\label{item:Gamma_liminf_L_1} If $(f_\alpha)_{\alpha\in(0,1)}\subset L^1(\R^n)$ satisfies $f_\alpha\to f$ in $L^1(\R^n)$ as $\alpha\to1^-$, then~\eqref{eq:Gamma_liminf} holds.
\end{enumerate}

\end{theorem}  

\begin{proof}
We prove the two statements separately.

\smallskip

\textit{Proof of~(\ref{item:Gamma_liminf_L_infty})}.
Let $\phi\in C^\infty_c(\Omega;\R^n)$ be such that $\|\phi\|_{L^\infty(\Omega;\R^n)}\le1$. Since we can estimate
\begin{align*}
\bigg|\int_{\R^n}f_\alpha\,&\div^\alpha\phi\,dx-\int_{\R^n}f\,\div\phi\,dx\bigg|
\le\int_{\R^n}\abs*{f_\alpha-f}\,\abs*{\div\phi}\,dx
+\int_{\R^n}|f_\alpha|\,\abs*{\div^\alpha\phi-\div\phi}\,dx\\
&\le\|\div\phi\|_{L^\infty(\R^n;\,\R^n)}\int_{\supp\phi}\abs*{f_\alpha-f}\,dx
+\big(\sup_{\alpha\in(0,1)}\|f_\alpha\|_{L^\infty(\R^n)}\big)\,\|\div^\alpha\phi-\div\phi\|_{L^1(\R^n)},
\end{align*}
by \cref{res:L1_converg} we get that
\begin{align*}
\int_{\R^n}f\,\div\phi\,dx
=\lim_{\alpha\to1^-}\int_{\R^n}f_\alpha\,\div^\alpha\phi\,dx
\le\liminf_{\alpha\to1^-}|D^\alpha f|(\Omega)
\end{align*}
and the conclusion follows. 

\smallskip

\textit{Proof of~(\ref{item:Gamma_liminf_L_1})}.
Let $\phi\in C^\infty_c(\Omega;\R^n)$ be such that $\|\phi\|_{L^\infty(\Omega;\R^n)}\le1$. Since we can estimate
\begin{align*}
\bigg|\int_{\R^n} f_\alpha\,\div^\alpha\phi\,dx&-\int_{\R^n}f\,\div\phi\,dx\bigg|
\le\int_{\R^n}\abs*{f_\alpha-f}\,\abs*{\div\phi}\,dx
+\int_{\R^n}|f_\alpha|\,\abs*{\div^\alpha\phi-\div\phi}\,dx\\
&\le\|\div\phi\|_{L^\infty(\R^n)}\|f_\alpha-f\|_{L^1(\R^n)}+\|\div^\alpha\phi-\div\phi\|_{L^\infty(\R^n)}\|f_\alpha\|_{L^1(\R^n)},
\end{align*}
by \cref{res:L1_converg} we get that
\begin{align*}
\int_{\R^n}f\,\div\phi\,dx
=\lim_{\alpha\to1^-}\int_{\R^n}f_\alpha\,\div^\alpha\phi\,dx
\le\liminf_{\alpha\to1^-}|D^\alpha f_\alpha|(\Omega)
\end{align*}
and the conclusion follows. 
\end{proof}

We now pass to the $\Gamma\mhypen\limsup$ inequality.

\begin{theorem}[$\Gamma\mhypen\limsup$ inequalities as $\alpha\to1^-$]\label{res:Gamma_limsup}
Let $\Omega\subset\R^n$ be an open set.

\begin{enumerate}[(i)]

\item\label{item:Gamma_limsup_BV} If $f\in BV(\R^n)$ and either $\Omega$ is bounded or $\Omega=\R^n$, then
\begin{equation}\label{eq:limsup}
\limsup_{\alpha\to1^-}|D^\alpha f|(\Omega)\le |Df|(\overline{\Omega}).
\end{equation}

\item\label{item:Gamma_limsup_BV_loc} If $f\in BV_{\loc}(\R^n)$ and $\Omega$ is bounded, then
\begin{equation*}
\Gamma(L^1_{\loc})\mhypen\limsup_{\alpha\to1^-}|D^\alpha f|(\Omega)\le |Df|(\overline{\Omega}).
\end{equation*}

\end{enumerate}
In addition, if $f=\chi_E$, then the recovering sequences $(f_\alpha)_{\alpha\in(0,1)}$ in~(\ref{item:Gamma_limsup_BV}) and~(\ref{item:Gamma_limsup_BV_loc}) can be taken such that $f_\alpha=\chi_{E_\alpha}$ for some measurable sets $(E_\alpha)_{\alpha\in(0,1)}$.

\end{theorem}

\begin{proof}
Assume $f\in BV(\R^n)$. By \cref{res:weak_conv_tot_var}, we know that $|D^\alpha f|\weakto|Df|$ as $\alpha\to1^-$. Thus, by~\cite{M12}*{Proposition~4.26}, we get that
\begin{equation}\label{eq:limsup_BV_Omega_bounded}
\limsup_{\alpha\to1^-}|D^\alpha f|(\Omega)
\le\limsup_{\alpha\to1^-}|D^\alpha f|(\closure[-.5]{\Omega})
\le|Df|(\closure[-.5]{\Omega})
\end{equation}
for any bounded open set $\Omega\subset\R^n$.
If $\Omega=\R^n$, then~\eqref{eq:limsup} follows immediately from~\eqref{eq:conv_tot_mass}. This concludes the proof of~\eqref{item:Gamma_limsup_BV}.

Now assume that $f\in BV_{\loc}(\R^n)$ and $\Omega$ is bounded. Let $(R_k)_{k\in\N}\subset(0,+\infty)$ be a sequence such that $R_k\to+\infty$ as $k\to+\infty$ and set $f_k:=f\chi_{B_{R_k}}$ for all $k\in\N$. By \cref{res:BV_cut_ball}, we can choose the sequence $(R_k)_{k\in\N}$ such that, in addition, $f_k\in BV(\R^n)$ with $Df_k=\chi_{B_{R_k}}^\star Df+f^\star D\chi_{B_{R_k}}$ for all $k\in\N$. Consequently, $f_k\to f$ in~$L^1_{\loc}(\R^n)$ as $k\to+\infty$ and, moreover, since~$\Omega$ is bounded, $|Df_k|(\Omega)=|Df|(\Omega)$ and $|Df_k|(\de\Omega)=|Df|(\de\Omega)$ for all $k\in\N$ sufficiently large. By~\eqref{eq:limsup_BV_Omega_bounded}, we have that
\begin{equation}\label{eq:Gamma_limsup_step_1_for_each_k}
\limsup_{\alpha\to1^-}|D^\alpha f_k|(\Omega)
\le|Df_k|(\overline{\Omega})
\end{equation}
for all $k\in\N$ sufficiently large. Hence, by~\cite{B02}*{Proposition~1.28}, by~\cite{D93}*{Proposition~8.1(c)} and by~\eqref{eq:Gamma_limsup_step_1_for_each_k}, we get that
\begin{align*}
\Gamma(L^1_{\loc})\mhypen\limsup_{\alpha\to1^-}|D^\alpha f|(\Omega)
&\le\liminf_{k\to+\infty}\big(\Gamma(L^1_{\loc})\mhypen\limsup_{\alpha\to1^-}|D^\alpha f_k|(\Omega)\big)\\
&\le\lim_{k\to+\infty}|Df_k|(\overline{\Omega})
=|Df|(\overline{\Omega}).
\end{align*}
This concludes the proof of~\eqref{item:Gamma_limsup_BV_loc}.

Finally, if $f=\chi_E$, then we can repeat the above argument \textit{verbatim} in the metric spaces $\set*{\chi_F\in L^1(\R^n): F\subset\R^n}$ for~\eqref{item:Gamma_limsup_BV} and $\set*{\chi_F\in L^1_{\loc}(\R^n): F\subset\R^n}$ for~\eqref{item:Gamma_limsup_BV_loc} endowed with their natural distances.
\end{proof}

\begin{remark}
Thanks to~\eqref{eq:limsup}, a \emph{recovery sequence} in \cref{res:Gamma_limsup}\eqref{item:Gamma_limsup_BV} is the constant sequence (also in the special case $f = \chi_{E}$).
\end{remark}

Combining \cref{res:Gamma_liminf}\eqref{item:Gamma_liminf_L_infty} and \cref{res:Gamma_limsup}\eqref{item:Gamma_limsup_BV_loc}, we can prove that the fractional Caccioppoli $\alpha$-perimeter $\Gamma$-converges to De Giorgi's perimeter as $\alpha\to1^-$ in~$L^1_{\loc}(\R^n)$. We refer to~\cite{ADPM11} for the same result on the classical fractional perimeter.

\begin{theorem}[$\Gamma(L^1_{\loc})\mhypen\lim$ of perimeters as $\alpha\to1^-$]
\label{res:gamma_lim_alpha_to_1}
Let $\Omega\subset\R^n$ be a bounded open set with Lipschitz boundary. For every measurable set $E\subset\R^n$, we have
\begin{equation*}
\Gamma(L^1_{\loc})\mhypen\lim_{\alpha\to1^-}|D^\alpha\chi_E|(\Omega)
=P(E;\Omega).
\end{equation*}	
\end{theorem}

\begin{proof}
By \cref{res:Gamma_liminf}\eqref{item:Gamma_liminf_L_infty}, we already know that 
\begin{equation*}
\Gamma(L^1_{\loc})\mhypen\liminf_{\alpha\to1^-}|D^\alpha\chi_E|(\Omega)
\ge P(E;\Omega),
\end{equation*} 
so we just need to prove the $\Gamma(L^1_{\loc})\mhypen\limsup$ inequality. Without loss of generality, we can assume $P(E;\Omega)<+\infty$. Now let $(E_k)_{k\in\N}$ be given by \cref{res:approx_polyhedral}. Since $\chi_{E_k}\in BV_{\loc}(\R^n)$ and $P(E_k;\de\Omega)=0$ for all $k\in\N$, by \cref{res:Gamma_limsup}\eqref{item:Gamma_limsup_BV_loc} we know that
\begin{equation*}
\Gamma(L^1_{\loc})\mhypen\limsup_{\alpha\to1^-}|D^\alpha\chi_{E_k}|(\Omega)\le P(E_k;\Omega)
\end{equation*}
for all $k\in\N$. Since $\chi_{E_k}\to\chi_E$ in $L^1_{\loc}(\R^n)$ and $P(E_k;\Omega)\to P(E;\Omega)$ as $k\to+\infty$, by~\cite{B02}*{Proposition~1.28} we get that
\begin{align*}
\Gamma(L^1_{\loc})\mhypen\limsup_{\alpha\to1^-}|D^\alpha\chi_E|(\Omega)
&\le\liminf_{k\to+\infty}\big(\Gamma(L^1_{\loc})\mhypen\limsup_{\alpha\to1^-}|D^\alpha\chi_{E_k}|(\Omega)\big)\\
&\le\lim_{k\to+\infty} P(E_k;\Omega)
=P(E;\Omega)
\end{align*}
and the proof is complete.
\end{proof}

Finally, by combining \cref{res:Gamma_liminf}\eqref{item:Gamma_liminf_L_1} and \cref{res:Gamma_limsup}, we can prove that the fractional $\alpha$-variation $\Gamma$-converges to De Giorgi's variation as $\alpha\to1^-$ in $L^1(\R^n)$.

\begin{theorem}[$\Gamma(L^1)\mhypen\lim$ of variations as $\alpha\to1^-$]\label{res:gamma_lim_BV_alpha_to_1}
Let $\Omega\subset\R^n$ be an open set such that either $\Omega$ is bounded with Lipschitz boundary or $\Omega=\R^n$. For every $f\in BV(\R^n)$, we have
\begin{equation*}
\Gamma(L^1)\mhypen\lim_{\alpha\to1^-}|D^\alpha f|(\Omega)
=|Df|(\Omega).
\end{equation*}	
\end{theorem}

\begin{proof}
The case $\Omega=\R^n$ follows immediately by~\cite{D93}*{Proposition~8.1(c)} combining \cref{res:Gamma_liminf}\eqref{item:Gamma_liminf_L_1} with \cref{res:Gamma_limsup}\eqref{item:Gamma_limsup_BV}. We can thus assume that~$\Omega$ is a bounded open set with Lipschitz boundary and argue as in the proof of \cref{res:gamma_lim_alpha_to_1}.
By \cref{res:Gamma_liminf}\eqref{item:Gamma_liminf_L_1}, we already know that 
\begin{equation*}
\Gamma(L^1)\mhypen\liminf_{\alpha\to1^-}|D^\alpha f|(\Omega)
\ge |Df|(\Omega),
\end{equation*} 
so we just need to prove the $\Gamma(L^1)\mhypen\limsup$ inequality. Without loss of generality, we can assume $|Df|(\Omega)<+\infty$. Now let $(f_k)_{k\in\N}\subset BV(\R^n)$ be given by \cref{res:BV_loc_approx_no_mass}. Since $|Df_k|(\de\Omega)=0$ for all $k\in\N$, by \cref{res:Gamma_limsup} we know that
\begin{equation*}
\Gamma(L^1)\mhypen\limsup_{\alpha\to1^-}|D^\alpha f_k|(\Omega)
\le|Df_k|(\overline{\Omega})
=|Df_k|(\Omega)
\end{equation*}
for all $k\in\N$. Since $f_k\to f$ in $L^1(\R^n)$ and $|D^\alpha f_k|(\Omega)\to|D^\alpha f|(\Omega)$ as $k\to+\infty$, by~\cite{B02}*{Proposition~1.28} we get that
\begin{align*}
\Gamma(L^1)\mhypen\limsup_{\alpha\to1^-}|D^\alpha f|(\Omega)
&\le\liminf_{k\to+\infty}\big(\Gamma(L^1)\mhypen\limsup_{\alpha\to1^-}|D^\alpha f_k|(\Omega)\big)\\
&\le\lim_{k\to+\infty} |Df_k|(\Omega)
=|Df|(\Omega)
\end{align*}
and the proof is complete.
\end{proof}

\begin{remark}
Thanks to \cref{res:gamma_lim_BV_alpha_to_1}, we can slightly improve \cref{res:gamma_lim_alpha_to_1}. Indeed, if $\chi_{E} \in BV(\R^{n})$, then we also have
\begin{equation*}
\Gamma(L^1)\mhypen\lim_{\alpha\to1^-}|D^\alpha\chi_E|(\Omega)
=|D\chi_E|(\Omega)
\end{equation*}	
for any open set $\Omega\subset\R^n$ such that either $\Omega$ is bounded with Lipschitz boundary or~$\Omega=\R^n$.
\end{remark}

\section{Asymptotic behavior of fractional \texorpdfstring{$\beta$}{beta}-variation as \texorpdfstring{$\beta\to\alpha^-$}{beta tends to alpha-}}
\label{sec:asymptotic_beta_to_alpha}

\subsection{Convergence of~\texorpdfstring{$\nabla^\beta$}{nablaˆbeta} and~\texorpdfstring{$\div^\beta$}{divˆbeta} as \texorpdfstring{$\beta\to\alpha$}{beta tends to alpha}}

We begin with the following simple result about the $L^1$-convergence of the operators~$\nabla^\beta$ and~$\div^\beta$ as $\beta\to\alpha$ with $\alpha\in(0,1)$.

\begin{lemma}\label{res:L1_conv_alpha}
Let $\alpha\in(0,1)$. If $f\in W^{\alpha,1}(\R^n)$ and $\phi\in W^{\alpha,1}(\R^n;\R^n)$, then
\begin{equation}\label{eq:L1_conv_alpha}
\lim_{\beta\to\alpha^-}\|\nabla^\beta f-\nabla^\alpha f\|_{L^1(\R^n;\,\R^n)}=0,
\qquad
\lim_{\beta\to\alpha^-}\|\div^\beta\phi-\div^\alpha\phi\|_{L^1(\R^n)}=0.
\end{equation} 
\end{lemma}

\begin{proof}
Given $\beta\in(0,\alpha)$, we can estimate
\begin{align*}
\int_{\R^n}|\nabla^\beta f(x)-\nabla^\alpha f(x)|\,dx
&\le|\mu_{n,\beta}-\mu_{n,\alpha}|\,[f]_{W^{\alpha,1}(\R^n)}\\
&\quad+\mu_{n,\beta}\int_{\R^n}\int_{\R^n}\frac{|f(y)-f(x)|}{|y-x|^n}\,\abs*{\frac{1}{|y-x|^\beta}-\frac{1}{|y-x|^\alpha}}\,dy\,dx.
\end{align*} 
Since the $\Gamma$~function is continuous (see~\cite{A64}), we clearly have
\begin{equation*}
\lim_{\beta\to\alpha^-}|\mu_{n,\beta}-\mu_{n,\alpha}|\,[f]_{W^{\alpha,1}(\R^n)}=0.
\end{equation*}
Now write
\begin{align*}
&\int_{\R^n}\int_{\R^n}\frac{|f(y)-f(x)|}{|y-x|^n}\,\abs*{\frac{1}{|y-x|^\beta}-\frac{1}{|y-x|^\alpha}}\,dy\,dx\\
&=\int_{\R^n}\int_{\R^n}\frac{|f(y)-f(x)|}{|y-x|^n}\,\abs*{\frac{1}{|y-x|^\beta}-\frac{1}{|y-x|^\alpha}}\,\chi_{(0,1)}(|y-x|)\,dy\,dx\\
&\quad+\int_{\R^n}\int_{\R^n}\frac{|f(y)-f(x)|}{|y-x|^n}\,\abs*{\frac{1}{|y-x|^\beta}-\frac{1}{|y-x|^\alpha}}\,\chi_{[1,+\infty)}(|y-x|)\,dy\,dx.
\end{align*}
On the one hand, since $f\in W^{\alpha,1}(\R^n)$, we have
\begin{align*}
\frac{|f(y)-f(x)|}{|y-x|^n}\,&\abs*{\frac{1}{|y-x|^\beta}-\frac{1}{|y-x|^\alpha}}\,\chi_{(0,1)}(|y-x|)\\
&=\frac{|f(y)-f(x)|}{|y-x|^n}\,\left(\frac{1}{|y-x|^\alpha}-\frac{1}{|y-x|^\beta}\right)\chi_{(0,1)}(|y-x|)\\
&\le\frac{|f(y)-f(x)|}{|y-x|^{n+\alpha}}\,\chi_{(0,1)}(|y-x|)
\in L^1_{x,y}(\R^{2n})
\end{align*}
and thus, by Lebesgue's Dominated Convergence Theorem, we get that
\begin{equation*}
\lim_{\beta\to\alpha^-}\int_{\R^n}\int_{\R^n}\frac{|f(y)-f(x)|}{|y-x|^n}\,\abs*{\frac{1}{|y-x|^\beta}-\frac{1}{|y-x|^\alpha}}\,\chi_{(0,1)}(|y-x|)\,dy\,dx=0.	
\end{equation*}
On the other hand, since one has 
\begin{align*}
[f]_{W^{\beta,1}(\R^n)}
&=\int_{\R^n}\int_{\set*{|h|<1}}\frac{|f(x+h)-f(x)|}{|h|^{n+\beta}}\,dh\,dx
+\int_{\R^n}\int_{\set*{|h|\ge1}}\frac{|f(x+h)-f(x)|}{|h|^{n+\beta}}\,dh\,dx\\
&\le[f]_{W^{\alpha,1}(\R^n)}
+\int_{\set*{|h|\ge1}}\frac{1}{|h|^{n+\beta}}\int_{\R^n}|f(x+h)|+|f(x)|\,dx\,dh\\
&=[f]_{W^{\alpha,1}(\R^n)}+\frac{2n\omega_n}{\beta}\,\|f\|_{L^1(\R^n)}
\end{align*}
for all $\beta\in(0,\alpha)$, we can estimate
\begin{align*}
\frac{|f(y)-f(x)|}{|y-x|^n}\,&\abs*{\frac{1}{|y-x|^\beta}-\frac{1}{|y-x|^\alpha}}\,\chi_{[1,+\infty)}(|y-x|)\\
&=\frac{|f(y)-f(x)|}{|y-x|^n}\,\left(\frac{1}{|y-x|^\beta}-\frac{1}{|y-x|^\alpha}\right)\chi_{[1,+\infty)}(|y-x|)\\
&\le\frac{|f(y)-f(x)|}{|y-x|^{n+\beta}}\,\chi_{[1,+\infty)}(|y-x|)\\
&\le\frac{|f(y)-f(x)|}{|y-x|^{n+\frac{\alpha}{2}}}\,\chi_{[1,+\infty)}(|y-x|)
\in L^1_{x,y}(\R^{2n})
\end{align*}
for all $\beta\in\left(\frac{\alpha}{2},\alpha\right)$ and thus, by Lebesgue's Dominated Convergence Theorem, we get that
\begin{equation*}
\lim_{\beta\to\alpha^-}\int_{\R^n}\int_{\R^n}\frac{|f(y)-f(x)|}{|y-x|^n}\,\abs*{\frac{1}{|y-x|^\beta}-\frac{1}{|y-x|^\alpha}}\,\chi_{[1,+\infty)}(|y-x|)\,dy\,dx=0	
\end{equation*}
and the first limit in~\eqref{eq:L1_conv_alpha} follows. The second limit in~\eqref{eq:L1_conv_alpha} follows similarly and we leave the proof to the reader.
\end{proof}

\begin{remark}
Let $\alpha\in(0,1)$. If $f\in W^{\alpha+\eps,1}(\R^n)$ and $\phi\in W^{\alpha+\eps,1}(\R^n)$ for some $\eps\in(0,1-\alpha)$, then, arguing as in the proof of \cref{res:L1_conv_alpha}, one can also prove that 
\begin{equation*}
\lim_{\beta\to\alpha^+}\|\nabla^\beta f-\nabla^\alpha f\|_{L^1(\R^n;\,\R^n)}=0,
\qquad
\lim_{\beta\to\alpha^+}\|\div^\beta\phi-\div^\alpha\phi\|_{L^1(\R^n)}=0.
\end{equation*}
We leave the details of proof of this result to the interested reader.
\end{remark}

If one deals with more regular functions, then \cref{res:L1_conv_alpha} can be improved as follows.

\begin{lemma} \label{res:Linfty_conv_beta_alpha}
Let $\alpha\in(0,1)$ and $p \in [1, + \infty]$. If $f\in \Lip_{c}(\R^n)$ and $\phi\in \Lip_{c}(\R^n;\R^n)$, then
\begin{equation}\label{eq:Linfty_conv_beta_alpha}
\lim_{\beta\to\alpha^-}\|\nabla^\beta f-\nabla^\alpha f\|_{L^p(\R^n;\,\R^n)}=0,
\qquad
\lim_{\beta\to\alpha^-}\|\div^\beta\phi-\div^\alpha\phi\|_{L^p(\R^n)}=0.
\end{equation}
\end{lemma}

\begin{proof}
Since clearly $f \in W^{\alpha, 1}(\R^{n})$ for any $\alpha \in (0, 1)$, the first limit in~\eqref{eq:Linfty_conv_beta_alpha} for the case $p=1$ follows from \cref{res:L1_conv_alpha}. Hence, we just need to prove the validity of the same limit for the case $p = + \infty$, since then the conclusion simply follows by an interpolation argument.

Let $\beta \in (0, \alpha)$ and $x \in \R^{n}$. We have
\begin{align*}
|\nabla^{\alpha} f(x) - \nabla^{\beta} f(x)| 
&\le |\mu_{n, \beta} - \mu_{n, \alpha}| \int_{\R^{n}} \frac{|f(x) - f(y)|}{|x - y|^{n + \alpha}} \, dy\\
& \quad + \mu_{n, \beta} \int_{\R^{n}} \frac{|f(x) - f(y)|}{|x - y|^{n}} \, \bigg| \frac{1}{|x - y|^{\beta}} - \frac{1}{|x - y|^{\alpha}} \bigg | \, dy \\
& = |\mu_{n, \beta} - \mu_{n, \alpha}| \int_{\R^{n}} \frac{|f(x + z) - f(x)|}{|z|^{n + \alpha}} \, dz \\
& \quad + \mu_{n, \beta} \int_{\R^{n}} \frac{|f(x + z) - f(x)|}{|z|^{n}} \, \bigg | \frac{1}{|z|^{\beta}} - \frac{1}{|z|^{\alpha}} \bigg | \, dz.
\end{align*}
Since
\begin{align*}
\int_{\R^{n}} \frac{|f(x + z) - f(x)|}{|z|^{n + \alpha}} \, dz & \le \int_{\{ |z| \le 1 \}} \frac{\Lip(f)}{|z|^{n + \alpha - 1}} \, dz + \int_{\{ |z| > 1 \}} \frac{2 \|f\|_{L^{\infty}(\R^{n})}}{|z|^{n + \alpha}} \, dz \\
& \le n \omega_{n} \left ( \frac{\Lip(f)}{1 - \alpha} + \frac{2 \|f\|_{L^{\infty}(\R^{n})}}{\alpha} \right )
\end{align*}
and
\begin{align*}
\int_{\R^{n}} \frac{|f(x+ z) - f(z)|}{|z|^{n}} \, \bigg | \frac{1}{|z|^{\beta}} - \frac{1}{|z|^{\alpha}} \bigg | \, dz 
& \le \int_{\{ |z| \le 1 \}} \frac{\Lip(f)}{|z|^{n - 1}} \left ( \frac{1}{|z|^{\alpha}} - \frac{1}{|z|^{\beta}} \right ) dz \\
& \quad + \int_{\{ |z| > 1 \}} \frac{2 \|f\|_{L^{\infty}(\R^{n})}}{|z|^{n}} \bigg ( \frac{1}{|z|^{\beta}} - \frac{1}{|z|^{\alpha}} \bigg ) \, dz \\
& \le (\alpha - \beta) n \omega_{n} \, \bigg ( \frac{\Lip(f)}{(1 - \alpha)(1 - \beta)} + \frac{2 \|f\|_{L^{\infty}(\R^{n})}}{\alpha \beta} \bigg ),
\end{align*}
for all $\beta \in \big(\frac{\alpha}{2}, \alpha\big)$ we obtain
\begin{equation*}
\|\nabla^{\alpha} f - \nabla^{\beta} f\|_{L^{\infty}(\R^{n};\, \R^{n})} \le c_{n, \alpha}\max\set*{\Lip(f),\|f\|_{L^\infty(\R^n)}} \, \big( |\mu_{n, \beta} - \mu_{n, \alpha}| + (\alpha - \beta)\big),
\end{equation*}
for some constant $c_{n, \alpha} > 0$ depending only on~$n$ and~$\alpha$. Thus the conclusion follows since $\mu_{n,\beta}\to\mu_{n,\alpha}$ as $\beta\to\alpha^-$.
The second limit in~\eqref{eq:Linfty_conv_beta_alpha} follows similarly and we leave the proof to the reader.
\end{proof}

\subsection{Weak convergence of \texorpdfstring{$\beta$}{beta}-variation as \texorpdfstring{$\beta\to\alpha^-$}{beta tends to alpha-}}

In \cref{res:weak_conv_nabla_beta_alpha} below, we prove the weak convergence of the $\beta$-variation as $\beta\to\alpha^-$. The proof is very similar to those of \cref{res:weak_conv_nabla_alpha_to_1} and \cref{res:weak_conv_tot_var} and is thus left to the reader.

\begin{theorem}\label{res:weak_conv_nabla_beta_alpha}
Let $\alpha\in(0,1)$. If $f \in BV^\alpha(\R^{n})$, then 
\begin{equation*}
D^\beta f\weakto D^\alpha f
\quad
\text{and}
\quad
|D^\beta f|\weakto |D^\alpha f|
\quad\text{as $\beta\to\alpha^-$}.
\end{equation*}
Moreover, we have
\begin{equation}\label{eq:conv_tot_mass_alpha_beta}
\lim_{\beta\to\alpha^-}|D^\beta f|(\R^n)=|D^\alpha f|(\R^n).	
\end{equation}
\end{theorem}

\subsection{\texorpdfstring{$\Gamma$}{Gamma}-convergence of \texorpdfstring{$\beta$}{beta}-variation as \texorpdfstring{$\beta\to\alpha^-$}{beta tends to alpha-}}

In this section, we study the $\Gamma$-con\-vergence of the fractional $\beta$-variation as $\beta\to\alpha^-$, partially extending the results obtained in \cref{subsec:Gamma_conv}. 

We begin with the $\Gamma\mhypen\liminf$ inequality.

\begin{theorem}[$\Gamma\mhypen\liminf$ inequality for $\beta\to\alpha^-$]\label{res:Gamma_liminf_alpha_beta}
Let $\alpha\in(0,1)$ and let $\Omega\subset\R^n$ be an open set. If $(f_\beta)_{\beta\in(0,\alpha)}\subset L^1(\R^n)$ satisfies $f_\beta\to f$ in $L^1(\R^n)$ as $\beta\to\alpha^-$, then
\begin{equation}\label{eq:Gamma_liminf_alpha_beta}
|D^\alpha f|(\Omega)\le\liminf\limits_{\beta\to\alpha^-}|D^\beta f_\beta|(\Omega).
\end{equation}
\end{theorem}  

\begin{proof}
We argue as in the proof of \cref{res:Gamma_liminf}\eqref{item:Gamma_liminf_L_1}. Let $\phi\in C^\infty_c(\Omega;\R^n)$ be such that $\|\phi\|_{L^\infty(\Omega;\R^n)}\le1$. Let $U\subset\R^n$ be a bounded open set such that $\supp\phi\subset U$. By~\eqref{eq:frac_div_repr_Lip_estimate}, we can estimate
\begin{align*}
\bigg|\int_{\R^n}f_\beta\,\div^\beta\phi\,dx&-\int_{\R^n}f\,\div^\alpha\phi\,dx\bigg|
\le\int_{\R^n}\abs*{f_\beta-f}\,|\div^\beta\phi|\,dx
+\int_{\R^n}|f|\,|\div^\beta\phi-\div^\alpha\phi|\,dx\\
&\le C_{n,\beta,U}\|\div\phi\|_{L^\infty(\R^n;\,\R^n)}\|f_\beta-f\|_{L^1(\R^n)}+\int_{\R^n}|f|\,|\div^\beta\phi-\div^\alpha\phi|\,dx
\end{align*}
for all $\beta\in(0,\alpha)$. Since $\div^\beta\phi \to\div^\alpha\phi$ in $L^{\infty}(\R^{n})$ as~$\beta\to\alpha^-$ by~\eqref{eq:Linfty_conv_beta_alpha}, we easily obtain
\begin{equation*}
\lim_{\beta\to\alpha^-}\int_{\R^n}|f|\,|\div^\beta\phi-\div^\alpha\phi|\,dx = 0.
\end{equation*}
Hence, we get
\begin{align*}
\int_{\R^n}f\,\div^\alpha\phi\,dx
=\lim_{\beta\to\alpha^-}\int_{\R^n}f_\beta\,\div^\beta\phi\,dx
\le\liminf_{\beta\to\alpha^-}|D^\beta f_\beta|(\Omega)
\end{align*}
and the conclusion follows. 
\end{proof}

We now pass to the $\Gamma\mhypen\limsup$ inequality.

\begin{theorem}[$\Gamma\mhypen\limsup$ inequality for $\beta\to\alpha^-$]\label{res:Gamma_limsup_alpha_beta}
Let $\alpha\in(0,1)$ and let $\Omega\subset\R^n$ be an open set. If $f\in BV^\alpha(\R^n)$ and either $\Omega$ is bounded or $\Omega=\R^n$, then
\begin{equation}\label{eq:limsup_alpha_beta}
\limsup_{\beta\to\alpha^-}|D^\beta f|(\Omega)\le |D^\alpha f|(\overline{\Omega}).
\end{equation}
\end{theorem}

\begin{proof}
We argue as in the proof of \cref{res:Gamma_limsup}. By \cref{res:weak_conv_nabla_beta_alpha}, we know that $|D^\beta f|\weakto|D^\alpha f|$ as $\beta\to\alpha^-$. Thus, by~\cite{M12}*{Proposition~4.26} and~\eqref{eq:conv_tot_mass_alpha_beta}, we get that
\begin{equation}\label{eq:limsup_BV_Omega_bounded_alpha_beta}
\limsup_{\beta\to\alpha^-}|D^\beta f|(\Omega)
\le\limsup_{\beta\to\alpha^-}|D^\beta f|(\closure[-.5]{\Omega})
\le|D^\alpha f|(\closure[-.5]{\Omega})
\end{equation}
for any open set $\Omega\subset\R^n$ such that either $\Omega$ is bounded or $\Omega=\R^n$.
\end{proof}

\begin{corollary}[$\Gamma(L^1)\mhypen\lim$ of variations in $\R^n$ as $\beta\to \alpha^-$]\label{res:gamma_lim_BV_beta_to_alpha}
Let $\alpha \in (0, 1)$. For every $f\in BV^{\alpha}(\R^n)$, we have
\begin{equation*}
\Gamma(L^1)\mhypen\lim_{\beta\to \alpha^-}|D^\beta f|(\R^{n})
=|D^{\alpha}f|(\R^{n}).
\end{equation*}
In particular, the constant sequence is a recovery sequence.
\end{corollary}

\begin{proof}
The result follows easily by combining \eqref{eq:Gamma_liminf_alpha_beta} and \eqref{eq:limsup_alpha_beta} in the case $\Omega = \R^{n}$.
\end{proof}

\begin{remark}
We recall that, by~\cite{CS18}*{Theorem 3.25}, $f \in BV^{\alpha}(\R^{n})$ satisfies $|D^{\alpha} f| \ll \Leb{n}$ if and only if $f \in S^{\alpha, 1}(\R^{n})$. Therefore, if $f \in S^{\alpha, 1}(\R^{n})$, then $|D^{\alpha}f|(\partial \Omega) = 0$ for any bounded open set $\Omega\subset\R^n$ such that $\Leb{n}(\partial \Omega) = 0$ (for instance, $\Omega$ with Lipschitz boundary). Thus, we can actually obtain the $\Gamma$-convergence of the fractional $\beta$-variation as~$\beta\to\alpha^-$ on bounded open sets with Lipschitz boundary for any $f \in S^{\alpha, 1}(\R^{n})$ too. Indeed, it is enough to combine \eqref{eq:Gamma_liminf_alpha_beta} and \eqref{eq:limsup_alpha_beta} and then exploit the fact that \mbox{$|D^{\alpha} f|(\de\Omega)=0$} to get
\begin{equation*}
\Gamma(L^1)\mhypen\lim_{\beta\to \alpha^-}|D^\beta f|(\Omega)
=|D^{\alpha}f|(\Omega)
\end{equation*}
for any $f \in S^{\alpha, 1}(\R^{n})$.
\end{remark}

We were not able to find a reference for the analogue of \cref{res:gamma_lim_BV_beta_to_alpha} for the usual fractional Sobolev seminorms. 
For the sake of completeness, we state and prove it below for all $p\in[1,+\infty)$ on a general open set.

\begin{theorem}[$\Gamma(L^p)\mhypen\lim$ of $W^{\beta, p}$-seminorm as $\beta\to \alpha^-$]
Let $\Omega\subset\R^n$ be a non-empty open set, $\alpha \in (0, 1)$ and $p\in[1,+\infty)$. 
For every $f\in W^{\alpha, p}(\Omega)$, we have
\begin{equation*}
\Gamma(L^p)\mhypen\lim_{\beta\to \alpha^-} [f]_{W^{\beta,p}(\Omega)}
=[f]_{W^{\alpha,p}(\Omega)}.
\end{equation*}
In particular, the constant sequence is a recovery sequence.
\end{theorem}

\begin{proof}
Let $(f_\beta)_{\beta\in(0,\alpha)}\subset L^p(\Omega)$ be such that $f_\beta\to f$ in $L^p(\Omega)$ as $\beta\to\alpha^-$. 
Let $(\beta_k) \subset (0, \alpha)$ be such that $\beta_k\to\alpha$ as $k\to+\infty$ and
\begin{equation*}
\liminf_{\beta \to \alpha^-} [f_{\beta}]_{W^{\beta,p}(\Omega;\,\R^m)} = \lim_{k \to + \infty} [f_{\beta_k}]_{W^{\beta_k,p}(\Omega;\,\R^m)}.
\end{equation*}
Up to extract a further subsequence, we can assume that $f_{\beta_k}(x) \to f(x)$ as $k \to + \infty$ for a.e.\ $x \in \Omega$.
Then we can estimate
\begin{align*}
\lim_{k \to + \infty} \int_\Omega\int_\Omega\frac{|f_{\beta_k}(x)-f_{\beta_k}(y)|^p}{|x-y|^{n+p\beta_k}}\,dx\,dy & \ge \int_\Omega\int_\Omega \liminf_{k \to + \infty} \frac{|f_{\beta_k}(x)-f_{\beta_k}(y)|^p}{|x-y|^{n+p\beta_k}}\,dx\,dy \\
& \ge \int_\Omega\int_\Omega\frac{|f(x)-f(y)|^p}{|x-y|^{n+p\alpha}}\,dx\,dy
\end{align*}
by Fatou's Lemma. 
We thus get that
\begin{equation*}
\Gamma(L^p)\mhypen\liminf_{\beta\to \alpha^-} \,
[f]_{W^{\beta,p}(\Omega)}
\ge
[f]_{W^{\alpha,p}(\Omega)}.
\end{equation*}
Since
\begin{align*}
\lim_{\beta \to \alpha^-} \int_\Omega\int_\Omega\frac{|f(x)-f(y)|^p}{|x-y|^{n+p\beta}}\,dx\,dy 
&= 
\lim_{\beta \to \alpha^-} \int_\Omega\int_\Omega\frac{|f(x)-f(y)|^p}{|x-y|^{n+p\beta}}\,\chi_{\set*{|x-y|<1}}\,dx\,dy 
\\
&\quad+
\lim_{\beta \to \alpha^-} \int_\Omega\int_\Omega\frac{|f(x)-f(y)|^p}{|x-y|^{n+p\beta}}\,\chi_{\set*{|x-y|>1}}\,dx\,dy 
\\
&=
\int_\Omega\int_\Omega\frac{|f(x)-f(y)|^p}{|x-y|^{n+p\alpha}}\,dx\,dy
\end{align*}
by the Monotone Convergence Theorem, we also have that 
\begin{equation*}
\Gamma(L^p)\mhypen\limsup_{\beta\to \alpha^-} \,[f]_{W^{\beta,p}(\Omega)}
\le
[f]_{W^{\alpha,p}(\Omega)}
\end{equation*} 
and the conclusion immediately follows.
\end{proof}


\appendix

\section{Truncation and approximation of \texorpdfstring{$BV$}{BV} functions}
\label{sec:appendix}

In this appendix, we deal with two results on $BV$ functions and sets with locally finite perimeter.
These results are well known to experts, but we decided to state and prove them here because either we were not able to find them formulated in the exact form we needed  or the results available in the literature were not proved in full correctness (see \cref{rem:gappone} below).

\subsection{Truncation of \texorpdfstring{$BV$}{BV} functions}
Following~\cite{AFP00}*{Section~3.6} and~\cite{EG15}*{Section~5.9}, given $f\in L^1_{\loc}(\R^n)$, we define its precise representative~$f^\star\colon\R^n\to[-\infty,+\infty]$ as
\begin{equation}\label{eq:def_precise_repres}
f^\star(x):=\lim_{r\to0^+}\frac{1}{\omega_n r^n}\int_{B_r(x)} f(y)\,dy,
\quad x\in\R^n,
\end{equation} 
if the limit exists, otherwise we let $f^\star(x)=0$ by convention.

\begin{theorem}[Truncation of $BV$ functions]\label{res:BV_cut_ball}
If $f\in BV_{\loc}(\R^n)$, then
\begin{equation}\label{eq:BV_cut_ball}
f\chi_{B_r}\in BV(\R^n),
\ \text{with}\ 
D(f\chi_{B_r})=\chi_{B_r}^\star Df+f^\star D\chi_{B_r},
\end{equation}
for $\Leb{1}$-a.e.\ $r>0$. If, in addition, $f\in L^\infty(\R^n)$, then~\eqref{eq:BV_cut_ball} holds for all~$r>0$. 
\end{theorem}

\begin{proof}
Fix $\phi\in C^\infty_c(\R^n;\R^n)$ and let $U\subset\R^n$ be a bounded open set such that $\supp(\phi)\subset U$. Let $(\rho_\eps)_{\eps>0}\subset C^\infty_c(\R^n)$ be a family of standard mollifiers as in~\cite{CS18}*{Section~3.3} and set $f_\eps:=f*\rho_\eps$ for all $\eps>0$. Note that $\supp\big(\rho_\eps*(\chi_{B_r}\phi)\big)\subset U$ and $\supp\big(\rho_\eps*(\chi_{B_r}\div\phi)\big)\subset U$ for all $\eps>0$ sufficiently small and for all $r>0$. Given $r>0$, by Leibniz's rule and Fubini's Theorem, we have
\begin{equation}\label{eq:split_cut_ball_test}
\begin{split}
\int_{\R^n}f_\eps\chi_{B_r}\,\div\phi\,dx
&=\int_{\R^n}\chi_{B_r}\div(f_\eps\phi)\,dx
-\int_{\R^n}\chi_{B_r}\phi\cdot\nabla f_\eps\,dx\\
&=-\int_{\R^n}f_\eps\phi\cdot\,dD\chi_{B_r}
-\int_{\R^n}\rho_\eps*(\chi_{B_r}\phi)\cdot dDf.	
\end{split}	
\end{equation}
Since $f_\eps\to f$ a.e.\ in~$\R^n$ as $\eps\to0^+$ and
\begin{equation*}
|f|\,\rho_\eps*(\chi_{B_r}|\div\phi|)\le|f|\chi_U\|\div\phi\|_{L^\infty(\R^n)}\in L^1(\R^n)
\end{equation*}
for all $\eps>0$, by Lebesgue's Dominated Convergence Theorem we have
\begin{equation*}
\lim_{\eps\to0^+}\int_{\R^n}f_\eps\chi_{B_r}\,\div\phi\,dx
=\int_{\R^n}f\chi_{B_r}\,\div\phi\,dx
\end{equation*}
for all $r>0$. Thus, since $\rho_\eps*(\chi_{B_r}\phi)\to\chi_{B_r}^\star\phi$ pointwise in~$\R^n$ as $\eps\to0^+$ and 
\begin{equation*}
|\rho_\eps*(\chi_{B_r}\phi)|\le\|\phi\|_{L^\infty(\R^n;\,\R^n)}\chi_U\in L^1(\R^n,|Df|)
\end{equation*}
for all $\eps>0$ sufficiently small, again by Lebesgue's Dominated Convergence Theorem we have
\begin{equation*}
\lim_{\eps\to0^+}\int_{\R^n}\rho_\eps*(\chi_{B_r}\phi)\cdot dDf
=\int_{\R^n}\chi_{B_r}^\star\phi\cdot dDf
\end{equation*}
for all $r>0$. Now, by~\cite{AFP00}*{Theorem~3.78 and Corollary~3.80}, we know that $f_\eps\to f^\star$ $\Haus{n-1}$-a.e.\ in~$\R^n$ as $\eps\to0^+$. As a consequence, given any $r>0$, we get that $f_\eps\to f^\star$ $|D\chi_{B_r}|$-a.e.\ in~$\R^n$ as $\eps\to0^+$. Thus, if $f\in L^\infty(\R^n)$, then
\begin{equation*}
|f_\eps\phi|\le\|f\|_{L^\infty(\R^n)}|\phi|\in L^1(\R^n,|D\chi_{B_r}|)
\end{equation*}
for all $\eps>0$ and so, again by Lebesgue's Dominated Convergence Theorem, we have
\begin{align*}
\lim_{\eps\to0^+}\int_{\R^n}f_\eps\phi\cdot\,dD\chi_{B_r}
=\int_{\R^n}f^\star\phi\cdot\,dD\chi_{B_r}
\end{align*}
for all $r>0$. Therefore, if $f\in L^\infty(\R^n)$, then we can pass to the limit as $\eps\to0^+$ in~\eqref{eq:split_cut_ball_test} and get
\begin{equation*}
\int_{\R^n}f\chi_{B_r}\,\div\phi\,dx
=-\int_{\R^n}f^\star\phi\cdot\,dD\chi_{B_r}
-\int_{\R^n}\chi_{B_r}^\star\phi\cdot dDf
\end{equation*}
for all $\phi\in C^\infty_c(\R^n;\R^n)$ and for all $r>0$. Since $\|f^\star\|_{L^\infty(\R^n)}\le\|f\|_{L^\infty(\R^n)}$, this proves~\eqref{eq:BV_cut_ball} for all $r>0$. 

If~$f$ is not necessarily bounded, then we argue as follows. We start by observing that, since
\begin{equation*}
\int_0^R\int_{\de B_r}|f^\star|\,d\Haus{n-1}\,dr
=\int_{B_R}|f^\star|\,dx = \int_{B_R}|f|\,dx <+\infty,
\end{equation*}
the set
\begin{equation*}
W:=\set*{r>0 : \int_{\de B_r}|f^\star|\,d\Haus{n-1}\,dr<+\infty}
\end{equation*}
satisfies $\Leb{1}((0,+\infty)\setminus W)=0$. 
Without loss of generality, assume that $\|\phi\|_{L^\infty(\R^n;\,\R^n)}\le1$. 
Hence, for all $r \in W$, we can estimate
\begin{equation}\label{eq:simple_estimate_cut_ball}
\abs*{\int_{\R^n}f_\eps\phi\cdot\,dD\chi_{B_r}
-\int_{\R^n}f^\star\phi\cdot\,dD\chi_{B_r}}
\le
\int_{\de B_r}|f_\eps-f^\star|\,d\Haus{n-1}.
\end{equation}
Given any $R > 0$, by Fatou's Lemma we thus get that
\begin{align*}
\int_0^R\liminf_{\eps\to0^+}\int_{\de B_r}|f_\eps-f^\star|\,d\Haus{n-1}\,dr & \le\liminf_{\eps\to0^+}\int_0^R\int_{\de B_r}|f_\eps-f^\star|\,d\Haus{n-1}\,dr\\
&=\lim_{\eps\to0^+}\int_{B_R}|f_\eps-f^\star|\,dx
=0.
\end{align*}
Hence, the set
\begin{equation}\label{eq:def_Z_cut_ball}
Z:=\set*{r>0 : \liminf_{\eps\to0^+}\int_{\de B_r}|f_\eps-f^\star|\,d\Haus{n-1}=0}
\end{equation}
satisfies $\Leb{1}((0,+\infty)\setminus Z)=0$ and clearly does not depend on the choice of~$\phi$. Now fix $r\in Z \cap W$ and let $(\eps_k)_{k\in\N}$ be any sequence realising the $\liminf$ in~\eqref{eq:def_Z_cut_ball}. By~\eqref{eq:simple_estimate_cut_ball}, we thus get
\begin{equation*}
\lim_{k\to+\infty}\int_{\R^n}f_{\eps_k}\phi\cdot\,dD\chi_{B_r}
=\int_{\R^n}f^\star \phi\cdot\,dD\chi_{B_r}
\end{equation*}
uniformly for all~$\phi$ satisfying $\|\phi\|_{L^\infty(\R^n;\,\R^n)}\le1$. Passing to the limit along the sequence $(\eps_k)_{k\in\N}$ as $k\to+\infty$ in~\eqref{eq:split_cut_ball_test}, we get that
\begin{equation*}
\int_{\R^n}f\chi_{B_r}\,\div\phi\,dx
=-\int_{\R^n}f^\star\phi\cdot\,dD\chi_{B_r}
-\int_{\R^n}\chi_{B_r}^\star\phi\cdot dDf
\end{equation*}
for all $\phi\in C^\infty_c(\R^n;\R^n)$ with $\|\phi\|_{L^\infty(\R^n;\,\R^n)}\le1$. Thus~\eqref{eq:BV_cut_ball} follows for all $r\in W\cap Z$ and the proof is concluded.
\end{proof}

\subsection{Approximation by sets with polyhedral boundary}

In this section we state and prove standard approximation results for sets with finite perimeter or, more generally, $BV_{\loc}(\R^n)$ functions, in a sufficiently regular bounded open set.

We need the following two preliminary lemmas.

\begin{lemma}\label{res:rotation_baire}
Let $V,W\subset\mathbb{S}^{n-1}$, with $V$ finite and $W$ at most countable. For any $\eps>0$, there exists $\mathcal{R}\in\mathrm{SO}(n)$ with $|\mathcal{R}-\Id|<\eps$, where $\Id$ is the identity matrix, such that $\mathcal{R}(V)\cap W=\varnothing$.
\end{lemma}

\begin{proof}
Let $N\in\N$ be such that $V=\set*{v_i\in\mathbb{S}^{n-1} : i=1,\dots,N}$. We divide the proof in two steps.

\smallskip

\textit{Step~1}. Assume that $W$ is finite and set $A_i:=\set*{\mathcal{R}\in\mathrm{SO}(n) : \mathcal{R}(v_i)\notin W}$
for all $i=1,\dots,N$. We now claim that~$A_i$ is an open and dense subset of $\mathrm{SO}(n)$ for all $i=1,\dots,N$. Indeed, given any $i=1,\dots,N$, since $W$ is finite, the set $A_i^c=\mathrm{SO}(n)\setminus A_i$ is closed in~$\mathrm{SO}(n)$. Moreover, we claim that $\mathrm{int}(A_i^c) = \varnothing$. Indeed, by contradiction, let us assume that $\mathrm{int}(A_i^c)\ne\varnothing$. Then there exist $\eps>0$ and $\mathcal{R}\in A_i^c$ such that any $\mathcal{S}\in\mathrm{SO}(n)$ with $|\mathcal{S}-\mathcal{R}|<\eps$ satisfies $\mathcal{S}\in A_i^c$.
Now let
\begin{equation*}
\mathcal Q_\theta
:=
\begin{pmatrix}
\begin{matrix}
\cos\theta & -\sin\theta\\
\sin\theta & \cos\theta\\	
\end{matrix}
& \scalebox{1.5}{$0$}\\[5mm]
\scalebox{1.5}{$0$} & \scalebox{1.25}{$\mathcal{I}_{n-2}$}\\
\end{pmatrix}
\in\mathrm{SO}(n)
\end{equation*}
and define $\mathcal S_\theta := \mathcal Q_\theta\,\mathcal R\in\mathrm{SO}(n)$ for all $\theta\in[0,2\pi)$.
Since
\begin{equation*}
|\mathcal S_\theta - \mathcal R|
=
|(\mathcal Q_\theta - \mathcal I)\,\mathcal R|
\le
|\mathcal Q_\theta - \mathcal I|\,|\mathcal R|
<
\eps
\end{equation*} 
for all $\theta\in[0,\delta]$ for some $\delta>0$ depending only on $\eps$ and $\mathcal R$, we get that $\mathcal S_\theta\in A_i^c$ for all $\theta\in[0,\delta]$. Therefore $\mathcal S_\theta(v_i)\in W$ for all $\theta\in[0,\delta]$, in contrast with the fact that~$W$ is finite. 
Thus, $A_i$ is an open and dense subset of $\mathrm{SO}(n)$ for all $i=1,\dots,N$, and so also the set
\begin{equation*}
A^W:=\bigcap_{i=1}^N A_i=\set*{\mathcal{R}\in\mathrm{SO}(n) : \mathcal{R}(v_i)\notin W\ \forall i=1,\dots,N}
\end{equation*}  
is an open and dense subset of $\mathrm{SO}(n)$. The result is thus proved for any finite set~$W$. 

\smallskip

\textit{Step~2}. Now assume that $W$ is countable, $W=\set*{w_k\in\mathbb{S}^{n-1} : k\in\N}$. For all $M\in\N$, set $W_M:=\set*{w_k\in W : k\le M}$. By Step~1, we know that $A^{W_M}$ is an open and dense subset of $\mathrm{SO}(n)$ for all $M\in\N$. Since $\mathrm{SO}(n)\subset\R^{n^2}$ is compact, by Baire's Theorem $A:=\bigcap_{M\in\N} A^{W_M}$ is a dense subset of $\mathrm{SO}(n)$. This concludes the proof.
\end{proof}

Since $\det\colon\mathrm{GL}(n)\to\R$ is a continuous map, there exists a dimensional constant $\delta_n\in(0,1)$ such that $\det\mathcal{R}\ge\frac{1}{2}$ for all $\mathcal{R}\in\mathrm{GL}(n)$ with $|\mathcal{R}-\Id|<\delta_n$. 

\begin{lemma}\label{res:rotation_estimate_vol}
Let $\eps\in(0,\delta_n)$ and let $E\subset\R^n$ be a bounded set with $P(E)<+\infty$. If $\mathcal{R}\in\mathrm{SO}(n)$ satisfies $|\mathcal{R}-\Id|<\eps$, then
\begin{equation*}
|\mathcal{R}(E)\bigtriangleup E|\le 2\eps r_E \, P(E),
\end{equation*}
where $r_E:=\sup\set*{r>0 : |E\setminus B_r|>0}$.
\end{lemma}

\begin{proof}
We divide the proof in two steps.

\smallskip

\textit{Step~1}. Let $r>0$ and let $f\in C^\infty_c(\R^n)$. Setting $\mathcal{R}_t:=(1-t)\Id+t\mathcal{R}$ for all $t\in[0,1]$, we can estimate
\begin{align*}
\int_{B_r}|f(\mathcal{R}(x))-f(x)|\,dx
&=\int_{B_r}\abs*{\int_0^1 \nabla f(\mathcal{R}_t(x)) \cdot (\mathcal{R}(x)-x)\,dt}\,dx\\
&\le|\mathcal{R}-\Id|\,r\int_0^1\int_{B_r}|\nabla f(\mathcal{R}_t(x))|\,dx\,dt.
\end{align*}
Since $|\mathcal{R}_t-\Id|=t|\mathcal{R}-\Id|<t\eps<\delta_n$ for all $t\in[0,1]$, $\mathcal{R}_t$ is invertible with $\det(\mathcal{R}_t^{-1})\le 2$ for all $t\in[0,1]$. Hence we can estimate
\begin{align*}
\int_{B_r}|\nabla f(\mathcal{R}_t(x))|\,dx
=\int_{\mathcal{R}_t(B_r)}|\nabla f(y)|\,|\det(\mathcal{R}^{-1}_t)|\,dy
\le 2\int_{\R^n}|\nabla f(y)|\,dy,
\end{align*} 
so that
\begin{equation}\label{eq:rotation_estimate_test}
\int_{B_r}|f(\mathcal{R}(x))-f(x)|\,dx
\le2\eps r\|\nabla f\|_{L^{1}(\R^n; \R^{n})}.
\end{equation}

\smallskip

\textit{Step~2}. Since $\chi_E\in BV(\R^n)$, by combining~\cite{EG15}*{Theorem~5.3} with a standard cut-off approximation argument, we find $(f_k)_{k\in\N}\subset C^\infty_c(\R^n)$ such that $f_k\to\chi_E$ pointwise a.e.\ in~$\R^n$ and $|\nabla f_k|(\R^n)\to P(E)$ as $k\to+\infty$. Given any $r>0$, by~\eqref{eq:rotation_estimate_test} in Step~1 we have
\begin{equation*}
\int_{B_r}|f_k(\mathcal{R}(x))-f_k(x)|\,dx
\le 2\eps r\|\nabla f_k\|_{L^{1}(\R^n; \R^{n})}
\end{equation*}
for all $k\in\N$. Passing to the limit as $k\to+\infty$, by Fatou's Lemma we get that
\begin{equation*}
|(\mathcal{R}(E)\bigtriangleup E)\cap B_r|\le 2\eps r\, P(E).
\end{equation*} 
Since $E\subset B_{r_E}$ up to $\Leb{n}$-negligible sets, also $\mathcal{R}(E)\subset B_{r_E}$ up to $\Leb{n}$-negligible sets. Thus we can choose $r=r_E$ and the proof is complete.
\end{proof}

We are now ready to prove the main approximation result, see also~\cite{ADPM11}*{Proposition~15}. 

\begin{theorem}\label{res:approx_polyhedral}
Let $\Omega\subset\R^n$ be a bounded open set with Lipschitz boundary and let $E\subset\R^n$ be a measurable set such that $P(E;\Omega)<+\infty$. There exists a sequence $(E_k)_{k\in\N}$ of bounded open sets with polyhedral boundary such that
\begin{equation}\label{eq:approx_polyhedral_1}
P(E_k;\de\Omega)=0
\end{equation}
for all $k\in\N$ and
\begin{equation}\label{eq:approx_polyhedral_2}
\chi_{E_k}\to\chi_E\ 
\text{in $L^1_{\loc}(\R^n)$} 
\quad\text{and}\quad
P(E_k;\Omega)\to P(E;\Omega)
\end{equation}
as $k\to+\infty$.
\end{theorem}

\begin{proof}
We divide the proof in four steps.

\smallskip

\textit{Step~1: cut-off}. Since $\Omega$ is bounded, we find $R_0>0$ such that $\overline{\Omega}\subset B_{R_0}$. Let us define $R_k=R_0+k$ and
\begin{equation*}
C_k:=\set*{x\in\Omega^c : \dist(x,\de\Omega)\le\frac{1}{k}}
\end{equation*}
for all $k\in\N$. We set $E^1_k:=E\cap B_{R_k}\cap C_k^c$ for all $k\in\N$. Note that $E_k^1$ is a bounded measurable set such that 
\begin{equation*}
\chi_{E_k^1}\to\chi_E\
\text{in~$L^1_{\loc}(\R^n)$ as $k\to+\infty$}
\end{equation*}
and
\begin{equation*}
P(E_k^1;\Omega)=P(E;\Omega)\ 
\text{for all $k\in\N$}.
\end{equation*}

\smallskip

\textit{Step~2: extension}. Let us define
\begin{equation*}
A_k:=\set*{x\in\R^n : \dist(x,\Omega)<\frac{1}{4k}}
\end{equation*}
for all $k\in\N$. Since $\chi_{E_k^1\cap\Omega}\in BV(\Omega)$ for all $k\in\N$, by~\cite{AFP00}*{Definition~3.20 and Proposition~3.21} there exists a sequence $(v_k)_{k\in\N}\subset BV(\R^n)$ such that 
\begin{equation*}
v_k=0 \text{ a.e.\ in $A^c_k$},
\quad
v_k=\chi_{E_k^1} \text{ in $\Omega$},
\quad
|Dv_k|(\de\Omega)=0
\end{equation*} 
for all $k\in\N$. Let us define $F^t_k:=\set*{v_k>t}$ for all $t\in(0,1)$. Given $k\in\N$, by the coarea formula~\cite{AFP00}*{Theorem~3.40}, for a.e.\ $t\in(0,1)$ the set $F_k^t$ has finite perimeter in~$\R^n$ and satisfies
\begin{equation*}
F_k^t\subset A_k,
\quad
F_k^t\cap\Omega=E_k^1\cap\Omega,
\quad
P(F_k^t;\de\Omega)=0
\end{equation*} 
for all $k\in\N$. We choose any such $t_k\in(0,1)$ for each $k\in\N$ and define $E_k^2:=E_k^1\cup F_k^{t_k}$ for all $k\in\N$. Note that $E_k^2$ is a bounded set with finite perimeter in~$\R^n$ such that
\begin{equation*}
\chi_{E_k^2}\to\chi_E\
\text{in~$L^1_{\loc}(\R^n)$ as $k\to+\infty$}
\end{equation*}
and
\begin{equation*}
P(E_k^2;\Omega)=P(E;\Omega)\
\quad\text{and}\quad
P(E_k^2;\de\Omega)=0
\quad
\text{for all $k\in\N$}.
\end{equation*}

\smallskip

\textit{Step~3: approximation}. Let us define
\begin{equation*}
D_k:=\set*{x\in\Omega^c : \dist(x,\de\Omega)\in\left[\frac{1}{4k},\frac{3}{4k}\right]}
\end{equation*}
for all $k\in\N$. First arguing as in the first part of the proof of~\cite{M12}*{Theorem~13.8} taking~\cite{M12}*{Remark~13.13} into account, and then performing a standard diagonal argument, we find a sequence of bounded open sets $(E_k^3)_{k\in\N}$ with polyhedral boundary such that
\begin{equation*}
E^3_k\subset D_k^c\
\text{for all $k\in\N$}
\end{equation*}
and
\begin{equation*}
\chi_{E_k^3}\to\chi_E\ 
\text{in $L^1_{\loc}(\R^n)$}, 
\quad
P(E_k^3;\Omega)\to P(E;\Omega)
\quad\text{and}\quad
P(E_k^3;\de\Omega)\to0
\end{equation*}
as $k\to+\infty$. If there exists a subsequence $(E_{k_j}^3)_{j\in\N}$ such that $P(E_{k_j}^3;\de\Omega)=0$ for all $j\in\N$, then we can set $E_j:=E_{k_j}$ for all $j\in\N$ and the proof is concluded. If this is not the case, then we need to proceed with the next last step. 

\smallskip

\textit{Step~4: rotation}. We now argue as in the last part of the proof of~\cite{ADPM11}*{Proposition~15}. Fix $k\in\N$ and assume $P(E^3_k;\de\Omega)>0$. Since $E_k^3$ has polyhedral boundary, we have $\Haus{n-1}(\de E^3_k\cap\de\Omega)>0$ if and only if there exist $\nu\in\mathbb{S}^{n-1}$ and $U\subset\redb\Omega$ such that $\Haus{n-1}(U)>0$, $\nu_\Omega(x)=\nu$ for all $x\in U$ and~$U\subset\de H$ for some (affine) half-space~$H$ satisfying $\nu_H=\nu$. Since $P(\Omega)=\Haus{n-1}(\de\Omega)<+\infty$, the set
\begin{align*}
W:&=\set*{\nu\in\mathbb{S}^{n-1} : \Haus{n-1}\left(\set*{x\in\de\Omega : \nu_\Omega(x)=\nu}\right)>0}\\
&=\bigcup_{h\in\N}\set*{\nu\in\mathbb{S}^{n-1} : \tfrac{P(\Omega)}{h}\ge\Haus{n-1}\left(\set*{x\in\de\Omega : \nu_\Omega(x)=\nu}\right)>\tfrac{P(\Omega)}{h+1})}
\end{align*} 
is at most countable. Since $E_k^3$ has polyhedral boundary, the set
\begin{equation*}
V_k:=\set*{\nu\in\mathbb{S}^{n-1} : \Haus{n-1}\left(\set*{x\in\de E^3_k : \nu_{E^3_k}(x)=\nu}\right)>0}
\end{equation*}
is finite. By \cref{res:rotation_baire}, given $\eps_k>0$, there exists $\mathcal{R}_k\in\mathrm{SO}(n)$ with $|\mathcal{R}_k-\Id|<\eps_k$ such that $\mathcal{R}_k(V_k)\cap W=\varnothing$. Hence the set $E^4_k:=\mathcal{R}_k(E^3_k)$ must satisfy $P(E^4_k;\de\Omega)=0$. By \cref{res:rotation_estimate_vol}, we can choose $\eps_k>0$ sufficiently small in order to ensure that $|E^4_k\bigtriangleup E^3_k|<\frac{1}{k}$. Now choose $\eta_k\in\left(0,\frac{1}{2k}\right)$ such that $P(E^3_k;Q_k)\le 2P(E^3_k;\de\Omega)$, where
\begin{equation*}
Q_k:=\set*{x\in\R^n : \dist(x,\de\Omega)<\eta_k}.
\end{equation*}
Since $\Omega$ is bounded, possibly choosing $\eps_k>0$ even smaller, we can also ensure that $\Omega\bigtriangleup\mathcal{R}^{-1}(\Omega)\subset Q_k$. Hence we can estimate
\begin{align*}
|P(E^4_k;\Omega)-P(E^3_k;\Omega)|
&=|\Haus{n-1}(\partial E^3_k\cap\mathcal{R}^{-1}(\Omega))-\Haus{n-1}(\partial E^3_k\cap\Omega)|\\
&\le\Haus{n-1}\big(\partial E^3_k\cap(\Omega\bigtriangleup\mathcal{R}^{-1}(\Omega))\big)\\
&\le\Haus{n-1}(\partial E^3_k\cap Q_k).
\end{align*}
We can thus set $E_k:=E^4_k$ for all $k\in\N$ and the proof is complete.
\end{proof}

\begin{remark}[A minor gap in the proof of~\cite{ADPM11}*{Proposition~15}]
\label{rem:gappone}
We warn the reader that the cut-off and the extension steps presented above were not mentioned in the proof of~\cite{ADPM11}*{Proposition~15}, although they are unavoidable for the correct implementation of the rotation argument in the last step. Indeed, in general, one cannot expect the existence of a rotation $\mathcal{R}\in\mathrm{SO}(n)$ arbitrarily close to the identity map such that $P(\mathcal{R}(E);\de\Omega)=0$ and, at the same time, the difference between $P(\mathcal{R}(E);\Omega)$ and $P(E;\Omega)$ is small. For example, one can consider $n=2$,
\begin{equation*}
\Omega=\set*{(x_1,x_2)\in A : x_1^2+x_2^2<25}
\end{equation*}
and
\begin{equation*}
E=\set*{(x_1,x_2)\in A : 1<x_1^2+x_2^2<4}
\cup\set*{(x_1,x_2)\in A^c : 9<x_1^2+x_2^2<16},
\end{equation*}
where $A=\set*{(x_1,x_2)\in\R^2 : x_1>0,\ x_2>0}$. In this case, for any rotation $\mathcal{R}\in\mathrm{SO}(2)$ arbitrarily close to the identity map, we have $P(\mathcal{R}(E);\Omega)>2+P(E;\Omega)$.
\end{remark}

We conclude this section with the following result, establishing an approximation of $BV_{\loc}$ functions similar to that given in \cref{res:approx_polyhedral}.

\begin{theorem}\label{res:BV_loc_approx_no_mass}
Let $\Omega\subset\R^n$ be a bounded open set with Lipschitz boundary and let $f\in BV_{\loc}(\R^n)$. There exists $(f_k)_{k\in\N}\subset BV(\R^n)$ such that
\begin{equation*}
|Df_k|(\de\Omega)=0
\end{equation*}
for all $k\in\N$ and
\begin{equation*}
f_k\to f\
\text{in $L^1_{\loc}(\R^n)$}
\quad\text{and}\quad
|Df_k|(\Omega)\to|Df|(\Omega)
\end{equation*}
as $k\to+\infty$. If, in addition, $f\in L^1(\R^n)$, then $f_k\to f$ in $L^1(\R^n)$ as $k\to+\infty$.
\end{theorem}

\begin{proof}

We argue as in the proof of \cref{res:approx_polyhedral}, in two steps.

\smallskip

\textit{Step~1: cut-off at infinity}. Since $\Omega$ is bounded, we find $R_0>0$ such that $\overline{\Omega}\subset B_{R_0}$. Given $(R_k)_k\subset(R_0,+\infty)$, we set $g_k:=f\chi_{B_{R_k}}$ for all $k\in\N$. By \cref{res:BV_cut_ball}, we have $g_k\in BV(\R^n)$ for a suitable choice of the sequence $(R_k)_{k\in\N}$, with $|Dg_k|(\Omega)=|Df|(\Omega)$ for all $k\in\N$ and $g_k\to f$ in~$L^1_{\loc}(\R^n)$ as $k\to+\infty$. If, in addition, $f\in L^1(\R^n)$, then $g_k\to f$ in~$L^1(\R^n)$ as $k\to+\infty$.

\smallskip

\textit{Step~2: extension and cut-off near~$\Omega$}. Let us define 
\begin{equation*}
A_k:=\set*{x\in\R^n : \dist(x,\Omega)<\frac{1}{k}}
\end{equation*}
for all $k\in\N$. Since $g_k\chi_\Omega\in BV(\Omega)$ with $|Dg_k|(\Omega)=|Df|(\Omega)$ for all $k\in\N$, by~\cite{AFP00}*{Definition~3.20 and Proposition~3.21} there exists a sequence $(h_k)_{k\in\N}\subset BV(\R^n)$ such that 
\begin{equation*}
\supp h_k\subset A_{2k},
\quad
h_k=g_k \text{ in $\Omega$},
\quad
|Dh_k|(\de\Omega)=0
\end{equation*} 
for all $k\in\N$ and
\begin{equation*}
\lim_{k\to+\infty}\int_{A_{2k}\setminus\Omega}|h_k|\,dx=0
\end{equation*}
(the latter property easily follows from the construction performed in the proof of~\cite{AFP00}*{Proposition~3.21}). Now let $(v_k)_{k\in\N}\subset C^\infty_c(\R^n)$ be such that $\supp v_k\subset A^c_k$ and $0\le v_k\le 1$ for all $k\in\N$ and $v_k\to\chi_{\closure[0]{\Omega}^c}$ pointwise in~$\R^n$ as $k\to+\infty$. We can thus set $f_k:=h_k+v_kg_k$ for all $k\in\N$. By~\cite{AFP00}*{Propositon~3.2(b)}, we have $v_kg_k\in BV(\R^n)$ for all $k\in\N$, so that $f_k\in BV(\R^n)$ for all $k\in\N$. Since we can estimate
\begin{align*}
|f_k-f|
&\le|h_k-f\chi_\Omega|+|v_k-\chi_{\Omega^c}|\,|g_k|+|g_k-f|\,\chi_{\Omega^c}\\
&=|h_k|\,\chi_{A_{2k}\setminus\Omega}+|v_k-\chi_{\Omega^c}|\,|g_k|+|g_k-f|\,\chi_{\Omega^c}
\end{align*}
for all $k\in\N$, we have $f_k\to f$ in~$L^1_{\loc}(\R^n)$ as $k\to+\infty$ (because $\partial \Omega$ is Lipschitz, so $\Leb{n}(\partial \Omega) = 0$), with $f_k\to f$ in~$L^1(\R^n)$ as $k\to+\infty$ if $f\in L^1(\R^n)$. By construction, we also have
\begin{equation*}
|Df_k|(\Omega)=|Dh_k|(\Omega)
\quad\text{and}\quad
|Df_k|(\de\Omega)=|Dh_k|(\de\Omega)=0
\end{equation*}
for all $k\in\N$. The proof is complete.
\end{proof}



\begin{bibdiv}
\begin{biblist}

\bib{AH96}{book}{
   author={Adams, David R.},
   author={Hedberg, Lars Inge},
   title={Function spaces and potential theory},
   series={Grundlehren der Mathematischen Wissenschaften [Fundamental
   Principles of Mathematical Sciences]},
   volume={314},
   publisher={Springer-Verlag, Berlin},
   date={1996},
}

\bib{ACPS20}{article}{
   author={Alberico, Angela},
   author={Cianchi, Andrea},
   author={Pick, Lubo\v{s}},
   author={Slav\'{\i}kov\'{a}, Lenka},
   title={On the limit as $s\to 0^+$ of fractional Orlicz-Sobolev spaces},
   journal={J. Fourier Anal. Appl.},
   volume={26},
   date={2020},
   number={6},
   pages={Paper No. 80, 19},
}

\bib{ADPM11}{article}{
   author={Ambrosio, Luigi},
   author={De Philippis, Guido},
   author={Martinazzi, Luca},
   title={Gamma-convergence of nonlocal perimeter functionals},
   journal={Manuscripta Math.},
   volume={134},
   date={2011},
   number={3-4},
   pages={377--403},
}


\bib{AFP00}{book}{
   author={Ambrosio, Luigi},
   author={Fusco, Nicola},
   author={Pallara, Diego},
   title={Functions of bounded variation and free discontinuity problems},
   series={Oxford Mathematical Monographs},
   publisher={The Clarendon Press, Oxford University Press, New York},
   date={2000},
}

\bib{A20}{article}{
   author={Ambrosio, Vincenzo},
   title={On some convergence results for fractional periodic Sobolev spaces},
   journal={Opuscula Math.},
   volume={40},
   date={2020},
   number={1},
   pages={5--20},
}

\bib{AGMP18}{article}{
   author={Antonucci, Clara},
   author={Gobbino, Massimo},
   author={Migliorini, Matteo},
   author={Picenni, Nicola},
   title={On the shape factor of interaction laws for a non-local approximation of the Sobolev norm and the total variation},
   journal={C. R. Math. Acad. Sci. Paris},
   volume={356},
   date={2018},
   number={8},
   pages={859--864},
}

\bib{AGMP20}{article}{
   author={Antonucci, Clara},
   author={Gobbino, Massimo},
   author={Migliorini, Matteo},
   author={Picenni, Nicola},
   title={Optimal constants for a nonlocal approximation of Sobolev norms and total variation},
   journal={Anal. PDE},
   volume={13},
   date={2020},
   number={2},
   pages={595--625},
}

\bib{AGP20}{article}{
   author={Antonucci, Clara},
   author={Gobbino, Massimo},
   author={Picenni, Nicola},
   title={On the gap between the Gamma-limit and the pointwise limit for a nonlocal approximation of the total variation},
   journal={Anal. PDE},
   volume={13},
   date={2020},
   number={3},
   pages={627--649},
}

\bib{A64}{book}{
   author={Artin, Emil},
   title={The Gamma function},
   series={Translated by Michael Butler. Athena Series: Selected Topics in Mathematics},
   publisher={Holt, Rinehart and Winston, New York-Toronto-London},
   date={1964},
   pages={vii+39},
}

\bib{AK09}{article}{
   author={Aubert, Gilles},
   author={Kornprobst, Pierre},
   title={Can the nonlocal characterization of Sobolev spaces by Bourgain et al. be useful for solving variational problems?},
   journal={SIAM J. Numer. Anal.},
   volume={47},
   date={2009},
   number={2},
   pages={844--860},
}

\bib{BMR20}{article}{
   author={Bal, Kaushik},
   author={Mohanta, Kaushik},
   author={Roy, Prosenjit},
   title={Bourgain-Brezis-Mironescu domains},
   journal={Nonlinear Anal.},
   volume={199},
   date={2020},
   pages={111928, 10},
}

\bib{B11}{article}{
   author={Barbieri, Davide},
   title={Approximations of Sobolev norms in Carnot groups},
   journal={Commun. Contemp. Math.},
   volume={13},
   date={2011},
   number={5},
   pages={765--794},
}

\bib{BL76}{book}{
   author={Bergh, J\"{o}ran},
   author={L\"{o}fstr\"{o}m, J\"{o}rgen},
   title={Interpolation spaces. An introduction},
   publisher={Springer-Verlag, Berlin-New York},
   date={1976},
}

\bib{BBM01}{article}{
   author={Bourgain, Jean},
   author={Brezis, Ha\"{\i}m},
   author={Mironescu, Petru},
   title={Another look at Sobolev spaces},
   conference={
      title={Optimal control and partial differential equations},
   },
   book={
      publisher={IOS, Amsterdam},
   },
   date={2001},
   pages={439--455},
}

\bib{BBM02-bis}{article}{
   author={Bourgain, Jean},
   author={Brezis, Ha\"{\i}m},
   author={Mironescu, Petru},
   title={Limiting embedding theorems for $W^{s,p}$ when $s\uparrow1$ and applications},
   note={Dedicated to the memory of Thomas H. Wolff},
   journal={J. Anal. Math.},
   volume={87},
   date={2002},
   pages={77--101},
}

\bib{BN06}{article}{
   author={Bourgain, Jean},
   author={Nguyen, Hoai-Minh},
   title={A new characterization of Sobolev spaces},
   journal={C. R. Math. Acad. Sci. Paris},
   volume={343},
   date={2006},
   number={2},
   pages={75--80},
}

\bib{B02}{book}{
   author={Braides, Andrea},
   title={$\Gamma$-convergence for beginners},
   series={Oxford Lecture Series in Mathematics and its Applications},
   volume={22},
   publisher={Oxford University Press, Oxford},
   date={2002},
}

\bib{Brezis02}{article}{
   author={Brezis, Ha\"{\i}m},
   title={How to recognize constant functions. A connection with Sobolev spaces},
   language={Russian, with Russian summary},
   journal={Uspekhi Mat. Nauk},
   volume={57},
   date={2002},
   number={4(346)},
   pages={59--74},
   translation={
      journal={Russian Math. Surveys},
      volume={57},
      date={2002},
      number={4},
      pages={693--708},
      issn={0036-0279},
   },
}

\bib{Brezis11}{book}{
   author={Brezis, Ha\"{\i}m},
   title={Functional analysis, Sobolev spaces and Partial Differential Equations},
   series={Universitext},
   publisher={Springer, New York},
   date={2011},
   pages={xiv+599},
}

\bib{B15}{article}{
   author={Brezis, Ha\"{\i}m},
   title={New approximations of the total variation and filters in imaging},
   journal={Atti Accad. Naz. Lincei Rend. Lincei Mat. Appl.},
   volume={26},
   date={2015},
   number={2},
   pages={223--240},
}

\bib{BN16}{article}{
   author={Brezis, Ha\"{\i}m},
   author={Nguyen, Hoai-Minh},
   title={The BBM formula revisited},
   journal={Atti Accad. Naz. Lincei Rend. Lincei Mat. Appl.},
   volume={27},
   date={2016},
   number={4},
   pages={515--533},
}

\bib{BN16-bis}{article}{
   author={Brezis, Ha\"{\i}m},
   author={Nguyen, Hoai-Minh},
   title={Two subtle convex nonlocal approximations of the BV-norm},
   journal={Nonlinear Anal.},
   volume={137},
   date={2016},
   pages={222--245},
}

\bib{BN18}{article}{
   author={Brezis, Ha\"{\i}m},
   author={Nguyen, Hoai-Minh},
   title={Non-local functionals related to the total variation and connections with image processing},
   journal={Ann. PDE},
   volume={4},
   date={2018},
   number={1},
   pages={Paper No. 9, 77},
}

\bib{BN20}{article}{
   author={Brezis, Ha\"{\i}m},
   author={Nguyen, Hoai-Minh},
   title={Non-local, non-convex functionals converging to Sobolev norms},
   journal={Nonlinear Anal.},
   volume={191},
   date={2020},
   pages={111626, 9},
}

\bib{BVY20}{article}{
   author={Brezis, Ha\"{\i}m},
   author={Van Schaftingen, Jean},
   author={Yung, Po-Lam},
   title={A surprising formula for Sobolev norms},
   journal={Proc. Natl. Acad. Sci. USA},
   volume={118},
   date={2021},
   number={8},
   pages={Paper No. e2025254118, 6},
}


\bib{BCCS}{article}{
   author={Brué, Elia},
   author={Calzi, Mattia},
   author={Comi, Giovanni Eugenio},
   author={Stefani, Giorgio},
   title={A distributional approach to fractional Sobolev spaces and fractional variation: asymptotics II},
   journal={C. R. Math.},
   status={to appear},
   eprint={https://arxiv.org/abs/2011.03928},
}

\bib{CS18}{article}{
   author={Comi, Giovanni E.},
   author={Stefani, Giorgio},
   title={A distributional approach to fractional Sobolev spaces and
   fractional variation: existence of blow-up},
   journal={J. Funct. Anal.},
   volume={277},
   date={2019},
   number={10},
   pages={3373--3435},
}

\bib{CF17}{article}{
   author={Cozzi, Matteo},
   author={Figalli, Alessio},
   title={Regularity theory for local and nonlocal minimal surfaces: an
   overview},
   conference={
      title={Nonlocal and nonlinear diffusions and interactions: new methods
      and directions},
   },
   book={
      series={Lecture Notes in Math.},
      volume={2186},
      publisher={Springer, Cham},
   },
   date={2017},
   pages={117--158},
}


\bib{D93}{book}{
   author={Dal Maso, Gianni},
   title={An introduction to $\Gamma$-convergence},
   series={Progress in Nonlinear Differential Equations and their
   Applications},
   volume={8},
   publisher={Birkh\"{a}user Boston, Inc., Boston, MA},
   date={1993},
}

\bib{D02}{article}{
   author={D\'{a}vila, J.},
   title={On an open question about functions of bounded variation},
   journal={Calc. Var. Partial Differential Equations},
   volume={15},
   date={2002},
   number={4},
   pages={519--527},
}

\bib{DiMS19}{article}{
   author={Di Marino, Simone},
   author={Squassina, Marco},
   title={New characterizations of Sobolev metric spaces},
   journal={J. Funct. Anal.},
   volume={276},
   date={2019},
   number={6},
   pages={1853--1874},
}

\bib{DiNPV12}{article}{
   author={Di Nezza, Eleonora},
   author={Palatucci, Giampiero},
   author={Valdinoci, Enrico},
   title={Hitchhiker's guide to the fractional Sobolev spaces},
   journal={Bull. Sci. Math.},
   volume={136},
   date={2012},
   number={5},
   pages={521--573},
}

\bib{DM20}{article}{
   author={Dominguez, Oscar},
   author={Milman, Mario},
   title={New Brezis-Van Schaftingen-Yung Sobolev type inequalities connected with maximal inequalities and one parameter families of operators},
   date={2020},
   eprint={https://arxiv.org/abs/2010.15873},
   status={preprint}
}

\bib{EG15}{book}{
   author={Evans, Lawrence C.},
   author={Gariepy, Ronald F.},
   title={Measure theory and fine properties of functions},
   series={Textbooks in Mathematics},
   edition={Revised edition},
   publisher={CRC Press, Boca Raton, FL},
   date={2015},
}

\bib{F69}{book}{
   author={Federer, Herbert},
   title={Geometric measure theory},
   series={Die Grundlehren der mathematischen Wissenschaften, Band 153},
   publisher={Springer-Verlag New York Inc., New York},
   date={1969},
}

\bib{FS19}{article}{
   author={Fern\'{a}ndez Bonder, Juli\'{a}n},
   author={Salort, Ariel M.},
   title={Fractional order Orlicz-Sobolev spaces},
   journal={J. Funct. Anal.},
   volume={277},
   date={2019},
   number={2},
   pages={333--367},
}

\bib{FHR20}{article}{
   author={Ferreira, Rita},
   author={H\"{a}st\"{o}, Peter},
   author={Ribeiro, Ana Margarida},
   title={Characterization of generalized Orlicz spaces},
   journal={Commun. Contemp. Math.},
   volume={22},
   date={2020},
   number={2},
   pages={1850079, 25},
}

\bib{FS08}{article}{
   author={Frank, Rupert L.},
   author={Seiringer, Robert},
   title={Non-linear ground state representations and sharp Hardy inequalities},
   journal={J. Funct. Anal.},
   volume={255},
   date={2008},
   number={12},
   pages={3407--3430},
}

\bib{G14-C}{book}{
   author={Grafakos, Loukas},
   title={Classical Fourier analysis},
   series={Graduate Texts in Mathematics},
   volume={249},
   edition={3},
   publisher={Springer, New York},
   date={2014},
}

\bib{G14}{book}{
   author={Grafakos, Loukas},
   title={Modern Fourier analysis},
   series={Graduate Texts in Mathematics},
   volume={250},
   edition={3},
   publisher={Springer, New York},
   date={2014},
}

\bib{G60}{article}{
   author={Gustin, William},
   title={Boxing inequalities},
   journal={J. Math. Mech.},
   volume={9},
   date={1960},
   pages={229--239},
}

\bib{H59}{article}{
   author={Horv\'ath, J.},
   title={On some composition formulas},
   journal={Proc. Amer. Math. Soc.},
   volume={10},
   date={1959},
   pages={433--437},
}

\bib{KL05}{article}{
   author={Kolyada, V. I.},
   author={Lerner, A. K.},
   title={On limiting embeddings of Besov spaces},
   journal={Studia Math.},
   volume={171},
   date={2005},
   number={1},
   pages={1--13},
}

\bib{KM19}{article}{
   author={Kreuml, Andreas},
   author={Mordhorst, Olaf},
   title={Fractional Sobolev norms and BV functions on manifolds},
   journal={Nonlinear Anal.},
   volume={187},
   date={2019},
   pages={450--466},
}

\bib{LMP19}{article}{
   author={Lam, Nguyen},
   author={Maalaoui, Ali},
   author={Pinamonti, Andrea},
   title={Characterizations of anisotropic high order Sobolev spaces},
   journal={Asymptot. Anal.},
   volume={113},
   date={2019},
   number={4},
   pages={239--260},
}

\bib{L09}{book}{
   author={Leoni, Giovanni},
   title={A first course in Sobolev spaces},
   series={Graduate Studies in Mathematics},
   volume={105},
   publisher={American Mathematical Society, Providence, RI},
   date={2009},
}

\bib{LS11}{article}{
   author={Leoni, Giovanni},
   author={Spector, Daniel},
   title={Characterization of Sobolev and $BV$ spaces},
   journal={J. Funct. Anal.},
   volume={261},
   date={2011},
   number={10},
   pages={2926--2958},
}

\bib{LS14}{article}{
   author={Leoni, Giovanni},
   author={Spector, Daniel},
   title={Corrigendum to ``Characterization of Sobolev and $BV$ spaces'' [J. Funct. Anal. 261 (10) (2011) 2926--2958]},
   journal={J. Funct. Anal.},
   volume={266},
   date={2014},
   number={2},
   pages={1106--1114},
}

\bib{MP19}{article}{
   author={Maalaoui, Ali},
   author={Pinamonti, Andrea},
   title={Interpolations and fractional Sobolev spaces in Carnot groups},
   journal={Nonlinear Anal.},
   volume={179},
   date={2019},
   pages={91--104},
}

\bib{M12}{book}{
   author={Maggi, Francesco},
   title={Sets of finite perimeter and geometric variational problems},
   series={Cambridge Studies in Advanced Mathematics},
   volume={135},
   publisher={Cambridge University Press, Cambridge},
   date={2012},
}

\bib{MS02}{article}{
   author={Maz\cprime ya, V.},
   author={Shaposhnikova, T.},
   title={On the Bourgain, Brezis, and Mironescu theorem concerning limiting
   embeddings of fractional Sobolev spaces},
   journal={J. Funct. Anal.},
   volume={195},
   date={2002},
   number={2},
   pages={230--238},
}

\bib{MS03}{article}{
   author={Maz\cprime ya, V.},
   author={Shaposhnikova, T.},
   title={Erratum to: ``On the Bourgain, Brezis and Mironescu theorem
   concerning limiting embeddings of fractional Sobolev spaces'' [J. Funct.
   Anal. {\bf 195} (2002), no. 2, 230--238; MR1940355 (2003j:46051)]},
   journal={J. Funct. Anal.},
   volume={201},
   date={2003},
   number={1},
   pages={298--300},
}

\bib{MS15}{article}{
   author={Mengesha, Tadele},
   author={Spector, Daniel},
   title={Localization of nonlocal gradients in various topologies},
   journal={Calc. Var. Partial Differential Equations},
   volume={52},
   date={2015},
   number={1-2},
   pages={253--279},
}

\bib{M05}{article}{
   author={Milman, Mario},
   title={Notes on limits of Sobolev spaces and the continuity of interpolation scales},
   journal={Trans. Amer. Math. Soc.},
   volume={357},
   date={2005},
   number={9},
   pages={3425--3442},
}

\bib{N07}{article}{
   author={Nguyen, Hoai-Minh},
   title={$\Gamma$-convergence and Sobolev norms},
   journal={C. R. Math. Acad. Sci. Paris},
   volume={345},
   date={2007},
   number={12},
   pages={679--684},
}

\bib{N08}{article}{
   author={Nguyen, Hoai-Minh},
   title={Further characterizations of Sobolev spaces},
   journal={J. Eur. Math. Soc. (JEMS)},
   volume={10},
   date={2008},
   number={1},
   pages={191--229},
}

\bib{N11}{article}{
   author={Nguyen, Hoai-Minh},
   title={$\Gamma$-convergence, Sobolev norms, and BV functions},
   journal={Duke Math. J.},
   volume={157},
   date={2011},
   number={3},
   pages={495--533},
}

\bib{NS19}{article}{
   author={Nguyen, Hoai-Minh},
   author={Squassina, Marco},
   title={On anisotropic Sobolev spaces},
   journal={Commun. Contemp. Math.},
   volume={21},
   date={2019},
   number={1},
   pages={1850017, 13},
}

\bib{PSV17}{article}{
   author={Pinamonti, Andrea},
   author={Squassina, Marco},
   author={Vecchi, Eugenio},
   title={The Maz\cprime ya-Shaposhnikova limit in the magnetic setting},
   journal={J. Math. Anal. Appl.},
   volume={449},
   date={2017},
   number={2},
   pages={1152--1159},
}

\bib{PSV19}{article}{
   author={Pinamonti, Andrea},
   author={Squassina, Marco},
   author={Vecchi, Eugenio},
   title={Magnetic BV-functions and the Bourgain-Brezis-Mironescu formula},
   journal={Adv. Calc. Var.},
   volume={12},
   date={2019},
   number={3},
   pages={225--252},
}

\bib{P04-1}{article}{
   author={Ponce, Augusto C.},
   title={An estimate in the spirit of Poincar\'{e}'s inequality},
   journal={J. Eur. Math. Soc. (JEMS)},
   volume={6},
   date={2004},
   number={1},
   pages={1--15},
}

\bib{P04-2}{article}{
   author={Ponce, Augusto C.},
   title={A new approach to Sobolev spaces and connections to $\Gamma$-convergence},
   journal={Calc. Var. Partial Differential Equations},
   volume={19},
   date={2004},
   number={3},
   pages={229--255},
}

\bib{P16}{book}{
   author={Ponce, Augusto C.},
   title={Elliptic PDEs, measures and capacities},
   series={EMS Tracts in Mathematics},
   volume={23},
   publisher={European Mathematical Society (EMS), Z\"{u}rich},
   date={2016},
}

\bib{PS17-box}{article}{
   author={Ponce, Augusto C.},
   author={Spector, Daniel},
   title={A boxing inequality for the fractional perimeter},
   journal={Ann. Sc. Norm. Super. Pisa Cl. Sci. (5)},
   volume={20},
   date={2020},
   number={1},
   pages={107--141},
}

\bib{PS17}{article}{
   author={Ponce, Augusto C.},
   author={Spector, Daniel},
   title={A note on the fractional perimeter and interpolation},
   journal={C. R. Math. Acad. Sci. Paris},
   volume={355},
   date={2017},
   number={9},
   pages={960--965},
}

\bib{SSVanS17}{article}{
   author={Schikorra, Armin},
   author={Spector, Daniel},
   author={Van Schaftingen, Jean},
   title={An $L^1$-type estimate for Riesz potentials},
   journal={Rev. Mat. Iberoam.},
   volume={33},
   date={2017},
   number={1},
   pages={291--303},
}

\bib{SSS15}{article}{
   author={Schikorra, Armin},
   author={Shieh, Tien-Tsan},
   author={Spector, Daniel},
   title={$L^p$ theory for fractional gradient PDE with $VMO$ coefficients},
   journal={Atti Accad. Naz. Lincei Rend. Lincei Mat. Appl.},
   volume={26},
   date={2015},
   number={4},
   pages={433--443},
}
		
\bib{SSS18}{article}{
   author={Schikorra, Armin},
   author={Shieh, Tien-Tsan},
   author={Spector, Daniel E.},
   title={Regularity for a fractional $p$-Laplace equation},
   journal={Commun. Contemp. Math.},
   volume={20},
   date={2018},
   number={1},
   pages={1750003, 6},
}

\bib{SS15}{article}{
   author={Shieh, Tien-Tsan},
   author={Spector, Daniel E.},
   title={On a new class of fractional partial differential equations},
   journal={Adv. Calc. Var.},
   volume={8},
   date={2015},
   number={4},
   pages={321--336},
}

\bib{SS18}{article}{
   author={Shieh, Tien-Tsan},
   author={Spector, Daniel E.},
   title={On a new class of fractional partial differential equations II},
   journal={Adv. Calc. Var.},
   volume={11},
   date={2018},
   number={3},
   pages={289--307},
}

\bib{S16}{article}{
	author={\v{S}ilhav\'y, Miroslav},
	title={Beyond fractional laplacean: fractional gradient and divergence},
	date={2016-01-19},
	eprint={https://doi.org/10.13140/RG.2.1.2554.0885},
	status={slides of the talk at the Department of Mathematics Roma Tor Vergata},
}
   
\bib{S19}{article}{
   author={\v{S}ilhav\'y, Miroslav},
   title={Fractional vector analysis based on invariance requirements (Critique of coordinate approaches)},
   date={2019},
   journal={M. Continuum Mech. Thermodyn.},
   pages={1--22},
}

\bib{Spector18}{article}{
   author={Spector, Daniel},
   title={An optimal Sobolev embedding for $L^1$},
   journal={J. Funct. Anal.},
   volume={279},
   date={2020},
   number={3},
   pages={108559, 26},
}

\bib{Spector19}{article}{
   author={Spector, Daniel},
   title={A noninequality for the fractional gradient},
   journal={Port. Math.},
   volume={76},
   date={2019},
   number={2},
   pages={153--168},
}

\bib{SV16}{article}{
   author={Squassina, Marco},
   author={Volzone, Bruno},
   title={Bourgain-Br\'{e}zis-Mironescu formula for magnetic operators},
   journal={C. R. Math. Acad. Sci. Paris},
   volume={354},
   date={2016},
   number={8},
   pages={825--831},
}

\bib{S70}{book}{
   author={Stein, Elias M.},
   title={Singular integrals and differentiability properties of functions},
   series={Princeton Mathematical Series, No. 30},
   publisher={Princeton University Press, Princeton, N.J.},
   date={1970},
}

\bib{S93}{book}{
   author={Stein, Elias M.},
   title={Harmonic analysis: real-variable methods, orthogonality, and oscillatory integrals},
   series={Princeton Mathematical Series},
   volume={43},
   publisher={Princeton University Press, Princeton, NJ},
   date={1993},
}


\bib{T07}{book}{
   author={Tartar, Luc},
   title={An introduction to Sobolev spaces and interpolation spaces},
   series={Lecture Notes of the Unione Matematica Italiana},
   volume={3},
   publisher={Springer, Berlin; UMI, Bologna},
   date={2007},
}

\bib{T11}{article}{
   author={Triebel, Hans},
   title={Limits of Besov norms},
   journal={Arch. Math. (Basel)},
   volume={96},
   date={2011},
   number={2},
   pages={169--175},
}

\end{biblist}
\end{bibdiv}
\end{document}